\documentclass[11pt]{article}

\usepackage[compact]{titlesec}
\usepackage{array}
\usepackage{amsfonts}
\usepackage{amscd}
\usepackage{amssymb}
\usepackage{amsthm}
\usepackage{amsmath}
\usepackage{stmaryrd}
\usepackage{graphicx}
\usepackage{color}
\usepackage{verbatim}
\usepackage{mathrsfs}
\usepackage{pstricks,tikz}
\usetikzlibrary{calc}
\usepackage{caption}
\usepackage[neveradjust]{paralist}
\usepackage[a4paper, total={15cm, 25cm},centering]{geometry}
\usepackage{float}
\usepackage[absolute]{textpos}
\usepackage{subcaption}
\usepackage{ytableau}
\usepackage{hyperref}
\usepackage{bbm}
\usepackage{lineno}

\hypersetup{hidelinks}

\usetikzlibrary{automata, positioning, arrows}

\let\OLDthebibliography\thebibliography
\renewcommand\thebibliography[1]{
  \OLDthebibliography{#1}
  \setlength{\parskip}{0pt}
  \setlength{\itemsep}{0pt plus 0.3ex}
}

\def\peq{\preccurlyeq}

\definecolor{green1}{RGB}{153,216,201}
\definecolor{green2}{RGB}{44,190,95}

\definecolor{blue1}{RGB}{158,202,225}
\definecolor{blue2}{RGB}{49,130,189}

\definecolor{RedOrange}{cmyk}{0,0.77,0.87,0} 
\definecolor{Mahogany}{cmyk}{0,0.85,0.87,0.35} 
\definecolor{Maroon}{cmyk}{0,0.87,0.68,0.32} 
\definecolor{BrickRed}{cmyk}{0,0.89,0.94,0.28} 
\definecolor{Red}{cmyk}{0,1.,1.,0} 
\definecolor{OrangeRed}{cmyk}{0,1.,0.50,0} 

\definecolor{purple}{rgb}{0.8,0.12,0.8}
\definecolor{orange}{rgb}{1.0,0.7,0.0}
\definecolor{pink}{rgb}{1,0.5,0.8}
\definecolor{blackg}{rgb}{0.1,0.25,0.1}
\definecolor{ForestGreen}{cmyk}{0.91,0,0.88,0.42}
\definecolor{Turquoise}{cmyk}{0.85,0,0.20,0}

 \theoremstyle{plain}

\newtheorem{thm}{Theorem}[section]
\newtheorem{lemma}[thm]{Lemma}
\newtheorem{prop}[thm]{Proposition}
\newtheorem{cor}[thm]{Corollary}

\theoremstyle{definition}

\newtheorem{defn}[thm]{Definition}

\newtheorem{remark}[thm]{Remark}
\newtheorem{example}[thm]{Example}

\numberwithin{equation}{section}

\def\Wext{\widetilde{W}}
\def\Wfin{W_0}
\def\CwJp{C_{\sw_{n}}}
\def\Hext{H_{\mathsf{ext}}}

\def\Waff{W_{\mathsf{aff}}}
\def\Saff{S_{\mathsf{aff}}}
\def\Wext{W_{\mathsf{ext}}}

\def\veps{\varepsilon}

\def\sA{\mathsf{A}}
\def\sB{\mathsf{B}}
\def\sC{\mathsf{C}}
\def\sD{\mathsf{D}}

\def\sF{\mathsf{F}}
\def\sG{\mathsf{G}}

\def\sT{\mathsf{T}}

\def\sb{\mathsf{b}}

\def\cA{\mathcal{A}}

\def\cD{\mathcal{D}}

\def\cJ{\mathcal{J}}

\def\cL{\mathcal{L}}

\def \cP{\mathcal{P}}

\def\cV{\mathcal{V}}

\newcommand{\la}{\lambda}

\newcommand{\Ga}{\Gamma}

\def\CC{\mathbb{C}}

\def\NN{\mathbb{N}}

\def\QQ{\mathbb{Q}}
\def\RR{\mathbb{R}}

\def\TT{\mathbb{T}}

\def\ZZ{\mathbb{Z}}

\DeclareMathOperator\Tr{\mathsf{Tr}}

\def\la{\lambda}

\def\<{\langle}
\def\>{\rangle}

\def\la{\lambda}

\def\sR{\mathsf{R}}
\def\fs{\mathfrak{s}}

\newcommand{\fc}{\mathfrak{c}}

\newcommand{\sw}{\mathsf{w}}
\newcommand{\sd}{\mathsf{d}}

\newcommand{\sm}{\mathsf{m}}
\newcommand{\ssp}{\mathsf{p}}

\newcommand{\su}{\mathsf{u}}

\newcommand{\sv}{\mathsf{v}}

\newcommand{\fu}{\mathfrak{u}}

\newcommand{\st}{\mathsf{t}}

\newcommand{\ba}{\mathbf{a}}
\newcommand{\bb}{\mathbf{b}}

\newcommand{\sq}{\mathsf{q}}

\makeatletter
\renewcommand{\@makefnmark}{\mbox{\textsuperscript{}}}
\makeatother

\title{Lusztig's asymptotic algebra and a cell factorisation for the second lowest two-sided cell in type $\tilde{\sB}_n$}
\author{J. Parkinson}
\date{\today}

\begin{document}

\maketitle

\vspace{-0.5cm}

\begin{abstract}
We determine Lusztig's asymptotic algebra for the second lowest two-sided cell~$\Delta_1$ in the equal parameter extended affine Weyl group $\tilde{\mathsf{B}}_n$ for all $n\geq 2$, realising it as a full matrix algebra over the representation ring of $\mathsf{Sp}_{2n-2}(\CC)\times C_2$ with $C_2$ the group of order~$2$. Our result is an explicit isomorphism of based rings, determining the image of each basis element $\mathsf{t}_w$ of the asymptotic algebra. As a consequence we verify that Lusztig's original conjecture from 1989 on the structure of the asymptotic algebra holds for this cell. A key ingredient in our proof is a refinement of Shi's description of~$\Delta_1$, giving an explicit parametrisation of the elements of this cell (a \textit{cell factorisation}). This cell factorisation simultaneously describes all intersections of left and right cells, with uniqueness and length-additivity of expressions.  Applications include a classification of the involutions in $\Delta_1$, and a calculation of the length generating functions for $\Delta_1$ and the right cells in~$\Delta_1$. 
\end{abstract}
%


\section{Introduction}

Two-sided Kazhdan-Lusztig cells in equal parameter affine Weyl groups are classified by unipotent conjugacy classes in the corresponding complex dual group (\cite[Theorem 4.8]{Lus:IV}). This classification does not, however, provide explicit parametrisations of the elements of a cell or determine its asymptotic algebra. The latter records the leading coefficients of Kazhdan-Lusztig multiplication and retains the interactions between the left and right cells. Explicit descriptions of these structures remain difficult outside affine type~$\sA$, the lowest two-sided cell, and low-rank cases (we will give a more precise survey of known results below). 

In \cite[Conjecture~10.5]{Lus:IV} Lusztig gave a conjectural description of the asymptotic algebra associated to an equal parameter affine Hecke algebra. It is known that this conjectural description does not always hold (see \cite{BDD:23}), and a weaker form of the conjecture has been proved by Bezrukavnikov and Ostrik~\cite{BO:04}. 

In this paper we study the second lowest two-sided cell $\Delta_1$ in the equal parameter extended affine Weyl group of type $\tilde{\sB}_n$, uniformly for all ranks $n\geq 2$. We refine Shi's recent description~\cite{ShiBn} of $\Delta_1$ into a unique length-additive factorisation
$$
\Delta_1=\{u^{-1}\sm_{\la} v\mid u,v\in\sB,\,\lambda\in P_n^+\}
$$
(with notation to be introduced below) in which $u$ and $v$ identify the right and left cells. We then show that this parametrisation realises the asymptotic algebra for this cell as a full matrix algebra over the representation ring of $\mathsf{Sp}_{2n-2}(\CC)\times C_2$ (with $C_2$ the group of order $2$) with $\st_{u^{-1}\sm_{\la}v}$ corresponding to $E_{u,v}\otimes [V_{\lambda}]$. Thus the leading Kazhdan-Lusztig structure constants within $\Delta_1$ are explicitly identified with tensor product multiplicities. This establishes Lusztig's original conjectural description of the asymptotic algebra for this cell, and gives a uniform model beyond the lowest two-sided cell in affine type~$\sB$.

%
%
%

We now introduce some notation and give an outline of Shi's description of $\Delta_1$ from~\cite{ShiBn}. We use the standard Bourbaki conventions~\cite{Bou:02} for root systems~$\Phi$ of type $\sB_n$. Let $\Wfin$, $\Waff=Q\rtimes \Wfin$, and $\Wext=P\rtimes \Wfin$ be the spherical, affine, and extended affine Weyl groups, where $Q$ and $P$ are the coroot and coweight lattices of $\Phi$, respectively. For $i\in\{0,1,\ldots,n\}$ let $W_i$ be the parabolic subgroup generated by $\Saff\backslash\{s_i\}$, and let $\sw_i$ be the longest element of $W_i$. For example, $W_0$ and $W_1$ are of type $\sB_n$, and $W_n$ is of type~$\sD_n$. Let $\Delta_0$ be the lowest two-sided cell, and let 
\begin{align*}
\sB=\{v\in \Waff\mid \ell(\sw_nv)=\ell(\sw_n)+\ell(v),\,\sw_n v\notin\Delta_0,\,\ba(s_i\sw_nv)<\ell(\sw_n)\text{ for $0\leq i\leq n-1$}\},
\end{align*}
where $\ba$ is Lusztig's $\ba$-function. An element $w=x\sw_n y$ with $x,y\in\Wext$ is called an \textit{extension} of $\sw_n$ if $\ell(w)=\ell(x)+\ell(\sw_n)+\ell(y)$. In \cite{ShiBn} Shi gives the following description of the two-sided cell $\Delta_1$ and the one-sided cells contained in $\Delta_1$. 

\begin{thm}[{\cite[Theorem~4.4, Proposition~5.9, and Corollary~5.10]{ShiBn}}]\label{thm:Shi}
The two-sided cell $\Delta_1$ consists of all extensions of $\sw_{n}$ in $\Wext\backslash\Delta_0$, and the right cells of $\Delta_1$ are the sets
\begin{align*}
\Ga_u=\{u^{-1}\sw_ny\mid y\in\Wext,\,\ell(u^{-1}\sw_ny)=\ell(u)+\ell(\sw_n)+\ell(y),\,u^{-1}\sw_ny\notin\Delta_0\}
\end{align*}
with $u\in \sB$, and if $u,v\in\sB$ with $u\neq v$ then $\Ga_u\neq\Ga_v$. 
\end{thm}

The description in Theorem~\ref{thm:Shi} is very elegant and parallels the description of $\Delta_0$ as all extensions in $\Wext$ of $\sw_0$ from~\cite{Shi:Lowest}. However converting Theorem~\ref{thm:Shi} into a statement explicitly listing the elements of $\Delta_1$, and each right cell~$\Ga_u$, appears to be a nontrivial task. For example, to list the elements of $\Ga_u$ one must explicitly determine all elements $y\in\Wext$ such that $\ell(u^{-1}\sw_ny)=\ell(u)+\ell(\sw_n)+\ell(y)$ with the property that $w=u^{-1}\sw_ny$ cannot be written in the form $w=a\sw_0b$ with $a,b\in\Wext$  with $\ell(w)=\ell(a)+\ell(\sw_0)+\ell(b)$. Moreover, it is desirable to have a more explicit description of the set $\sB$, and to be able to list the elements of $\Ga_u\cap\Ga_v^{-1}$ for each $u,v\in\sB$ (these intersections of right and left cells play an important role in Kazhdan-Lusztig theory, see in particular~\cite{LusztigXi}). 


We now describe our results. Let $\cA_n$ be the closed region of $\mathbb{R}^n$ bounded by the hyperplanes $x_n=0$ and $x_n=1$ (this is the \textit{fundamental $n$-alcove} in the language of~\cite{GLP:23,Lus:97}). Let $P_n=P\cap\cA_n$ be the set of coweights contained in $\cA_n$. There is a natural action of a Weyl group $G_n$ of type $\sC_{n-1}$ (equivalently $\sB_{n-1}$) on $P_n$, and we write $P_n^+$ for a certain choice of orbit representatives for this action. For each $\lambda\in P_n$ there is naturally associated an explicit \textit{pseudo-translation} $\tau_{\la}\in\Wext$, and the set $\mathsf{T}_n=\{\tau_{\la}\mid \la\in P_n\}$ is an abelian group acting on $\cA_n$. For each $\la\in P_n$ let $\sm_{\la}$ denote the unique maximal length element of the double coset $W_n\tau_{\la}W_n$. It turns out that $\sm_{g\la}=\sm_{\la}$ for $g\in G_n$, and hence the elements $\sm_{\la}$ are indexed by~$P_n^+$. 

Writing $\Ga=\Ga_1$ (with $1\in\Wext$ the identity), we give the following explicit parametrisation of the elements of $\Delta_1$, each right cell $\Ga_u$, and each intersection $\Ga_u\cap\Ga_v^{-1}$. 

\begin{thm}\label{thm:main1}
We have:
\begin{compactenum}[$(1)$]
\item $\Delta_1=\{u^{-1}\sm_{\lambda}v\mid u,v\in\sB,\,\lambda\in P_n^+\}$;
\item each $w\in \Delta_1$ has a unique representation as $w=u^{-1}\sm_{\la}v$ with $u,v\in\sB$ and $\la\in P_n^+$;
\item $\Ga_u=\{u^{-1}\sm_{\lambda}v\mid v\in \sB,\,\lambda\in P_n^+\}=u^{-1}\Ga$;
\item $\Ga_u\cap\Ga_v^{-1}=\{u^{-1}\sm_{\lambda}v\mid \lambda\in P_n^+\}=u^{-1}(\Ga\cap\Ga^{-1})v$;
\item $\ell(u^{-1}\sm_{\lambda}v)=\ell(u)+\ell(\sm_{\lambda})+\ell(v)$ for all $u,v\in \sB$ and $\lambda\in P_n^+$.
\end{compactenum}
\end{thm}

In particular, Theorem~\ref{thm:main1} shows that $\Delta_1$ displays a remarkable regularity in the sense that each right cell can be mapped bijectively to $\Ga=\Ga_1$ simply by left multiplication by an element of $\Waff$, and intersections of arbitrary left and right cells of $\Delta_1$ can be transported to the intersection $\Ga\cap\Ga^{-1}$ by left and right multiplication. Such highly regular behaviour is not a standard property of $KL$-cells (for example, such a phenomenon already fails for the second lowest two-sided cell in type $\tilde{\sA}_3$, where instead one must use more complicated \textit{$*$-operations} to transport structure between right cells; see \cite{ShiAn,Xibook}). 

To make Theorem~\ref{thm:main1} completely explicit one should concretely determine the set $\sB$, and compute the elements $\sm_{\lambda}$. Our description of $\sB$ in the theorem below is in terms of an explicit element $\sb\in\Waff$ given in~(\ref{eq:b}), and $\sb^{\sigma}=\sigma \sb\sigma^{-1}$ is the image of $\sb$ under the diagram automorphism $\sigma$ with $s_0\leftrightarrow s_1$. For $w\in\Waff$ let $[1,w]=\{v\in\Waff\mid \ell(v^{-1}w)=\ell(w)-\ell(v)\}$ (an interval in the right weak order). In other words, $[1,w]$ consists of all prefixes of $w$. 

The group $\sT_n$ acts on the set of (classical) alcoves contained in the fundamental $n$-alcove~$\cA_n$, however this action is not free. To rectify this, let $\mathbb{X}_n=\{w\in\Wext\mid wA_0\subseteq \cA_n\}$, where $A_0$ is the (classical) fundamental alcove. The group $\sT_n$ acts freely on this set. Let $\fu=\sw_{0,n}\sw_n\in W_n$ (where $\sw_{0,n}$ is the longest element of the $\sA_{n-1}$ subsystem generated by $s_1,\ldots,s_{n-1}$). Our description of the set $\sB$, and the elements $\sm_{\la}$, is given in the following theorem. 
\goodbreak

\begin{thm}\label{thm:main2}
For $\la\in P_n^+$ we have
\begin{compactenum}[$(1)$]
\item $\sm_{\la}=\sw_n\fu^{-1}\tau_{\la}\su$ with $\ell(\sm_{\la})=\ell(\sw_n)+\ell(\fu^{-1}\tau_{\la}\fu)$;
\item $\sB=[1,\sb]\cup[1,\sb^{\sigma}]$;
\item $\sF=\fu\sB$ is a fundamental domain for the action of $\sT_n$ on $\mathbb{X}_n$. 
\end{compactenum}
\end{thm}

Combining Theorems~\ref{thm:main1} and~\ref{thm:main2} gives a strong form of ``cell factorisation'' for $\Delta_1$ in the sense of J. Guilhot~\cite{Gui:10}, with each element $w\in \Delta_1$ written in a unique way as $w=u^{-1}\sw_n\tau v$ with $u,v\in\sB$ and $\tau\in T^+$ (a certain length-additive semigroup), with $\ell(w)=\ell(u)+\ell(\sw_n)+\ell(\tau)+\ell(v)$. See Corollary~\ref{cor:main3} for a precise statement. This cell factorisation leads to various corollaries. For example, in Corollary~\ref{cor:involutions} we completely classify the involutions in $\Delta_1$ (specifically, each $\Ga_u\cap\Ga_u^{-1}$ consists entirely of involutions, and each intersection $\Ga_u\cap\Ga_v^{-1}$ with $u\neq v$ contains no involution). Moreover we compute the rational length generating functions for the elements of $\Delta_1$, the right cells $\Ga_u$, and the set of involutions in~$\Delta_1$ in Corollary~\ref{cor:generating}. For example, we have
$$
\sum_{w\in\Delta_1}t^{\ell(w)}=\frac{2t^{n(n-1)}\sB(t)^2}{\prod_{i=1}^{n-1}(1-t^{i(2n-i)})}\quad\text{where}\quad \sB(t)=\sum_{u\in \sB}t^{\ell(u)}. 
$$

We now turn our attention to the component of the asymptotic algebra associated to this cell. Recall that the asymptotic algebra~$\cJ$ introduced by Lusztig~\cite{Lus:87} is a $\ZZ$-algebra with basis $(\st_w)_{w\in\Wext}$ and multiplication defined using certain leading coefficients of the structure constants of the Kazhdan-Lusztig basis in the Hecke algebra. If $\Delta$ is a two-sided cell then $\cJ_{\Delta}=\mathsf{span}_{\ZZ}\{\st_w\mid w\in\Delta\}$ is a two-sided ideal of $\cJ$, and $\cJ$ is a direct sum of $\cJ_{\Delta}$ as $\Delta$ runs through the set of two-sided cells of~$\Wext$.

A first step in determining $\cJ_{\Delta_1}$ is to convert the cell factorisation from Theorems~\ref{thm:main1} and~\ref{thm:main2} into a factorisation formula in $\cJ_{\Delta_1}$.

\begin{thm}\label{thm:main4}
For $u,v\in \sB$ and $\la\in P_n^+$ we have 
$$
\st_{u^{-1}\sw_n}\st_{\sm_{\la}}\st_{\sw_nv}=\st_{u^{-1}\sm_{\la}v}. 
$$
\end{thm}

After proving Theorem~\ref{thm:main4}, our general strategy follows the ideas initiated in~\cite{GP:19,GP:19b} and continued in \cite{CGLP:24,GLP:23}, however various modifications and refinements of the techniques are required in the process. In particular, since there is not a complete description of all cells in affine type $\sB$ arguments typically must be ``localised'' to the cell~$\Delta_1$, rather than relying on global information. Central to this approach is a detailed analysis of a certain representation $\pi$ of the Hecke algebra on a finite free module over a multivariate Laurent polynomial ring with a $C_2$ group-ring factor $\sR[z]$ (with $\sR=\ZZ[\sq,\sq^{-1}]$). We show that this representation shares many desirable properties with the (infinite dimensional) Kazhdan-Lusztig cell module associated to~$\Delta_1$. Specifically, $\pi$ annihilates all Kazhdan-Lusztig basis elements $C_w$ with $w\in\Delta_0$, and with respect to a natural basis the degree in $\sq$ of the entries of the matrix $\pi(T_w)$ (and also $\pi(C_w)$) is bounded by the value of Lusztig's $\ba$-function on $\Delta_1$. 

This allows us to define \textit{leading matrices} $\fc(w)=\mathsf{sp}_{\sq^{-1}=0}(\sq^{-\ba(\Delta_1)}\pi(T_w))$ with entries in~$\ZZ[z]$. Ultimately we are able to explicitly determine the leading matrices $\fc(w)$ for $w\in\Delta_1$ in terms of Schur functions $\fs_{\la}(z)\in\ZZ[z]^{G_n}$ associated to the group $\mathsf{Sp}_{2n-2}(\CC)\times C_2$, with $C_2$ the group of order~$2$. Two key ingredients in this important calculation are as follows. Firstly we use combinatorial interpretations for the matrix entries of $\pi$ in terms of \textit{$n$-folded alcove paths} in the fundamental $n$-alcove~$\cA_n$ (see \cite{GLP:23}) to prove certain entries attain an appropriate degree in~$\sq$. Secondly we make analytical estimates using Opdam's Plancherel formula~\cite{Opd:04}, suitably localised to the cell~$\Delta_1$. Positivity in the Plancherel formula allows us to bound the norm of a leading coefficient by one, and integrality and Schur orthogonality with respect to an asymptotic Plancherel measure then force it to be a single Schur function up to sign. This sign is determined using the combinatorial description in terms of alcove paths. 

\begin{thm}\label{thm:main5}
For $u,v\in\sB$ and $\la\in P_n^+$ we have 
$$
\fc(u^{-1}\sm_{\la}v)=\fs_{\la}(z)E_{u,v}. 
$$
\end{thm}

%
%
%
%

Using Theorem~\ref{thm:main5} we deduce the main theorem of the paper, realising $\cJ_{\Delta_1}$ as a full matrix algebra over a representation ring, paralleling the description in type $\tilde{\sA}_n$ from~\cite{Xibook}.

\begin{thm}\label{thm:main6}
The map 
$$
\cJ_{\Delta_1}\to \mathsf{Mat}_{N}(\mathsf{Rep}(\mathsf{Sp}_{2n-2}(\CC)\times C_2))\quad\text{with}\quad \st_{u^{-1}\sm_{\la}v}\mapsto E_{u,v}\otimes [V_{\lambda}],
$$
is an isomorphism of based rings, where $N=2^{n-1}n!$, and for $\la=(\la',\la_n)\in P_n^+$ we write $V_{\lambda}=V_{\la'}\otimes \epsilon^{\la_n}$ with $V_{\la'}$ the irreducible representation of $\mathsf{Sp}_{2n-2}(\CC)$ with highest weight~$\la'$, and  $\epsilon$ is the nontrivial character of the cyclic group~$C_2$. 
\end{thm}

We then prove that Lusztig's original conjectural description holds for~$\cJ_{\Delta_1}$. 

\begin{cor}\label{cor:main7} Lusztig's conjecture \cite[Conjecture 10.5]{Lus:IV} holds for the cell $\Delta_1$. 
\end{cor}

We present several further applications in Section~\ref{sec:applications2}, including multiplication formulae in $\cJ_{\Delta_1}$ in terms of tensor product multiplicities, a calculation of the centre of $\cJ_{\Delta_1}$, a characterisation of $\Delta_1$ (respectively the involutions in~$\Delta_1$) in terms of the representation $\pi$ (respectively the character of $\pi$), an asymptotic Plancherel Theorem for the cell~$\Delta_1$, and a calculation of the asymptotic algebra of the second lowest two-sided cell $\Delta_1'$ in the non-extended affine Weyl group of type $\tilde{\sB}_n$. Finally, in Section~\ref{sec:An} we discuss the $\tilde{\sA}_n$ case, and take the opportunity to give an elementary proof of a conjecture from~\cite{CGLP:24}, making the description of the leading matrices from \cite[Theorem~7.8]{CGLP:24} in the $\tilde{\sA}_n$ case unconditional. 

To put our results in perspective, we conclude this introduction with a brief survey of some known results on the determination of cells and the asymptotic algebra in affine types. Cells in type $\tilde{\sA}_n$ are understood by the fundamental work of Lusztig~\cite{Lus:IV} and Shi~\cite{ShiAn}. Various specific low rank cases in other affine types are known, including all irreducible cases of affine rank at most $4$ (see \cite{Bedard,DuB3,DuC3,DuD4,Lus:I,ShiD4,ShiF4,ShiC4,Zhang}). Shi has determined the structure of the lowest two-sided cell in~\cite{Shi:Lowest}. To the author's knowledge, the only affine types for which the asymptotic algebra is well understood are the rank~$2$ cases (see \cite{GP:19,GP:19b}, which also treat unequal parameters), the $\tilde{\sA}_n$ case (see \cite{Xibook} and the recent approaches in~\cite{CGLP:24,KP:23}), and $\tilde{\sB}_3$ with equal parameters (see the recent series \cite{B3:1,B3:2,B3:3,B3:4} and \cite{BDD:23}). The component of the asymptotic algebra associated to the lowest two-sided cell has been determined by Xi~\cite{JLow:1,JLow:2,JLow:3}.

We now outline the structure of the paper. Section~\ref{sec:background} gives background and preliminary material, and Sections~\ref{sec:Jalcove} and~\ref{sec:main} are devoted to the study of the cell $\Delta_1$, proving Theorems~\ref{thm:main1} and~\ref{thm:main2}. In Section~\ref{sec:Jfact} we prove the factorisation theorem (Theorem~\ref{thm:main4}), and in Section~\ref{sec:J} we study the representation $\pi$ and the leading matrices, proving Theorems~\ref{thm:main5} and~\ref{thm:main6} and Corollary~\ref{cor:main7}.

\section{Background and preliminary results}\label{sec:background}

In this section we give background and fix notation on affine Weyl groups of type $\tilde{\sB}_n$ and the affine root system. We also outline some properties of intervals and convex hulls in Coxeter groups, and we give a brief summary of the basic parts of Kazhdan-Lusztig theory that are required in the paper. 

\subsection{Root systems and Weyl groups}

Let $V=\mathbb{R}^n$ with standard inner product $\langle\cdot,\cdot\rangle$, and let $\Phi_0\subseteq V$ be a root system of type $\sB_n$, $n\geq 2$, with simple roots $\alpha_i=e_i-e_{i+1}$ for $1\leq i\leq n-1$ and $\alpha_n=e_n$ and associated positive roots $\Phi_0^+$. For $\alpha\in V\backslash\{0\}$ let $\alpha^{\vee}=2\alpha/\langle\alpha,\alpha\rangle$. Let $\varphi\in\Phi_0^+$ be the highest root (explicitly, $\varphi=e_1+e_2=\alpha_1+2(\alpha_2+\alpha_3+\cdots+\alpha_n)$). 

For $\alpha\in\Phi_0$ let $s_{\alpha}\in\mathsf{GL}(V)$ be the reflection 
$s_{\alpha}(x)=x-\langle x,\alpha\rangle\alpha^{\vee}$,  
and write $s_i=s_{\alpha_i}$ for $1\leq i\leq n$. Let $W_0$ be the subgroup of $\mathsf{GL}(V)$ generated by $\{s_{\alpha}\mid \alpha\in\Phi_0\}$. Thus $(W_0,S_0)$ is a Coxeter system of type $\sB_n$, where $S_0=\{s_i\mid 1\leq i\leq n\}$. 

The \textit{affine root system} is 
$
\Phi=\Phi_0+\ZZ\delta
$
(a subset of the vector space $V\oplus\RR\delta$). Let $\alpha_0=-\varphi+\delta$. The simple affine roots are $\alpha_0,\alpha_1,\ldots,\alpha_n$. Each root $\alpha+k\delta\in\Phi$ has a unique expression as $\alpha+k\delta=a_0\alpha_0+\cdots+a_n\alpha_n$ with $a_0,\ldots,a_n\in\ZZ$, and $\alpha+k\delta$ is \textit{positive} if and only if $a_0,\ldots,a_n\geq 0$. Explicitly, the set of positive affine roots is 
$$
\Phi^+=(\Phi_0^++\ZZ_{\geq 0}\delta)\cup(-\Phi_0^++\ZZ_{>0}\delta). 
$$
The \textit{support} of a root $\alpha=a_0\alpha_0+\cdots+a_n\alpha_n$ is 
$\mathsf{supp}(\alpha)=\{0\leq i\leq n\mid a_i\neq 0\}$. For $J\subseteq\{0,1,\ldots,n\}$ let 
$
\Phi_J=\{\alpha\in\Phi\mid \mathsf{supp}(\alpha)\subseteq J\}$ and $\Phi_J^+=\Phi_J\cap\Phi^+$. 
If $J=\{i_1,\ldots,i_k\}$ we write $\Phi_{i_1,\ldots,i_k}=\Phi_{\{0,1,\ldots,n\}\backslash J}$. In particular, $\Phi_n$ is a system of type~$\sD_n$.

For $\alpha+k\delta\in\Phi$, define an affine reflection $s_{\alpha+k\delta}\in\mathsf{Aff}(V)$ by 
$
s_{\alpha+k\delta}(x)=x-(\langle x,\alpha\rangle+k)\alpha^{\vee}.
$
Thus $s_{\alpha+k\delta}$ is the orthogonal reflection in the hyperplane 
$
H_{\alpha+k\delta}=\{x\in V\mid \langle x,\alpha\rangle+k=0\}.
$
Note that $s_{\alpha+k\delta}=t_{-k\alpha^{\vee}}s_{\alpha}$, where $t_{\lambda}\in\mathsf{Aff}(V)$ is the translation $t_{\lambda}(x)=x+\lambda$ by $\lambda\in V$. Let $s_0=s_{\alpha_0}$ (and so $s_0=t_{\varphi^{\vee}}s_{\varphi}$) and $\Saff=\{s_0,s_1,\ldots,s_n\}$.

The \textit{affine Weyl group} is the subgroup $\Waff$ of $\mathsf{Aff}(V)$ generated by $\{s_{\alpha+k\delta}\mid \alpha+k\delta\in \Phi\}$. Then $(\Waff,\Saff)$ is a Coxeter system of type $\tilde{\sB}_n$. Let $\ell:\Waff\to\ZZ_{\geq 0}$ be the length function. We have 
$
\Waff=Q\rtimes W_0
$
where $Q=\ZZ\alpha_1^{\vee}+\cdots+\ZZ\alpha_n^{\vee}=\{x\in \ZZ^n\mid x_1+\cdots+x_n=0\pmod{2}\}$ is the \textit{coroot lattice}. The left and right descent sets of $w\in\Waff$ are
$$
D_L(w)=\{s\in \Saff\mid \ell(sw)=\ell(w)-1\}\quad\text{and}\quad D_R(w)=\{s\in \Saff\mid \ell(ws)=\ell(w)-1\}. 
$$

If $J\subseteq \{0,1,\ldots,n\}$ is a proper subset, then the parabolic subgroup $W_J=\langle s_i\mid i\in J\rangle$ of $\Waff$ is spherical and hence has a longest element~$\sw_J$. If $i_1,\ldots,i_k\in\{0,1,\ldots,n\}$ we write 
\begin{align*}
W_{i_1,\ldots,i_k}=W_{\{0,1,\ldots,n\}\backslash\{i_1,\ldots,i_k\}}\quad\text{and}\quad \sw_{i_1,\ldots,i_k}=\sw_{\{0,1,\ldots,n\}\backslash\{i_1,\ldots,i_k\}}.
\end{align*}
In particular, $W_n$ is a subgroup of type $\sD_n$, with longest element $\sw_n$, and $\sw_0$ is the longest element of~$W_0$.

Let $\omega_1,\ldots,\omega_n\in V$ be the \textit{fundamental coweights} (defined by $\langle\omega_i,\alpha_j\rangle=\delta_{i,j}$ for $1\leq i,j\leq n$). Thus $\omega_i=e_1+\cdots+e_i$ for $1\leq i\leq n$. The \textit{coweight lattice} is 
$
P=\ZZ\omega_1+\cdots+\ZZ\omega_n=\ZZ^n.
$ The \textit{extended affine Weyl group} is 
$
\Wext=P\rtimes W_0,
$
and since $Q\subseteq P$ we have $\Waff\leq \Wext$. The group $\Wext$ naturally acts on $V$, as a subgroup of $\mathsf{Aff}(V)$. There is a distinct action of $\Wext$ on $V\oplus\RR\delta$, preserving the root system $\Phi$, by (for $u\in W_0$ and $\lambda\in P$) 
$$
u\star(\alpha+k\delta)=u(\alpha)+k\delta\quad\text{and}\quad t_{\lambda}\star(\alpha+k\delta)=\alpha+(k-\langle\lambda,\alpha\rangle)\delta\,.
$$
The \textit{inversion set} of $w\in\Wext$ is 
$
\Phi(w)=\{\alpha+k\delta\in\Phi^+\mid w^{-1}\star(\alpha+k\delta)\in-\Phi^+\},
$
and the \textit{length} of $w\in\Wext$ is $\ell(w)=|\Phi(w)|$ (agreeing with Coxeter length on restriction to $\Waff$).

We have
$
\Wext=\Waff\rtimes \Sigma
$
where $\Sigma=\{w\in\Wext\mid \ell(w)=0\}\cong P/Q\cong \ZZ_2$. The nontrivial element $\sigma$ of $\Sigma$ satisfies $\sigma s_0\sigma^{-1}=s_1$, $\sigma s_1\sigma^{-1}=s_0$, and $\sigma s_i\sigma^{-1}=s_i$ for $2\leq i\leq n$. Explicitly, $\sigma=t_{e_1}s_{e_1}$. For $w,w'\in\Wext$ we have $\Phi(w)=\Phi(w')$ if and only if $w'\in\{w,w\sigma\}$.

The closures of the open connected components of $V\backslash\bigcup_{\beta\in\Phi}H_{\beta}$ are called \textit{alcoves}. Let $\mathbb{A}$ denote the set of all alcoves. The \textit{fundamental alcove} is given by 
$$
A_0=\{x\in V\mid 0\leq \langle x,\alpha\rangle\leq 1\text{ for all $\alpha\in\Phi_0^+$}\}.
$$
The vertex set of $A_0$ is $\{0\}\cup\{\omega_i/m_i\mid 1\leq i\leq n\}$ where $\varphi=\sum_{i=1}^nm_i\alpha_i$ is the highest root of~$\Phi_0$. We call $0$ the \textit{type $0$} vertex, and $\omega_i/m_i$ the \textit{type $i$} vertex of $A_0$. The affine Weyl group $\Waff$ acts simply transitively on $\mathbb{A}$, and the \textit{type} of an arbitrary vertex is given by this action. 

The extended affine Weyl group $\Wext$ also acts transitively on $\mathbb{A}$, and the stabiliser of $A_0$ is~$\Sigma$. The inversion set $\Phi(w)$ has the geometric interpretation as the set of roots $\alpha+k\delta\in\Phi^+$ such that $H_{\alpha+k\delta}$ separates the alcoves $A_0$ and~$wA_0$. With this interpretation, the length function $\ell(w)$ on $\Wext$ is given by counting the hyperplanes separating $A_0$ from $wA_0$. Explicitly, for $\lambda\in P$ and $u\in\Wfin$ we have
\begin{align}\label{eq:length}
\ell(t_{\lambda}u)=\sum_{\alpha\in\Phi_0^+,\,u^{-1}(\alpha)>0}|\langle\lambda,\alpha\rangle|+\sum_{\alpha\in\Phi_0^+,\,u^{-1}(\alpha)<0}|\langle\lambda,\alpha\rangle-1|.
\end{align}

It is useful to note that, for $x\in\Wext$
\begin{align}\label{eq:add}
\ell(x\sw_n)=\ell(x)+\ell(\sw_n)\,\,\Longleftrightarrow \,\,D_R(x)\subseteq\{s_n\} 
\end{align}
and similarly $\ell(\sw_nx)=\ell(\sw_n)+\ell(x)\Longleftrightarrow D_L(x)\subseteq\{s_n\}$. This is a special case of the general Coxeter group fact that for a spherical subset $J$ we have $\ell(x \sw_J)=\ell(x)+\ell(\sw_J)$ if and only if $\ell(xs)>\ell(x)$ for all $s\in J$. In particular, note that 
$$
\ell(x\sw_ny)=\ell(x)+\ell(\sw_n)+\ell(y)\,\,\Longrightarrow\,\, D_R(x),D_L(y)\subseteq\{s_n\},
$$
however the converse is not true (for example $x=y=s_3s_2s_3s_0s_2s_3$ in $\tilde{\sB}_3$). See Remark~\ref{rem:Shipaper}. 

The element
\begin{align}\label{eq:fu}
\fu=\sw_{0,n}\sw_{n}\in W_n
\end{align}
will play an important role. 

\begin{lemma}\label{lem:etalebasics}
We have
\begin{compactenum}[$(1)$]
\item $\Phi_n^+=\{e_i-e_j\mid 1\leq i<j\leq n\}\cup\{-(e_i+e_j)+\delta\mid 1\leq i<j\leq n\}$;
\item $\Phi(\fu)=\{-(e_i+e_j)+\delta\mid 1\leq i<j\leq n\}$ and $\ell(\fu)=n(n-1)/2$. 
\end{compactenum}
\end{lemma}

\begin{proof}
(1) If $\beta\in\Phi_{n}^+$ then $\beta=a_0\alpha_0+\cdots+a_{n-1}\alpha_{n-1}$ with $a_0,\ldots,a_{n-1}\in\ZZ_{\geq 0}$. Since $\alpha_0=-\varphi+\delta$ this gives 
$
\beta=\alpha+a_0\delta
$ where $\alpha=-a_0\varphi+a_1\alpha_1+\cdots+a_{n-1}\alpha_{n-1}$. Since $\beta\in\Phi$ we have $\alpha\in\Phi_0$, and this forces $a_0\in\{0,1\}$ (for otherwise the absolute value of the coefficient of $\alpha_n$ in $\alpha$ strictly exceeds $2$, contradicting the fact that $\Phi_0$ is of type $\sB_n$). If $a_0=0$ then $\beta\in\Phi_{0,n}^+$, and if $a_0=1$ then by inspection of the $\sB_n$ root system we see that 
$
\alpha\in\{-(e_i+e_j)\mid 1\leq i<j\leq n\}
$ (the set of roots of $\Phi_0$ with coefficient of $\alpha_n$ equal to~$-2$). 

(2) Since $\fu=\sw_{0,n}\sw_n$, by properties of longest elements we have $\ell(\fu)=\ell(\sw_n)-\ell(\sw_{0,n})=n(n-1)/2$. Moreover, $\Phi(\fu)=\Phi_n^+\backslash\Phi_{0,n}^+=\{-(e_i+e_j)+\delta\mid 1\leq i<j\leq n\}$. 
\end{proof}

To make specific calculations (for example, using the length formula~(\ref{eq:length})) it is convenient to write elements of $\Wext$ in the standard form $t_{\la}v$ with $\la\in P$ and $v\in\Wfin$. The following lemma gives these standard forms for $\sw_n$ and $\fu$. 

\begin{lemma}\label{lem:etaleform}
We have 
\begin{align*}
\sw_{n}=\begin{cases}
t_{\omega_n}\sw_0&\text{if $n$ is even}\\
t_{\omega_n-\omega_1}\sw_{0,1}&\text{if $n$ is odd}
\end{cases}&&\fu&=\begin{cases}
t_{\omega_n}\sw_{0,n}\sw_0&\text{if $n$ is even}\\
t_{\omega_{n-1}}\sw_{0,n}\sw_{0,1}&\text{if $n$ is odd}.
\end{cases}
\end{align*}
\end{lemma}

\begin{proof}
Write $w$ for the right hand side of the claimed formula for $\sw_n$ (in the two cases $n$ even/odd). Since $\omega_n\in Q$ for $n$ even, and $\omega_n-\omega_1\in Q$ for $n$ odd, we have $w\in\Waff$. A direct calculation shows that $\Phi(w)=\Phi_n^+$, and hence $w=\sw_n$. The formulae for $\fu$ now follow directly from the definition $\fu=\sw_{0,n}\sw_n$. 
\end{proof}

\subsection{Intervals and convex hulls}\label{sec:convex}

Let $\peq$ denote the right weak order on $\Waff$ (we will not use this order on the extended affine Weyl group). That is, for $x,y\in\Waff$ we have $x\peq y$ if and only if $\ell(x^{-1}y)=\ell(y)-\ell(x)$, if and only if $x$ is a prefix of $y$ (meaning that there exists a reduced expression for $y$ starting with a reduced expression for $x$). For $x,y\in\Waff$ with $x\peq y$, the \textit{interval} $[x,y]$ is the set
$$
[x,y]=\{z\in\Waff\mid x\peq z\peq y\}.
$$
It is known that any Coxeter group is a complete meet semilattice with respect to the right weak order (see~\cite{BB}). In particular, for any $x,y\in\Waff$ there exists a unique \textit{meet} (greatest lower bound) denoted $z=x\wedge y$.

The \textit{convex hull} $\mathsf{conv}\{A,B\}$ of alcoves $A,B$ is defined to be the union of all alcoves $C$ such that there exists a gallery of reduced type from $A$ to $B$ passing through $C$. Geometrically, if $\Phi(A,B)$ denotes the set of roots $\beta\in\Phi^+$ such that $H_{\beta}$ separates $A$ and $B$, then 
$$
\mathsf{conv}\{A,B\}=\bigcap_{\beta\in\Phi^+\backslash\Phi(A,B)}H_{\beta}(A,B),
$$
where $H_{\beta}(A,B)$ denotes the closed half space bounded by $H_{\beta}$ that contains both $A$ and $B$ (see \cite[Theorem~3.131]{AB:08} for the equivalence of these two interpretations of convex hulls). 

For $x,y\in\Wext$ we will use the abbreviated notation 
$
\mathsf{conv}\{x,y\}=\mathsf{conv}\{xA_0,yA_0\}$. If $w\in\Waff$ then the description of $\mathsf{conv}\{1,w\}$ in terms of galleries gives
\begin{align}\label{eq:convexhullinterval}
\mathsf{conv}\{1,w\}=\bigcup_{v\in [1,w]}vA_0,
\end{align}
providing a bijection between elements of the interval $[1,w]$ and alcoves contained in $\mathsf{conv}\{1,w\}$. 

If $w\in\Wext$ then $\Phi(A_0,wA_0)=\Phi(w)$, and hence 
\begin{align}\label{eq:intersectionspace}
\mathsf{conv}\{1,w\}=\bigcap_{\beta\in\Phi^+\backslash\Phi(w)} H_{\beta}(A_0,wA_0).
\end{align}
Thus $\mathsf{conv}\{1,w\}$ is described by a family of inequalities, one for each root $\alpha\in\Phi_0^+$, of the form $M_{\alpha}\leq \langle x,\alpha\rangle\leq M_{\alpha}'$ where $M_{\alpha}$ and $M_{\alpha}'$ are integers. The integers $M_{\alpha}$ and $M_{\alpha}'$ are determined as follows, for $\alpha\in\Phi_0^+$. (a) If $\alpha\in\Phi(w)$ then $M_{\alpha}=-(k_{\alpha}+1)$ and $M_{\alpha}'=1$ where $k_{\alpha}\geq 0$ is the maximal integer such that $\alpha+k_{\alpha}\delta\in\Phi(w)$. (b) If $-\alpha+\delta\in\Phi(w)$ then $M_{\alpha}=0$ and $M_{\alpha}'=k_{\alpha}+1$ where $k_{\alpha}\geq 1$ is the maximal integer such that $-\alpha+k_{\alpha}\delta\in\Phi(w)$. (c) If $\Phi(w)$ contains neither $\alpha$ nor $-\alpha+\delta$ then $M_{\alpha}=0$ and $M_{\alpha}'=1$.

Equations~(\ref{eq:convexhullinterval}) and~(\ref{eq:intersectionspace}) give a useful method of counting the number of elements in an interval $[1,w]$ for $w\in\Waff$. Specifically, we have
\begin{align}\label{eq:volume}
\#[1,w]=\frac{\mathsf{Vol}(\mathsf{conv}\{1,w\})}{\mathsf{Vol}(A_0)}=2^{n-1}n!\times \mathsf{Vol}(\mathsf{conv}\{1,w\}),
\end{align}
where $\mathsf{Vol}$ denotes usual Euclidean volume. The calculation $\mathsf{Vol}(A_0)=1/(2^{n-1}n!)$ follows from the fact that $A_0$ has vertex set $\{0,e_1,\frac{1}{2}(e_1+e_2),\ldots,\frac{1}{2}(e_1+\cdots+e_n)\}$. Computing $\mathsf{Vol}(\mathsf{conv}\{1,w\})$ can typically be achieved by finding the inequalities determining $\mathsf{conv}\{1,w\}$ (as described above) and performing an iterated integration.

Let $\mathsf{conv}\{x,y\}^{\circ}$ denote the interior (in the Euclidean sense) of $\mathsf{conv}\{x,y\}$. Thus $\mathsf{conv}\{x,y\}^{\circ}$ is described by the strict inequalities $M_{\alpha}<\langle x,\alpha\rangle<M_{\alpha}'$, with $M_{\alpha}$ and $M_{\alpha}'$ as described above.

\subsection{Kazhdan-Lusztig theory}\label{sec:KL}

Kazhdan-Lusztig theory partitions the Coxeter group into left cells, and into right cells, and into two-sided cells. We do not need the precise definitions in this paper, and so we will only recall here the very few parts of the theory that will be required. See \cite{KL:79,Lus:03} for details. 

Let $\leq_L$, $\leq_R$, and $\leq_{LR}$ denote the left, right, and two-sided Kazhdan-Lusztig preorders on~$\Wext$ (defined using the equal parameter extended affine Hecke algebra, however we will not need the details of this construction). The equivalence classes of the associated equivalence relations $\sim_L$, $\sim_R$, and $\sim_{LR}$ are left, right, and two-sided cells (for example, $x\sim_R y$ means $x\leq_R y$ and $y\leq_R x$). Each two-sided cell is a union of left cells, and a union of right cells. Moreover, $x\leq_L y$ if and only if $x^{-1}\leq_R y^{-1}$, and so $x\sim_L y$ if and only if $x^{-1}\sim_R y^{-1}$, and hence $\Gamma$ is a right cell if and only if $\Gamma^{-1}$ is a left cell. Thus every statement about right/left cells has an obvious analogous statement about left/right cells. 

The preorder $\leq_{LR}$ naturally induces a partial order on the set of two-sided cells by $\Delta\leq_{LR}\Delta'$ if and only if $x\leq_{LR}y$ for some $x\in\Delta$ and $y\in\Delta'$ (and ``for some'' can be replaced by ``for all''). 

Let $\ba:\Wext\to \mathbb{Z}_{\geq0}$ be Lusztig's $\ba$-function. We will need the following general properties  (we restrict to the ``equal parameter'' setting):
\begin{compactenum}[$(1)$]
\item If $x\leq_R y$ then $D_L(x)\supseteq D_L(y)$. Hence if $x\sim_R y$ then $D_L(x)=D_L(y)$. 
\item If $x\leq_{LR} y$ then $\ba(x)\geq \ba(y)$. In particular, $\ba$ is constant on two-sided cells. 
\item If $x\leq_L y$ and $\ba(x)=\ba(y)$ then $x\sim_L y$.
\item If $w\sim_{LR}sw$ with $s\in \Saff$ then $w\sim_L sw$. 
\end{compactenum}
Statement (1) is a basic standard consequence of the definitions, and (2) and (3) are Lusztig's properties $\mathsf{P4}$ and $\mathsf{P9}$ (see \cite{Lus:03}; these statements are known to hold in the equal parameter setting). Statement (4) easily follows: interchanging the roles of $w$ and $sw$ if necessary, we may assume that $\ell(sw)=\ell(w)+1$, and hence $sw\leq_L w$ (this is immediate from the definitions). Since $w\sim_{LR}sw$ we have $\ba(w)=\ba(sw)$ (by (2)), and so (3) gives $sw\sim_L w$. 

Lusztig proved in~\cite[Theorem~4.8]{Lus:IV} that for an irreducible affine Weyl group (with equal parameters) there exists a bijection between the set of two-sided cells and unipotent conjugacy classes in the complex dual group $G^{\vee}$, and moreover that if the cell $\Delta$ corresponds to the class of a unipotent element $u\in G^{\vee}$ then $\ba(w)=\dim\mathcal{B}_u$ for all $w\in \Delta$, where $\mathcal{B}_u$ is the Springer fibre of all Borel subgroups containing~$u$. Furthermore, Lusztig conjectured that the partial order $\leq_{LR}$ on two-sided cells agrees with partial order on unipotent classes given by the closure relation, and Bezrukavnikov proved this conjecture in \cite[Theorem 4(b)]{Bez} (see also \cite{Shi:96} for the $\tilde{\sA}_n$ case). 

In particular, in any irreducible affine Weyl group with equal parameters it is known that there is a unique lowest two-sided cell with respect to $\leq_{LR}$, denoted $\Delta_0$ (corresponding to the class of the unipotent element $1\in G^{\vee}$). This cell has the following description: Let $\mathcal{J}$ be the set of subsets $J\subseteq \Saff$ such that $W_J$ is isomorphic to $\Wfin$ (for example, in type $\tilde{E}_6$ there are three such subsets). Then $\Delta_0$ consists of all extensions in $\Wext$ of the elements $\sw_J$ with $J\in\mathcal{J}$, and $\ba(w)=\ell(\sw_0)$ for all $w\in\Delta_0$. See \cite{Shi:Lowest} for details. 

Furthermore, there is always a unique second lowest two-sided cell $\Delta_1$ in an affine Weyl group, corresponding to the class of the unipotent element $x_{\tilde{\varphi}}(1)\in G^{\vee}$ where $\tilde{\varphi}$ is the highest root of $\Phi_0^{\vee}$, and we use standard Chevalley notations in $G^{\vee}$ (see~\cite{Shi:secondlowest}). Writing $\ba(\Delta_1)$ for the constant value of $\ba(w)$ for $w\in\Delta_1$, we have $\ba(\Delta_1)<\ba(\Delta_0)=\ell(\sw_0)$, and if $w\notin\Delta_0\cup\Delta_1$ then $\ba(w)<\ba(\Delta_1)$. For type $\tilde{\sB}_n$ we have $\ba(\Delta_1)=n(n-1)$, see \cite{Shi:secondlowest}.

\begin{remark}
There is a generalisation of Kazhdan-Lusztig theory to the ``unequal parameter'' setting by Lusztig, see \cite{Lus:03} and~\cite{Bon:17}. The description of the elements of $\Delta_0$ given above also holds for unequal parameters with minor modifications required for type $\tilde{\sC}_n$ with $n\geq 1$ (including $\tilde{\sA}_1=\tilde{\sC}_1$ and $\tilde{\sB}_2=\tilde{\sC}_2$; see \cite{Gui:08}). For unequal parameters it is not necessarily true that there is a unique second lowest two-sided cell, for example see the $\tilde{\sG}_2$ example for $2>r>3/2$ in \cite[Figure~2]{GP:19}.  
\end{remark}

\section{The fundamental $J$-alcove}\label{sec:Jalcove}

In this section we recall the definition and basic properties of the \textit{fundamental $J$-alcove} from~\cite{GLP:23}. Moreover, we study certain maximal length double coset representatives $\sm_{\la}$ that will ultimately describe the intersection $\Ga\cap\Ga^{-1}$ (where $\Ga=\Ga_1$ is the right cell containing~$\sw_n$). More generally, the geometry and combinatorics of the fundamental $J$-alcove, for $J=\{n\}$, will turn out to describe the structure of the cell~$\Delta_1$.

\subsection{The fundamental $J$-alcove and pseudo-translations}

Following~\cite{GLP:23}, if $J\subseteq\{1,2,\ldots,n\}$ the \textit{fundamental $J$-alcove} is 
$$
\cA_J=\{x\in V\mid 0\leq \langle x,\alpha\rangle\leq 1\text{ for all $\alpha\in\Phi_J^+$}\}
$$
(this is a fundamental domain for the action of the ``$J$-affine Weyl group'' $Q_J\rtimes W_J$ on $V$, where $Q_J=\sum_{\alpha\in\Phi_J}\ZZ\alpha^{\vee}$, hence the terminology). Let 
$$
P_J=P\cap\cA_J\quad\text{and}\quad  G_J=\{w\in\Wfin\mid w\cA_J=\cA_J\}.
$$
We call $P_J$ the set of \textit{$J$-weights}. The group $G_J$ naturally acts on $P_J$. 

In this paper the relevant subset is $J=\{n\}$, and we write $\cA_n=\cA_{\{n\}}$, $Q_n=Q_{\{n\}}$, $P_n=P_{\{n\}}$, and $G_n=G_{\{n\}}$. Thus explicitly $\cA_n=\{x\in V\mid 0\leq \langle x,e_n\rangle\leq 1\}$ (the \textit{fundamental $n$-alcove}), $Q_n=2\ZZ e_n$, and 
$$
P_n=\{\la_1e_1+\cdots+\la_{n-1}e_{n-1}+\la_ne_n\mid \la_1,\ldots,\la_{n-1}\in\ZZ\text{ and }\la_n\in\{0,1\}\}.
$$
We have $G_n=\langle s_1,s_2,\ldots,s_{n-2},s_{n-1}'\rangle$ where $s_{n-1}'=s_{n-1}s_ns_{n-1}$.

\begin{prop} The group $G_n$ acts on $P_n$, and the set
$$
P_n^+=\{\la_1e_1+\cdots+\la_{n-1}e_{n-1}+\la_ne_n\mid \la_1\geq \la_2\geq\cdots\geq \la_{n-1}\geq 0,\,\la_n\in\{0,1\}\}
$$
is a set of orbit representatives for this action. 
\end{prop}

\begin{proof}
By definition, $G_n$ acts on $P_n$. Since $s_{n-1}'=s_{e_{n-1}}$ we see that $G_n$ is a Coxeter group of type $\sB_{n-1}$, acting on $P_n$ as signed permutations on the first $n-1$ coordinates, and fixing the final coordinate. Hence the stated set $P_n^+$ is a set of orbit representatives.
\end{proof}

\begin{remark}
If $c=s_1s_2\cdots s_{n-1}$ then $cG_nc^{-1}$ is the standard parabolic subgroup of type $\sB_{n-1}$.
\end{remark}

\begin{remark}
Note that $P_n^+\neq P_n\cap P^+$, where $P^+=\NN\omega_1+\cdots+\NN\omega_n$ is the usual set of dominant coweights. For example $\lambda=e_n=-\omega_{n-1}+\omega_n\in P_n^+$ yet $\lambda\notin P^+$. 
\end{remark}

Following \cite[Definition~3.9]{GLP:23}, for $\lambda\in P_n$ we define \textit{pseudo-translations} $\tau_{\la}\in\Wext$ by
$$
\tau_{\lambda}=\begin{cases}
t_{\lambda}&\text{if $\la_n=0$}\\
t_{\lambda}s_n&\text{if $\la_n=1$},
\end{cases}
$$
and set $\sT_n=\{\tau_{\lambda}\mid \lambda \in P_n\}$. By \cite[Corollary~3.15]{GLP:23} we have that $\sT_n$ is an abelian group, with $\sT_n\cong P/Q_n\cong \ZZ^{n-1}\times \ZZ_2$ via the natural map $\lambda\mapsto \lambda+Q_n$.

\begin{lemma}\label{lem:tauinverse}
For $g\in G_n$ and $\la\in P_n$ we have
\begin{compactenum}[$(1)$]
\item $\tau_{\lambda}^{-1}=\tau_{\la^*}$ where $\la^*=-\la_1e_1-\cdots-\la_{n-1}e_{n-1}+\la_ne_n\in P_n$. 
\item $y\tau_{\lambda}^{-1}y=\tau_{\lambda}$, where $y=\sw_{0,n}\sw_n\sw_{0,n}$ (if $n$ is even then $y$ simplifies to $\sw_n$). 
\item $g\tau_{\lambda}g^{-1}=\tau_{g\lambda}$.
\end{compactenum}
\end{lemma}

\begin{proof}
(1) is an elementary calculation (it also follows from the general results in \cite[Lemma~3.14]{GLP:23}). Then (2) follows by a direct calculation using Lemma~\ref{lem:etaleform}. For (3), see \cite[Lemma~3.21]{GLP:23}. 
\end{proof}

The type $n$ vertex of the fundamental alcove is $\xi=\frac{1}{2}(e_1+\cdots+e_n)$. Let $\cV_n$ be the set of all type $n$ vertices (of all alcoves). 

\begin{lemma}\label{lem:typenvertices}
We have $\mathsf{Stab}_{\Wext}(\xi)=\langle W_n,\sigma\rangle$ and  
$$\cV_n=\{x_1e_1+\cdots+x_ne_n\mid x_i\in\ZZ+\tfrac{1}{2}\text{ for $1\leq i\leq n$}\}.
$$
\end{lemma}

\begin{proof}
Both $\Waff$ and $\Wext$ act transitively on $\cV_n$. For $\lambda\in P$ and $v\in \Wfin$ we have $t_{\lambda}v(\xi)=\lambda+v(\xi)$. Since $\lambda\in P=\ZZ^n$ and $v$ acts by signed permutations, every coordinate of $t_{\lambda}v(\xi)$ is in $\ZZ+\frac{1}{2}$, and it is clear that all such points are in the $\Wext$ orbit of $\xi$. Hence $\cV_n=(\ZZ+\frac{1}{2})^n$ as claimed. To compute the stabiliser of $\xi$, if $w=t_{\lambda} v$ with $\lambda\in P$ and $v\in \Wfin$ then $w(\xi)=\xi$ if and only if $\lambda=\xi-v(\xi)$. Thus since $\xi-v(\xi)\in P$ we have $\mathsf{Stab}_{\Wext}(\xi)=\{t_{\xi-v(\xi)}v\mid v\in\Wfin\}$. Taking linear parts gives an isomorphism $\mathsf{Stab}_{\Wext}(\xi)\cong \Wfin$. Clearly $W_n\leq \mathsf{Stab}_{\Wext}(\xi)$. We also have $\sigma\in\mathsf{Stab}_{\Wext}(\xi)$ (as $\sigma=t_{e_1}s_{e_1}$), and so (by orders) 
$\mathsf{Stab}_{\Wext}(\xi)=\langle W_n,\sigma\rangle$. 
\end{proof}

The group $\sT_n$ acts on the set $\{A\in\mathbb{A}\mid A\subseteq \cA_n\}$ of (classical) alcoves contained in the fundamental $n$-alcove, however this action is not free. To resolve this issue, let
$$
\mathbb{X}_n=\{w\in\Wext\mid wA_0\subseteq\cA_n\}=\{w\in\Wext\mid e_n,-e_n+\delta\notin\Phi(w)\}
$$
(so $\mathbb{X}_n$ is a two fold cover of the set of alcoves contained in $\cA_n$). The group $\sT_n$ acts freely on~$\mathbb{X}_n$. By \cite[Theorem~3.11 and Corollary~3.15]{GLP:23} the set
$$
{^n}\Wfin=\{w\in \Wfin \mid \ell(s_nw)=\ell(w)+1\}
$$
is a fundamental domain for the action of $\sT_n$ on $\mathbb{X}_n$. In other words, if $w\in\Wext$ is such that the alcove $wA_0$ lies in the fundamental $n$-alcove, then $w=\tau_{\lambda}u$ for a unique $\lambda\in P_n$ and $u\in {^n}\Wfin$. Fundamental domains for the action of $\sT_n$ on $\mathbb{X}_n$ play an important role in the combinatorial approaches to Kazhdan-Lusztig theory initiated in \cite{GP:19,GP:19b} and continued in~\cite{GLP:23}. In this theory, formulae are often simplified by selecting a fundamental domain well adapted to the specific setting. The numerical observation $|{^n}\Wfin|=|W_n|=|\sB|$ hints that there are other natural fundamental domains for this action in the case of~$\Delta_1$. In the following lemma we show that $W_n\subseteq \mathbb{X}_n$ is indeed a fundamental domain, and in Section~\ref{sec:funddom} we will show that $\fu\sB$ is a third fundamental domain.

\begin{lemma}\label{lem:Dnfunddom}
The set $W_n$ is a fundamental domain for the action of $\mathsf{T}_n$ on $\mathbb{X}_n$. Thus if $w\in\Wext$ with $wA_0\subseteq \cA_n$ then there exist unique elements $\lambda\in P_n$ and $u\in W_n$ such that $w=\tau_{\lambda}u$. 
\end{lemma}

\begin{proof}
Since $|W_n|=|{^n}\Wfin|$, and since ${^n}\Wfin$ is a fundamental domain, it suffices to show that each element $w\in\mathsf{X}_n$ has at least one expression as $w=\tau_{\lambda}v$ for some $\lambda\in P_n$ and $v\in W_n$. Let $w\in\mathbb{X}_n$. Then $w(\xi)$ is the type $n$ vertex of an alcove in $\cA_n$, and hence by Lemma~\ref{lem:typenvertices} we have $w(\xi)=x_1e_1+\cdots+x_{n-1}e_{n-1}+\frac{1}{2}e_n$ with $x_1,\ldots,x_{n-1}\in \ZZ+\frac{1}{2}$. Since $t_{e_1},\ldots,t_{e_{n-1}}\in \sT_n$ it is clear that there is $\lambda\in P_n$ such that $\tau_{\la}^{-1}w\xi=\xi$. Thus $\tau_{\la}^{-1}w\in\mathsf{Stab}_{\Wext}(\xi)$, and so by Lemma~\ref{lem:typenvertices} we have $w\in \tau_{\lambda}W_n\sqcup\tau_{\lambda}W_n\sigma$. If $w\in\tau_{\lambda}W_n$ then we are done. If $w\in \tau_{\la}W_n\sigma$ then note that $\tau_{e_n}=t_{e_n}s_n\in\mathsf{Stab}_{\Wext}(\xi)\backslash W_n$, and since $W_n$ is an index $2$ subgroup of the stabiliser we have $W_n\sigma=\tau_{e_n}W_n$, giving $w\in\tau_{\la}\tau_{e_n}W_n\in \sT_nW_n$ as required. 
\end{proof}

\subsection{Maximal length double coset representatives}

Maximal length double coset representatives naturally play an important role in Kazhdan-Lusztig theory (see \cite{LusztigXi}). In particular, since $\Ga$ is the right cell containing $\sw_n$ we have $D_L(w)=D_R(w)=\{s_0,\ldots,s_{n-1}\}$ for all $w\in\Ga\cap\Ga^{-1}$ (as left descent sets are constant on right cells, and right descent sets are constant on left cells). That is, each element of $\Ga\cap\Ga^{-1}$ is the unique maximal length element in its $W_n$-double coset. Thus determining $\Ga\cap\Ga^{-1}$ is equivalent to determining which $W_n$-double cosets intersect $\Ga\cap\Ga^{-1}$, and to determine explicitly the (unique) maximal length elements in these double cosets. 

For $\lambda\in P_n$ let
$$
\sm_{\lambda}=(\text{the longest element of the double coset $W_n\tau_{\lambda} W_n$}).
$$
We have the following basic properties. 
\begin{prop}\label{prop:dominance}
For $\la\in P_n$ and $g\in G_n$ we have $\sm_{g\lambda}=\sm_{\lambda}$, $\sm_{\lambda}^{\sigma}=\sm_{\lambda}$, and $\sm_{\lambda}^{-1}=\sm_{\lambda}$. 
\end{prop}

\begin{proof}
To prove $\sm_{g\lambda}=\sm_{\lambda}$ it suffices to prove that $W_n\tau_{g\lambda}W_n=W_n\tau_{\lambda}W_n$ for $\lambda\in P_n$ and $g\in G_n$, and it is sufficient to verify this for the generators of $G_n$. If $g\in\{s_1,\ldots,s_{n-2}\}$ then $g\in W_n$ and so $W_n\tau_{g\lambda}W_n=W_ng\tau_{\lambda}g^{-1}W_n=W_n\tau_{\lambda}W_n$. The case $g=s_{n-1}'$ requires a different argument, because $s_{n-1}'\notin W_n$. Let $\alpha=-e_{n-1}-e_n+\delta$ and $\beta=e_{n-1}-e_n$. Since $\alpha,\beta\in\Phi_n$ we have $y=s_{\alpha}s_{\beta}\in W_n$. A direct calculation shows that $y\tau_{\lambda}y^{-1}=\tau_{s_{n-1}'\lambda}$ for all $\lambda \in P_n$, and hence $W_n\tau_{g\lambda}W_n=W_n\tau_{\lambda}W_n$ for all $g\in G_n$. 

Now, since $\sigma=t_{e_1}s_{e_1}$ we have $\sigma\tau_{\lambda}\sigma=\tau_{s_{e_1}\lambda}$. Thus $\sm_{\lambda}^{\sigma}=\sm_{s_{e_1}\lambda}$. Since $s_{e_1}\in G_n$ the result follows from the statement $\sm_{\lambda}=\sm_{g\lambda}$ for all $g\in G_n$. Finally, by Lemma~\ref{lem:tauinverse} the double coset $W_n\tau_{\la}W_n$ is stable under inversion, and so $\sm_{\la}^{-1}=\sm_{\la}$. 
\end{proof}

Since $P_n^+$ is a set of orbit representatives for the action of $G_n$ on $P_n$, by Proposition~\ref{prop:dominance} we have $\sm_{\la}=\sm_{\la^+}$ where $\la^+$ is the unique element of $P_n^+$ in the $G_n$-orbit of $\la$. Thus we can restrict to $\lambda\in P_n^+$, in which case the following theorem gives an explicit formula for~$\sm_{\lambda}$.

\begin{thm}\label{thm:mlambda}
If $\lambda\in P_n^+$ then $\sm_{\lambda}=\sw_n\fu^{-1}\tau_{\lambda}\fu$, and $\ell(\sw_n\fu^{-1}\tau_{\lambda}\fu)=\ell(\sw_n)+\ell(\fu^{-1}\tau_{\lambda}\fu)$.
\end{thm}

\begin{proof}
Temporarily write $\mathsf{p}_{\la}=\sw_n\fu^{-1}\tau_{\lambda}\fu$. Since $\ssp_{\la}\in W_n\tau_{\la}W_n$, showing that $\ssp_{\la}=\sm_{\la}$ is equivalent to showing that $D_L(\ssp_{\la})=D_R(\ssp_{\la})=\{s_0,\ldots,s_{n-1}\}$. Equivalently, we must show that $\ssp_{\la}\star\alpha_i<0$ and $\ssp_{\la}^{-1}\star\alpha_i<0$ for all $0\leq i\leq n-1$. Using Lemma~\ref{lem:tauinverse} and the definition $\fu=\sw_{0,n}\sw_n$ we have 
$
\ssp_{\la}^{-1}=\fu^{-1}\tau_{\la}^{-1}\fu\sw_n=\sw_n\fu^{-1}\tau_{\la}\fu=\ssp_{\la},
$
and so it is sufficient to show that $\ssp_{\la}^{-1}\star\alpha_i<0$ for $0\leq i\leq n-1$. Since $\sw_n\star\alpha_i=-\alpha_{j}$ for some $0\leq j\leq n-1$ it is enough to show that $(\fu^{-1}\tau_{\la}^{-1}\fu)\star\alpha_i>0$ for all $0\leq i\leq n-1$. We have
\begin{align*}
(\fu^{-1}\tau_{\lambda}^{-1}\fu)\star\alpha_i=\begin{cases}
\alpha_i+(\la_{n-i}-\la_{n-i+1})\delta&\text{if $\la_n=0$ and $1\leq i\leq n-1$}\\
-\varphi+(1+\la_{n-1})\delta&\text{if $\la_n=0$ and $i=0$}\\
-\varphi+(1+\la_{n-1})\delta&\text{if $\la_n=1$ and $i=1$}\\
\alpha_i+(\la_{n-i}-\la_{n-i+1})\delta&\text{if $\la_n=1$ and $2\leq i\leq n-1$}\\
\alpha_1+\la_{n-1}\delta&\text{if $\la_n=1$ and $i=0$}.
\end{cases}
\end{align*}
In all cases $(\fu^{-1}\tau_{\lambda}^{-1}\fu)\star \alpha_i>0$ as required. Thus $\sm_{\la}=\ssp_{\la}$. 

Since $s_i\in D_L(\fu^{-1}\tau_{\la}\fu)$ if and only if $(\fu^{-1}\tau_{\la}^{-1}\fu)\star\alpha_i<0$, the calculation above shows that $D_L(\fu^{-1}\tau_{\la}\fu)\subseteq\{s_n\}$, and hence $\ell(\sw_n\fu^{-1}\tau_{\la}\fu)=\ell(\sw_n)+\ell(\fu^{-1}\tau_{\la}\fu)$ by~(\ref{eq:add}). 
\end{proof}

\begin{cor}
For $\lambda,\mu\in P_n$ we have $\sm_{\lambda}=\sm_{\mu}$ if and only if $G_n\lambda=G_n\mu$. 
\end{cor}

\begin{proof}
If $G_n\la=G_n\mu$ then $\sm_{\la}=\sm_{\mu}$ by Proposition~\ref{prop:dominance}. Conversely, if $\sm_{\la}=\sm_{\mu}$ then by Proposition~\ref{prop:dominance} we have $\sm_{\la^+}=\sm_{\mu^+}$, where $\la^+\in G_n\la\cap P_n^+$ and $\mu^+\in G_n\mu\cap P_n^+$. The formula in Theorem~\ref{thm:mlambda} then gives $\tau_{\lambda^+}=\tau_{\mu^+}$, and hence $\lambda^+=\mu^+$, so $G_n\la=G_n\mu$.  
\end{proof}

Since $\ell(\sm_{\lambda})=\ell(\sw_n)+\ell(\fu^{-1}\tau_{\lambda}\fu)$, determining $\ell(\sm_{\lambda})$ is reduced to computing $\ell(\fu^{-1}\tau_{\lambda}\fu)$. We make this calculation in the following proposition. For $\la\in P_n$ we write 
$$
\tau_{\la}'=\fu^{-1}\tau_{\la}\fu.
$$

\begin{prop}\label{prop:taulength}
If $\la\in P_n$ then 
$$
\ell(\tau_{\la}')=3\sum_{i=1}^{n-1}|\la_i|+\sum_{1\leq i<j\leq n-1}(|\la_i-\la_j|+|\la_i+\la_j|).
$$
In particular, this length is independent of $\la_n$. Moreover, if $\la\in P_n^+$ we have 
$$\ell(\tau_{\la}')=\sum_{i=1}^{n-1}(2n-2i+1)\la_i,
$$
and so if $\la,\mu\in P_n^+$ then $\ell(\tau_{\la}'\tau_{\mu}')=\ell(\tau_{\la}')+\ell(\tau_{\mu}')$. 
\end{prop}

\begin{proof}
It follows quickly from the definitions of $\fu$ and $\tau_{\la}$ that if $\la\in P_n$ then
$$
\tau_{\la}'=\fu^{-1}\tau_{\lambda}\fu=\begin{cases}
t_{\sw_0\sw_{0,n}\lambda}&\text{if $\la_n=0$}\\
t_{\sw_0\sw_{0,n}(\lambda-2e_n)}s_{e_1}&\text{if $\la_n=1$}.
\end{cases}
$$
Straightforward calculations, using~(\ref{eq:length}), complete the proof of the first formula, and assuming $\la\in P_n^+$ it is clear that the formula simplifies as claimed. To verify the formula $\ell(\tau_{\la}'\tau_{\mu}')=\ell(\tau_{\la}')+\ell(\tau_{\mu}')$ for $\la,\mu\in P_n^+$, note that if $\la_n+\mu_n\in\{0,1\}$ then $\tau_{\la}\tau_{\mu}=\tau_{\la+\mu}$, and if $\la_n+\mu_n=2$ then $\tau_{\la}\tau_{\mu}=\tau_{\la+\mu-2e_n}$. In particular, since the formula for $\ell(\tau_{\la'})$ is linear in $\la_1,\ldots,\la_{n-1}$, and is independent of $\la_n$, the claimed formula follows.
\end{proof}

\begin{cor}\label{cor:mlength}
If $\la\in P_n^+$ then 
\begin{align*}
\ell(\sm_{\lambda})=n(n-1)+\sum_{i=1}^{n-1}(2n-2i+1)\la_i. 
\end{align*}
\end{cor}

\begin{proof}
This is immediate from Proposition~\ref{prop:taulength} and Theorem~\ref{thm:mlambda}. 
\end{proof}

\section{The second lowest two-sided cell in type $\tilde{\sB}_n$}\label{sec:main}

Recall Shi's description of $\Delta_1$ from the introduction (Theorem~\ref{thm:Shi}). We note that Shi works in the non-extended affine Weyl group, however the statements easily extend to the extended affine Weyl group (as right cells are invariant under right multiplication by $\sigma$, and the set of right cells is permuted by left multiplication by $\sigma$, with analogous statements for left cells), and hence we have stated the results for the extended affine Weyl group. 

In this section we derive some criteria for membership in $\Delta_0$ and $\Delta_1$. These criteria are used to deduce an explicit description of the set $\sB$ in Theorem~\ref{thm:determineB}. We then prove our main theorems, describing precisely the elements of the intersection $\Ga_u\cap\Ga_v^{-1}$ for $u,v\in\sB$. 

\subsection{Membership conditions for the cells $\Delta_0$ and $\Delta_1$}

Let $v\in\Wext$. We say that $w\in\Wext$ \textit{is an extension of~$v$} if there exist $x,y\in \Wext$ such that $w=xvy$ with $\ell(w)=\ell(x)+\ell(v)+\ell(y)$. 

The \textit{parallelism class} of $\alpha\in\Phi_0^+$ is $[\alpha]=\{\alpha+k\delta\mid k\geq 0\}\cup\{-\alpha+k\delta\mid k\geq 1\}$. For a subset $R\subseteq\Phi^+$, we say that $R$ \textit{misses} $[\alpha]$ if $R\cap[\alpha]=\emptyset$, while if $R\cap[\alpha]\neq\emptyset$ we say that $R$ \textit{meets} $[\alpha]$. If $R=\Phi(w)$ is an inversion set then it is clear that $\Phi(w)$ meets $[\alpha]$ if and only if either $\alpha\in \Phi(w)$ or $-\alpha+\delta\in \Phi(w)$ (and only one of these roots can be contained in $\Phi(w)$). In particular, to show that $\Phi(w)$ misses $[\alpha]$ it is sufficient to show that $\{\alpha,-\alpha+\delta\}\cap\Phi(w)=\emptyset$.

\begin{thm}[{\cite{Shi:Lowest}}]\label{thm:Delta0}
Let $w\in\Wext$. Then $w\in\Delta_0$ if and only if $w$ is an extension of $\sw_0$ or $\sw_1$, if and only if $\Phi(w)$ meets every parallelism class. 
\end{thm}

Shi's description of $\Delta_1$ says that $w\in\Delta_1$ if and only if $w$ is an extension of $\sw_n$ and is not an extension of either $\sw_0$ or $\sw_1$. Therefore it is desirable to have a simple criterion to decide if $w$ is an extension of~$\sw_n$. There is an obvious necessary condition. 

\begin{prop}\label{prop:conditionA}
If $w$ is an extension of $\sw_n$ then $\Phi(w)$ meets every parallelism class of long roots. In particular, if $w\in\Delta_1$ then $\Phi(w)$ meets every parallelism class of long roots. 
\end{prop}

\begin{proof}
If $w=x\sw_n y$ with $\ell(w)=\ell(x)+\ell(\sw_n)+\ell(y)$ then the roots $x\star\Phi_n^+$ are contained in $\Phi(w)$. Since $\Phi_n^+$ contains roots from every parallelism class of long roots, and since the action of $\Wext$ on roots permutes long root parallelism classes, it follows that $\Phi(w)$ meets all parallelism classes of long roots. For the final statement, apply Shi's description Theorem~\ref{thm:Shi}. 
\end{proof}

Thus if $w$ is an extension of $\sw_n$ but is not an extension of either $\sw_0$ nor $\sw_1$ (that is, $w\in\Delta_1$), then necessarily $\Phi(w)$ meets all long root parallelism classes and misses at least one short root parallelism class. The converse does not hold. Indeed, the statement ``$\Phi(w)$ meets all long root parallelism classes and misses at least one short root parallelism class'' is precisely Shi's \textit{condition (A)} from \cite[Section~5]{ShiBn}, and the elements $\sw_0s_n$ and $\sw_1s_n$ give examples of elements satisfying this condition, but not being an extension of $\sw_n$ (for example, no reduced expression for $\sw_0s_n$ can contain $s_0$). 

It turns out that there is a useful necessary and sufficient condition for $w$ to be an extension of $\sw_n$ in terms of convex hulls (see Theorem~\ref{thm:convexhulls}), which when combined with Theorem~\ref{thm:Delta0} gives a necessary and sufficient condition for membership of $\Delta_1$. 

\begin{thm}\label{thm:convexhulls}
Let $i\in\{0,1,\ldots,n\}$. The element $w\in\Wext$ is an extension of $\sw_i$ if and only if there exists a type $i$ vertex in $\mathsf{conv}\{1,w\}^{\circ}$. 
\end{thm}

\begin{proof}
Without loss of generality we may assume that $w\in\Waff$. Suppose that $w=x\sw_iy$ with $\ell(w)=\ell(x)+\ell(\sw_i)+\ell(y)$. Let $\nu$ be the type $i$ vertex of $xA_0$. The alcoves containing $\nu$ are precisely the alcoves $xvA_0$ with $v\in W_i$, and since $xv\peq x\sw_iy=w$ these alcoves are contained in $\mathsf{conv}\{1,w\}$. If $\nu$ is on the boundary of $\mathsf{conv}\{1,w\}$ then there exists an alcove containing $\nu$ not contained in $\mathsf{conv}\{1,w\}$, a contradiction, and so $\nu\in\mathsf{conv}\{1,w\}^{\circ}$. 

Now suppose that there is a type $i$ vertex $\nu$ in $\mathsf{conv}\{1,w\}^{\circ}$. Let $R$ be the set of all alcoves containing~$\nu$ (a ``residue'' of type $W_i$). Since $\nu\in\mathsf{conv}\{1,w\}^{\circ}$ we have $R\subseteq\mathsf{conv}\{1,w\}$ (because each defining wall inequality is strict at the vertex). Choose any $x_1\in \Waff$ with $x_1A_0\in R$, and let $x$ be the minimal length element of the coset $x_1W_i$. Then $xA_0\subseteq R$ and $\ell(x\sw_i)=\ell(x)+\ell(\sw_i)$. Since $x\sw_iA_0\subseteq R\subseteq \mathsf{conv}\{1,w\}$ we have $x\sw_i\in [1,w]$ and hence $w=x\sw_iy$ with $\ell(w)=\ell(x\sw_i)+\ell(y)=\ell(x)+\ell(\sw_i)+\ell(y)$. 
\end{proof}

\begin{remark} Theorem~\ref{thm:convexhulls} holds for arbitrary irreducible affine Coxeter groups (with the same proof), giving a useful geometric criteria to determine if a given element is an extension of the longest element of a maximal proper standard parabolic subgroup. 
\end{remark}

\begin{remark}\label{rem:bexample}
We will apply the criteria in Theorem~\ref{thm:convexhulls} in the case $i=n$. The criterion is particularly useful in proving that an element $w$ is not an extension of~$\sw_n$, and often only a small amount of information is required. For example, suppose we know only the following pieces of information: $e_2,-e_2+\delta\notin\Phi(w)$, $e_1-e_2\in\Phi(w)$ but $e_1-e_2+\delta\notin\Phi(w)$, and $e_1+e_2\in\Phi(w)$ (precisely this configuration will arise later in the proof of Theorem~\ref{thm:determineB}). If there is a type $n$ vertex $x=x_1e_1+\cdots+x_ne_n$ in the interior of $\mathsf{conv}\{1,w\}$, then from the discussion in Section~\ref{sec:convex} we have $0<x_2<1$, $-1< x_1-x_2$, and $x_1+x_2< 1$. By Lemma~\ref{lem:typenvertices} we have $x_1,x_2\in\ZZ+\frac{1}{2}$, and so $x_2=\frac{1}{2}$ and $-\frac{1}{2}<x_1<\frac{1}{2}$, a contradiction. Thus $w$ is not an extension of $\sw_n$.
\end{remark}

\begin{cor}
Let $w\in\Wext$. We have $w\in\Delta_1$ if and only if there exists a type $n$ vertex in $\mathsf{conv}\{1,w\}^{\circ}$ and $\Phi(w)$ misses at least one parallelism class. 
\end{cor}

\begin{proof}
This follows from Shi's description~\cite[Proposition~5.9]{ShiBn} and Theorems~\ref{thm:Delta0} and~\ref{thm:convexhulls}.
\end{proof}

\subsection{The set $\sB$}

We now explicitly determine the set $\sB$ (see the introduction). Note that in the definition of $\sB$, the condition $\ba(s_i\sw_nv)<\ell(\sw_n)$ can be replaced with $s_i\sw_nv\notin\Delta_1$ (because $\ba(s_i\sw_nv)<\ell(\sw_n)$ is equivalent to $s_i\sw_nv\notin\Delta_0\cup\Delta_1$, and $s_i\sw_nv\notin\Delta_0$ is already forced by the condition $\sw_nv\notin\Delta_0$ and the fact that $\ell(s_i\sw_nv)<\ell(\sw_nv)$ for $0\leq i\leq n-1$). Thus we have
\begin{align}\label{eq:cleanB}
\sB=\{v\in \Waff\mid \ell(\sw_nv)=\ell(\sw_n)+\ell(v),\,\sw_n v\notin\Delta_0,\,s_i\sw_nv\notin\Delta_1\text{ for $0\leq i\leq n-1$}\}.
\end{align}
In \cite[Section~3]{ShiBn} Shi establishes a bijection with a certain subset of sign types, and hence computes the cardinality of $\sB$. 

\begin{thm}[{\cite[Theorem~3.9]{ShiBn}}]\label{thm:Bcount}
We have $|\sB|=2^{n-1}n!$
\end{thm}

\begin{lemma}\label{lem:square}
Let $x\in\Wext$, $s\in \Saff$, and $t\in \Saff\backslash\{s_n\}$. Suppose that $x,xs,tx\in\Delta_1$ and that $\ell(xs)=\ell(x)+1$. Then $txs\in\Delta_1$. 
\end{lemma}

\begin{proof}
Since $x,tx\in\Delta_1$ the facts listed in Section~\ref{sec:KL} give $x\sim_L tx$. Since right descent sets are constant on left cells we have $D_R(x)=D_R(tx)$, and since $\ell(xs)=\ell(x)+1$ we have $s\notin D_R(tx)$. Thus $\ell(txs)=\ell(tx)+1$. Since $tx\in\Delta_1$ we can write $tx=a\cdot \sw_n\cdot b$. Since $s\notin D_R(tx)$ we have $\ell(bs)=\ell(b)+1$, and so $txs=a \sw_n(bs)$ with $\ell(a\sw_nbs)=\ell(a)+\ell(\sw_n)+\ell(bs)$. Thus $txs$ is an extension of~$\sw_n$, and so it remains to prove that $txs\notin\Delta_0$. If $\ell(txs)=\ell(xs)-1$ then $xs=t\cdot(txs)$ with $\ell(t(txs))=\ell(txs)+1$, and so if $txs\in\Delta_0$ then $xs\in\Delta_0$, a contradiction. So we have $\ell(txs)=\ell(xs)+1$, and hence $\Phi(txs)=\{\alpha_t\}\cup t\Phi(xs)$. By Proposition~\ref{prop:conditionA} $\Phi(xs)$ meets all long root parallelism classes, and since $xs\notin\Delta_0$ it must miss at least one short root parallelism class. Since the linear part of $\alpha_t$ is long it follows that $\Phi(txs)$ also misses a short-root class, hence $txs\notin\Delta_0$ as required. 
\end{proof}

\begin{prop}\label{prop:orderideal}
The set $\sB$ is a lower order ideal in the right weak order. That is, $\sB$ is closed under taking prefixes. \end{prop}

\begin{proof}
We must show that if $u\in\sB$, $s\in \Saff$, and $u=u_1s$ with $\ell(u)=\ell(u_1)+1$, then $u_1\in \sB$. We verify the conditions defining $\sB$. Since $u\in\sB$ we have $\ell(\sw_n u)=\ell(\sw_n)+\ell(u)$, and so clearly $\ell(\sw_nu_1)=\ell(\sw_n)+\ell(u_1)$. Also, it is clear that $\sw_nu_1\notin\Delta_0$ (for if so, $\sw_n u\in\Delta_0$). Thus $\sw_nu_1\in\Delta_1$. It remains to prove that $s_i\sw_nu_1\notin\Delta_1$ for all $0\leq i\leq n-1$. Suppose, for a contradiction, that $s_i\sw_nu_1\in\Delta_1$ for some $0\leq i\leq n-1$. Let $x=\sw_nu_1$ and $t=s_i$. Lemma~\ref{lem:square} gives $s_i\sw_nu=txs\in\Delta_1$, contradicting the fact that $u\in\sB$. Hence the result. 
\end{proof}
%
%

Let $\rho=\omega_1+\cdots+\omega_n$ and define
\begin{align}\label{eq:b}
\sb=\fu^{-1}t_{\rho}\sw_0\sigma^k
\end{align}
where $k\in\{0,1\}$ is chosen to make $\sb\in\Waff$ (explicitly, $k=0$ if $n=0,3\pmod{4}$ and $k=1$ if $n=1,2\pmod{4}$). Using Lemma~\ref{lem:etaleform} we have
\begin{align}\label{eq:bstandard}
\sb=\begin{cases}
t_{\rho-n\omega_n}\sw_{0,n}\sigma^k&\text{if $n$ is even}\\
t_{\rho+\omega_1-n\omega_n}s_{e_1}\sw_{0,n}\sigma^k&\text{if $n$ is odd},
\end{cases}
\end{align} 
and in both cases the length formula~(\ref{eq:length}) gives
$$
\ell(\sb)=\frac{1}{3}n(n-1)(2n-1). 
$$

The following theorem gives an explicit description of the set~$\sB$. 

\begin{thm}\label{thm:determineB} We have $\sB=[1,\sb]\cup[1,\sb^{\sigma}]$. 
\end{thm}

\begin{proof}
We first show that $\sb\in\sB$. A routine inversion set calculation shows that $D_L(\sb)=\{s_n\}$, and hence $\ell(\sw_n\sb)=\ell(\sw_n)+\ell(\sb)$. Moreover 
$$
\sb^{-1}\star e_1=\begin{cases}
e_n&\text{if $n$ even}\\
-e_n+\delta&\text{if $n$ odd}
\end{cases}\qquad \sb^{-1}\star(-e_1+\delta)=\begin{cases}
-e_n+\delta&\text{if $n$ even}\\
e_n&\text{if $n$ odd}
\end{cases}
$$
and so $\Phi(\sw_n \sb)$ misses the parallelism class $[e_1]$, giving $\sw_n\sb\notin\Delta_0$ (see Theorem~\ref{thm:Delta0}). For $0\leq i\leq n-1$ we have $\ell(s_i\sw_n\sb)=\ell(\sw_n\sb)-1$ (because $\ell(\sw_n\sb)=\ell(\sw_n)+\ell(\sb)$). For $1\leq i\leq n-1$ the set $\Phi(s_i\sw_n\sb)$ misses the parallelism class $[\alpha_i]$, and hence $s_i\sw_n\sb\notin\Delta_1$ (by Proposition~\ref{prop:conditionA}, noting that $\alpha_i$ is long). For $i=0$ it turns out that $\Phi(s_0\sw_n\sb)$ meets all long root parallelism classes, and so Proposition~\ref{prop:conditionA} cannot be applied. Instead, note that $\Phi(s_0\sw_n\sb)$ misses $[e_2]$, and that $e_1-e_2\in\Phi(s_0\sw_n\sb)$ but $e_1-e_2+\delta\notin\Phi(s_0\sw_n\sb)$, and $e_1+e_2\in\Phi(s_0\sw_n\sb)$. This is precisely the configuration from Remark~\ref{rem:bexample}, and hence by Theorem~\ref{thm:convexhulls} we have $s_0\sw_n\sb\notin\Delta_1$. Thus $\sb\in\sB$. 

It is clear from~(\ref{eq:cleanB}) that $\sigma \sB\sigma=\sB$ (as two-sided cells are preserved by conjugation by $\sigma$), and hence $\sb^{\sigma}\in\sB$. Thus by Proposition~\ref{prop:orderideal} we have $[1,\sb]\cup[1,\sb^{\sigma}]\subseteq\sB$. To prove equality, it is sufficient by Theorem~\ref{thm:Bcount} to show that $\#([1,\sb]\cup[1,\sb^{\sigma}])=2^{n-1}n!$. We have
$$
\#([1,\sb]\cup[1,\sb^{\sigma}])=\#[1,\sb]+\#[1,\sb^{\sigma}]-\#([1,\sb]\cap[1,\sb^{\sigma}])=\#[1,\sb]+\#[1,\sb^{\sigma}]-\#([1,\sb\wedge \sb^{\sigma}]),
$$
and so the task is reduced to computing the meet $\sb\wedge \sb^{\sigma}$ and determining the cardinalities of $[1,\sb]$ and $[1,\sb\wedge \sb^{\sigma}]$ (as $\#[1,\sb^{\sigma}]=\#[1,\sb]$). 

We have $\sb=\fu^{-1}y\sigma^k$ where $y=t_{\rho}\sw_0$. Let $\mathsf{c}=\fu^{-1}y\fu\sigma^k$ (note that $\mathsf{c}\in\Waff$). We claim that $\mathsf{c}=\sb\wedge \sb^{\sigma}$. We have $\mathsf{c}^{-1}\sb=\sigma^k\fu^{-1}\sigma^k$, which has length $\frac{1}{2}n(n-1)$ (see Lemma~\ref{lem:etalebasics}). As noted above, we have $\ell(\sb)=\frac{1}{3}n(n-1)(2n-1)$. By Lemma~\ref{lem:etaleform} we have $\mathsf{c}=\fu^{-1}y\fu\sigma^k=t_{\gamma}\sw_0\sigma^k$ where $\gamma=\omega_1+\cdots+\omega_{n-1}-(n-2)\omega_n$, and~(\ref{eq:length}) gives $\ell(\mathsf{c})=\frac{1}{6}n(n-1)(4n-5)$. Hence $\ell(\mathsf{c}^{-1}\sb)=\ell(\sb)-\ell(\mathsf{c})$. Since $\mathsf{c}^{\sigma}=\mathsf{c}$ it follows that $\mathsf{c}\peq \sb$ and $\mathsf{c}\peq \sb^{\sigma}$. The right ascent set of $\mathsf{c}$ is $\{s_0,s_1,\ldots,s_{n-1}\}$. Since $\mathsf{c}^{-1}\sb=\sigma^k\fu^{-1}\sigma^k$ it follows that for $0\leq i\leq n-1$ we have $\mathsf{c}s_i\peq \sb$ if and only if $i=0$ for $n=0,1\pmod{4}$ or $i=1$ for $n=2,3\pmod{4}$. Let $j\in\{0,1\}$ be the unique index such that $\mathsf{c} s_j\peq \sb$. Since $\mathsf{c}^{\sigma}=\mathsf{c}$ we have $\mathsf{c} s_i\peq \sb^{\sigma}$ if and only if $\mathsf{c} s_{\sigma(i)}\peq \sb$, and therefore the unique cover of $\mathsf{c}$ lying below $\sb^{\sigma}$ is $\mathsf{c} s_{\sigma(j)}$. Since $\sigma(j)\neq j$, no cover $\mathsf{c} s_i$ of $\mathsf{c}$ lies below both $\sb$ and $\sb^{\sigma}$, and so $\mathsf{c}=\sb\wedge \sb^{\sigma}$. 

We claim that
\begin{align}\label{eq:volumes}
\#[1,\sb]=3\times 2^{n-3}n!\quad\text{and}\quad \#[1,\mathsf{c}]=2^{n-2}n!. 
\end{align}
In each case we make the count using~(\ref{eq:volume}). First consider $\mathsf{c}=t_{\gamma}\sw_0\sigma^k$, with $\gamma$ as above. Following the process described in Section~\ref{sec:convex}, the inequalities describing $\mathsf{conv}\{1,\mathsf{c}\}$ are $0\leq x_1+x_2\leq 1$ and $0\leq x_i-x_{i+1}\leq 1$ for all $1\leq i\leq n-1$. Thus setting $y_0=x_1+x_2$ and $y_i=x_i-x_{i+1}$ for $1\leq i\leq n-1$, the region becomes the unit cube $0\leq y_i\leq 1$, and the absolute value of the determinant of this transformation is $2$. Hence $\mathsf{Vol}(\mathsf{conv}\{1,\mathsf{c}\})=\frac{1}{2}$. 

Now consider $[1,\sb]$. Using~(\ref{eq:bstandard}), if $n$ is even then the inequalities describing $\mathsf{conv}\{1,\sb\}$ are $x_1\geq 0$, $x_1+x_2\leq 1$, and $0\leq x_i-x_{i+1}\leq 1$ for $1\leq i\leq n-1$. Setting $y_i=x_i-x_{i+1}$ for $1\leq i\leq n-1$ and $z=x_1+x_2$ this region becomes $0\leq y_i\leq 1$ and $-y_1\leq z\leq 1$, and the absolute value of the Jacobian of this transformation is $\frac{1}{2}$. Thus
$$
\mathsf{Vol}(\mathsf{conv}\{1,\sb\})=\frac{1}{2}\int_{[0,1]^{n-1}}\int_{-y_1}^11\,dzdy_1\cdots dy_{n-1}=\frac{3}{4}.
$$
A similar calculation also gives $\mathsf{Vol}(\mathsf{conv}\{1,\sb\})=\frac{3}{4}$ in the case that $n$ is odd (here the inequalities are $x_1\leq 1$, $0\leq x_1+x_2\leq 1$, $x_1-x_2\geq 0$, and $0\leq x_i-x_{i+1}\leq 1$ for $2\leq i\leq n-1$). The formulae~(\ref{eq:volumes}) now follow, and hence $\#([1,\sb]\cup[1,\sb^{\sigma}])=2^{n-1}n!$ as required. 
\end{proof}

Using the explicit description of $\sB$ it is now easy to check certain properties that will be useful later (and that appear to be otherwise difficult to deduce directly from the definition of~$\sB$).

\begin{cor}\label{cor:Binvset}
For $u\in \sB$ we have $u^{-1}\star\{e_1,-e_1+\delta\}=\{e_r,-e_r+\delta\}$ for some $1\leq r\leq n$. In particular, $\Phi(u)$ misses the parallelism class $[e_1]$. 
\end{cor}

\begin{proof}
We first show that $\alpha_u=u^{-1}\star e_1$ and $\beta_u=u^{-1}\star(-e_1+\delta)$ are both positive. Since $u\peq \sb$ or $u\peq \sb^{\sigma}$ we have $\Phi(u)\subseteq\Phi(\sb)\cup\Phi(\sb^{\sigma})$, and so it suffices to check that $b_1^{-1}\star e_1>0$ and $b_1^{-1}\star(-e_1+\delta)>0$ for $b_1\in\{\sb,\sb^{\sigma}\}$, and since $\Phi(\sb^{\sigma})=\sigma(\Phi(\sb))$ it is sufficient to consider $b_1=\sb$. From the proof of Theorem~\ref{thm:determineB} we have $\sb^{-1}\star e_1>0$ and $\sb^{-1}\star(-e_1+\delta)>0$, and hence $\alpha_u,\beta_u>0$. Thus $\alpha_u,\beta_u$ are positive roots in the same (short root) parallelism class, and since $\alpha_u+\beta_u=u^{-1}\star \delta=\delta$ we have $\{\alpha_u,\beta_u\}=\{e_r,-e_r+\delta\}$ for some $1\leq r\leq n$. 
\end{proof}

\begin{cor}\label{cor:additive} If $u,v\in\sB$ and $\lambda\in P_n^+$ then 
$$
\ell(u^{-1}\sm_{\lambda}v)=\ell(u)+\ell(\sm_{\lambda})+\ell(v).
$$
\end{cor}

\begin{proof}
Since each element of $\sB$ is a prefix of either $\sb$ or $\sb^{\sigma}$, it is sufficient to prove that $\ell(b_1^{-1}\sm_{\lambda}b_2)=\ell(\sm_{\lambda})+2\ell(\sb)$ for all choices $b_1,b_2\in\{\sb,\sb^{\sigma}\}$. By Proposition~\ref{prop:dominance} we have $\sm_{\lambda}^{\sigma}=\sm_{\lambda}$ and $\sm_{\lambda}^{-1}=\sm_{\lambda}$ and so it is sufficient to check the two cases $b_1=b_2=\sb$, and $b_1=\sb$ with $b_2=\sb^{\sigma}$. Using the definition of $\fu$, and Lemma~\ref{lem:etaleform}, and the formula for $\sm_{\la}$ in Theorem~\ref{thm:mlambda}, one can write $b_1^{-1}\sm_{\la}b_2$ in the form $t_{\mu}v$ with $\mu\in P$ and $v\in \Wfin$, and then the length can be calculated using~(\ref{eq:length}). After performing the calculations, one arrives at 
$$
\ell(b_1^{-1}\sm_{\la}b_2)=\frac{1}{3}n(n-1)(4n+1)+\sum_{i=1}^{n-1}(2n-2i+1)\la_i
$$
for both choices $(b_1,b_2)=(\sb,\sb),(\sb,\sb^{\sigma})$, and the result follows using $\ell(\sb)=\frac{1}{3}n(n-1)(2n-1)$ and the formula for $\ell(\sm_{\la})$ in Corollary~\ref{cor:mlength}.
\end{proof}

\begin{remark}\label{rem:Shipaper} In \cite[(2.9.8)]{ShiBn} it is stated that if $x,y\in\Waff$ with $D_R(x),D_L(y)\subseteq \{s_n\}$ then $\ell(x\sw_n y)=\ell(x)+\ell(\sw_n)+\ell(y)$. This statement could be used to give a much quicker proof of Corollary~\ref{cor:additive}, however we note that the statement \cite[(2.9.8)]{ShiBn} is false (with $x=y=323023$ providing a counterexample in $\tilde{\sB}_3$, as $\ell(x)=\ell(\sw_3)=6$, while $\ell(x\sw_3x)=16<18$). In \cite{ShiBn} the statement (2.9.8) is used a few times, however only in the situation that $x^{-1},y\in\sB$. We note that Corollary~\ref{cor:additive} is independent of \cite[Section~2]{ShiBn}, and hence may be used to validate the applications of (2.9.8) in \cite{ShiBn} without introducing any circularity into the arguments. 
\end{remark}

\subsection{The intersections $\Ga_u\cap\Ga_v^{-1}$}

The intersection of a left and right cell plays a central role in Kazhdan-Lusztig theory. In this section we compute $\Ga_u\cap\Ga_v^{-1}$ in terms of $\Ga\cap\Ga^{-1}$, where $\Ga=\Ga_1$. Let $\mathsf{E}=\{e_1,-e_1+\delta\}$. By Corollary~\ref{cor:Binvset} there is a function $r:\sB\to \{1,2,\ldots,n\}$ given by 
$$u^{-1}\star\mathsf{E}=\{e_{r(u)},-e_{r(u)}+\delta\}.
$$ 

\begin{lemma}\label{lem:generalisedabsence}
Let $u\in\sB$ and $w\in\Ga_u$. Then $\Phi(w)$ misses the parallelism class $[e_{r(u)}]$. 
\end{lemma}

\begin{proof}
Since $w\in\Ga_u$ we have $w=u^{-1}\sw_ny$ with $\ell(w)=\ell(u)+\ell(\sw_n)+\ell(y)$. The inversion set of $u^{-1}\sw_n$ contains $u^{-1}\Phi_n^+$ (these roots are positive because $\ell(u^{-1}\sw_n)=\ell(u)+\ell(\sw_n)$ gives $\Phi(u^{-1}\sw_n)=\Phi(u^{-1})\sqcup u^{-1}\Phi(\sw_n)$). Thus the roots $u^{-1}\star(e_i-e_j)$ and $u^{-1}\star(-e_i-e_j+\delta)$ are positive and are contained in $\Phi(w)$, for $1\leq i<j\leq n$. By Corollary~\ref{cor:Binvset} the roots $u^{-1}\star e_1$ and $u^{-1}\star(-e_1+\delta)$ are positive. Let $\beta_1=u^{-1}\star e_1$, $\beta_1'=u^{-1}\star(-e_1+\delta)$, $\gamma_j=u^{-1}\star(e_1-e_j)$, and $\gamma_j'=u^{-1}\star(-e_1-e_j+\delta)$ for $2\leq j\leq n$, we have that $\beta_1,\beta_1',\gamma_j,\gamma_j'$ are all positive roots, and since 
$$
u^{-1}\star(-e_j+\delta)=\beta_1+\gamma_j'=\beta_1'+\gamma_j
$$
it follows that if $\{\beta_1,\beta_1'\}\cap\Phi(w)\neq\emptyset$ then $u^{-1}\star(-e_j+\delta)\in\Phi(w)$ for all $2\leq j\leq n$, and hence $\Phi(w)$ meets every parallelism class. Thus $w\in\Delta_0$ by Theorem~\ref{thm:Delta0}, a contradiction. 
\end{proof}

\begin{prop}\label{prop:transfer}
Suppose that $u\in\sB$, $s\in \Saff$, $\ell(us)=\ell(u)+1$, and $us\in\sB$. Then $\Gamma_{us}=s\Gamma_u$. 
\end{prop}

\begin{proof}
Immediately from the definitions, we have 
$$
\Gamma_{us}=\{sx\mid x\in \Gamma_u,\,\ell(sx)=\ell(x)+1,\,sx\notin\Delta_0\}.
$$
We will show that the conditions $\ell(sx)=\ell(x)+1$ and $sx\notin\Delta_0$ are automatic under the given assumptions. 

Since $\Gamma_u$ is a right cell, and left descent sets are constant on right cells, we know $D_L(x)=D_L(u^{-1}\sw_n)$ for all $x\in \Gamma_u$. Now, since $us\in \sB$ and we have $\ell((us)^{-1}\sw_n)=\ell(us)+\ell(\sw_n)=\ell(u)+1+\ell(\sw_n)=\ell(u^{-1}\sw_n)+1$ and hence $s\notin D_L(u^{-1}\sw_n)$. Thus $s\notin D_L(x)$ for all $x\in\Gamma_u$, and hence $\ell(sx)=\ell(x)+1$ for all $x\in\Gamma_u$. 

To show that $sx\notin\Delta_0$ for all $x\in\Gamma_u$,  write $\mathsf{E}_u=u^{-1}\star \mathsf{E}=\{e_{r(u)},-e_{r(u)}+\delta\}$ and $\mathsf{E}_{us}=(us)^{-1}\star \mathsf{E}=s\star \mathsf{E}_u$. For $x\in \Gamma_u$ we have $\mathsf{E}_u\cap \Phi(x)=\emptyset$ (by Lemma~\ref{lem:generalisedabsence}). Also, $\mathsf{E}_{us}\subseteq\Phi^+$ because $us\in\sB$. Let $\alpha\in \mathsf{E}$ and write $\gamma=u^{-1}\star \alpha$ and $\gamma'=(us)^{-1}\star \alpha=s\star \gamma$. Then $(sx)^{-1}\star\gamma'=x^{-1}\star\gamma>0$ and since $\gamma'>0$ we have $\gamma'\notin\Phi(sx)$. Therefore $\mathsf{E}_{us}\cap\Phi(sx)=\emptyset$, and so $sx\notin\Delta_0$ (by Theorem~\ref{thm:Delta0}). Hence the result. 
\end{proof}

\begin{thm}\label{thm:intersections}
For $u,v\in \sB$ we have 
\begin{compactenum}[$(1)$]
\item $\Ga_u=u^{-1}\Ga$, and
\item $\Ga_u\cap\Ga_v^{-1}=u^{-1}(\Ga\cap\Ga^{-1})v$. 
\end{compactenum}
\end{thm}

\begin{proof}
(1) is immediate from Proposition~\ref{prop:transfer} and the fact that $\sB$ is closed under taking prefixes (Proposition~\ref{prop:orderideal}). For (2), suppose that $w\in\Ga_u\cap\Ga_v^{-1}$. Write $u=u_1s$ with $\ell(u_1s)=\ell(u_1)+1$. By Proposition~\ref{prop:orderideal} we have $u_1\in \sB$ and by Proposition~\ref{prop:transfer} we have $\Gamma_u=s\Gamma_{u_1}$. For every $x\in\Ga_{u_1}$, both $x$ and $sx$ lie in $\Delta_1$, and hence $sx\sim_L x$ (see Section~\ref{sec:KL}). Thus $x\in\Gamma_v^{-1}$ if and only if $sx\in\Gamma_v^{-1}$, and so 
$$
\Gamma_u\cap\Gamma_v^{-1}=s(\Gamma_{u_1}\cap\Gamma_v^{-1}). 
$$
Continuing gives $\Gamma_u\cap\Gamma_v^{-1}=u^{-1}(\Gamma\cap\Gamma_v^{-1})$. Similarly, we have $\Gamma\cap\Gamma_v^{-1}=(\Gamma\cap\Gamma^{-1})v$, and the result follows. 
\end{proof}

\subsection{The intersection $\Ga\cap\Ga^{-1}$}

We now compute the intersection $\Ga\cap\Ga^{-1}$.

\begin{lemma}\label{lem:onedirection}
For $\lambda\in P_n^+$ the inversion set $\Phi(\sm_{\la})$ misses the parallelism class $[e_1]$. 
\end{lemma}

\begin{proof}
Using the formula for $\sm_{\la}$ in Theorem~\ref{thm:mlambda} we compute $\sm_{\la}^{-1}\star e_1=-e_1+\delta$ or $e_1$, and $\sm_{\la}^{-1}\star(-e_1+\delta)=e_1$ or $-e_1+\delta$ (depending on $n+\la_n$ even or odd), hence the result. 
\end{proof}

\begin{thm}\label{thm:gammaintersection}
We have $\Ga\cap\Ga^{-1}=\{\sm_{\lambda}\mid \lambda\in P_n^+\}$. 
\end{thm}

\begin{proof}
For $\la\in P_n^+$ we have $\sm_{\la}=\sw_nx_{\la}$ where $x_{\la}=\fu^{-1}\tau_{\la}\fu$ with $\ell(\sm_{\la})=\ell(\sw_n)+\ell(x_{\la})$. Moreover, by Lemma~\ref{lem:onedirection} and Theorem~\ref{thm:Delta0} we have $\sm_{\la}\notin\Delta_0$. Thus $\sm_{\la}\in \Ga$ by Theorem~\ref{thm:Shi}. By Proposition~\ref{prop:dominance} we have that $\sm_{\la}$ is an involution, and so $\{\sm_{\lambda}\mid \lambda\in P_n^+\}\subseteq\Ga\cap\Ga^{-1}$. 

Let $w\in\Ga\cap\Ga^{-1}$. Since left descent sets are constant on right cells, and right descent sets are constant on left cells, and since $\sw_n\in\Ga\cap\Ga^{-1}$, we have $D_L(w)=D_R(w)=\{s_i\mid 0\leq i\leq n-1\}$. Hence $w$ is of maximal length in its $W_n$-double coset. Moreover, by Lemma~\ref{lem:generalisedabsence} $\Phi(w)$ misses the parallelism class $[e_1]$. Thus the alcove $wA_0$ lies between the hyperplanes $H_{e_1}$ and $H_{-e_1+\delta}$, giving
$$
wA_0\subseteq \{\lambda\in\RR^n\mid  0\leq \langle\lambda,e_1\rangle\leq 1\}=\sw_{0,n}\cA_n
$$
(to obtain the final equality, note that $\sw_{0,n}e_n=e_1$). Thus $\sw_{0,n}wA_0\subseteq\cA_n$, and since $W_n$ is a fundamental domain for the action of $\sT_n$ on $\mathbb{X}_n$ (see Lemma~\ref{lem:Dnfunddom}) there exists $\lambda\in P_n$ and $v\in W_n$ such that $\sw_{0,n}w=\tau_{\lambda}v$. Thus $W_nwW_n=W_n\tau_{\lambda}W_n$, and so $w=\sm_{\lambda}$ for some $\lambda\in P_n$, and by Proposition~\ref{prop:dominance} we can take $\lambda\in P_n^+$. 
\end{proof}
%
%

\begin{cor}\label{cor:intersections}
For $u,v\in \sB$ we have 
$$
\Ga_u\cap\Ga_v^{-1}=\{u^{-1}\sm_{\lambda}v\mid \lambda\in P_n^+\},
$$ 
and for $\la\in P_n^+$ we have $\ell(u^{-1}\sm_{\lambda}v)=\ell(u)+\ell(\sm_{\la})+\ell(v)$. 
\end{cor}

\begin{proof}
This is immediate from Theorems~\ref{thm:intersections}, \ref{thm:gammaintersection}, and Corollary~\ref{cor:additive}. 
\end{proof}

\subsection{Cell factorisation}

The previous results imply a strong form of cell factorisation for the cell~$\Delta_1$ (compare with the formulae in \cite[Sections~4.2,4.3]{GP:19}). Recall that we write $\tau_{\la}'=\fu^{-1}\tau_{\la}\fu$ for $\la\in P_n$. In particular, 
$$
T^+=\{\tau_{\la}'\mid \la\in P_n^+\}
$$ is a commutative semigroup with $\ell(\tau\tau')=\ell(\tau)+\ell(\tau')$ for all $\tau,\tau'\in T^+$ (see Proposition~\ref{prop:taulength}). Since the semigroup $\{\tau_{\la}\mid \la\in P_n^+\}$ is generated by the elements $\tau_{\omega_1},\tau_{\omega_2},\ldots,\tau_{\omega_{n-1}}$ and $\tau_{e_n}$, the semigroup $T^+$ is generated by the $\fu$-conjugates of these elements. Thus by the calculation in Proposition~\ref{prop:taulength} we see that $T^+$ is generated by the elements $\tau_1',\ldots,\tau_{n-1}'$ and $\sigma$, where $\tau_i'=t_{\omega_{n-i}-\omega_n}$ for $1\leq i\leq n-1$ (note that $\fu^{-1}\tau_{e_n}\fu=t_{e_1}s_{e_1}=\sigma$). Therefore, each element $\tau\in T^+$ can be written as 
$$
\tau=\tau_1'^{k_1}\cdots\tau_{n-1}'^{k_{n-1}}\sigma^k\quad\text{with $k_1,\ldots,k_{n-1}\in\mathbb{Z}_{\geq 0}$ and $k\in\{0,1\}$.}
$$
Moreover, $\ell(\tau)=k_1\ell(\tau_1')+\cdots+k_{n-1}\ell(\tau_{n-1}')$. 

%
%

\begin{cor}[Cell factorisation]\label{cor:main3}
The map $\psi:\sB^2\times T^+\to  \Delta_1$ with $\psi(u,v,\tau)=u^{-1}\sw_n\tau v$ is a bijection, and $\ell(u^{-1}\sw_n\tau v)=\ell(u)+\ell(\sw_n)+\ell(\tau)+\ell(v)$ for all $u,v\in\sB$ and $\tau\in T^+$. Moreover, if $w,w'\in\Delta_1$ with $\psi^{-1}(w)=(u,v,\tau)$ and $\psi^{-1}(w')=(u',v',\tau')$ then $w\sim_L w'$ if and only if $v=v'$ and $w\sim_R w'$ if and only if $u=u'$. 
\end{cor}

\begin{proof} The map $\psi:\sB^2\times T^+\to \Delta_1$ with $\psi(u,v,\tau)=u^{-1}\sw_n\tau v$ is well defined (that is, $\psi(u,v,\tau)\in\Delta_1$) and surjective by Corollary~\ref{cor:intersections} and Theorem~\ref{thm:mlambda}. Moreover, if $w=\psi(u,v,\tau)=\psi(u',v',\tau')$ then Corollary~\ref{cor:intersections} gives $w\in \Ga_u\cap \Ga_{u'}$ and $w\in\Ga_v^{-1}\cap\Ga_{v'}^{-1}$, and so $u=u'$ and $v=v'$ (by Theorem~\ref{thm:Shi}). Thus $\tau=\tau'$, and so $\psi$ is bijective. 
\end{proof}

The proof of Theorem~\ref{thm:main1} is now complete. 

\subsection{The fundamental domain $\sF=\fu\sB$}\label{sec:funddom}

To complete the proof of Theorem~\ref{thm:main2} we must show that $\sF=\fu\sB$ is a fundamental domain for the action of $\sT_n$ on $\mathbb{X}_n$. 


\begin{proof}[Proof of Theorem~\ref{thm:main2}(3)]
By Lemma~\ref{lem:Dnfunddom} we have that $W_n$ is a fundamental domain for the action of $\sT_n$ on $\mathbb{X}_n$. Moreover, by Theorem~\ref{thm:Bcount} we have $|\sF|=2^{n-1}n!=|W_n|$, and so it is sufficient to prove that every orbit $\sT_nz$, for $z\in W_n$, contains an element of $\sF=\fu\sB$.

Fix $z\in W_n$. Let $\xi=\frac{1}{2}(e_1+\cdots+e_n)$ be the type $n$ vertex of the alcove $A_0$, and let $\mathcal{C}$ be the cone based at $\xi$ containing $A_0$. Thus
$$
\mathcal{C}=\{x\in\mathbb{R}^n\mid x_1\geq x_2\geq \cdots \geq x_n,\,x_1+x_2\leq 1\}. 
$$
For $y\in\Wext$ we have $yA_0\subseteq \mathcal{C}$ if and only if $\ell(\sw_n y)=\ell(\sw_n)+\ell(y)$. To see this, by~(\ref{eq:add}) we have $\ell(\sw_ny)=\ell(\sw_n)+\ell(y)$ if and only if $D_L(y)\subseteq\{s_n\}$, if and only if $\alpha_i\notin\Phi(y)$ for $0\leq i\leq n-1$, if and only if $yA_0$ is contained in $H_{\alpha_i}^+$ for each $0\leq i\leq n-1$ (where $H_{\alpha_i}^+$ denotes the half space bounded by $H_{\alpha_i}$ containing $A_0$). For $1\leq i\leq n-1$ we have $H_{\alpha_i}^+=\{x\in V\mid x_i\geq x_{i+1}\}$, and for $i=0$ we have $H_{\alpha_0}^+=\{x\in V\mid x_1+x_2\leq 1\}$, hence the claim. 

Consider the vector $\eta=-e_2-2e_3-\cdots-(n-1)e_n$ lying in the interior of the cone $\mathcal{C}$. Write $\fu=t_{\gamma}u$ with $u\in \Wfin$ (by Lemma~\ref{lem:etaleform} we have $u=\sw_{0,n}\sw_0$ if $n$ is even, and $u=\sw_{0,n}\sw_{0,1}$ if $n$ is odd). Write $\nu=u(\eta)$. Since $u(e_1)=\pm e_n$ the last coordinate of $\nu$ is $0$, and so $\nu\in P_n$. Thus $\tau_{N\nu}=t_{N\nu}$ for all $N\in\ZZ$. Since $\fu^{-1}t_{N\nu}\fu=t_{N\eta}$ we have $\fu^{-1}\tau_{N\nu}z=t_{N\eta}\fu^{-1}z$. Since $\eta$ is in the interior of $\mathcal{C}$ there is $N>0$ sufficiently large such that $t_{N\eta}\fu^{-1}z A_0\subseteq \mathcal{C}$. Fix such a choice of $N$, write $\nu_0=N\nu\in P_n$, and let $y=\fu^{-1}\tau_{\nu_0}z$. Since $y\cA_0\subseteq\mathcal{C}$ we have $\ell(\sw_ny)=\ell(\sw_n)+\ell(y)$. 

We claim that $\sw_ny\notin\Delta_0$. Since $\fu=\sw_{0,n}\sw_n$ we have $\sw_ny=\sw_{0,n}\tau_{\nu_0}z$. Because $z\in W_n$ the alcove $z A_0$ lies in $\cA_n$, and since $\sT_n$ acts on $\cA_n$ we have $\tau_{\nu_0}zA_0\subseteq\cA_n$. But $\sw_{0,n}(e_n)=e_1$, hence $\sw_nyA_0\subseteq\{x\in\mathbb{R}^n\mid 0\leq \langle x,e_1\rangle\leq 1\}$. Thus the inversion set of $\sw_n y$ misses the parallelism class $[e_1]$, and so $\sw_ny\notin\Delta_0$ by Theorem~\ref{thm:Delta0}. 

Thus from Shi's description of $\Ga$ in Theorem~\ref{thm:Shi} we have $\sw_ny\in \Ga$. Since $\Gamma=\bigsqcup_{v\in \sB}(\Ga\cap\Ga_v^{-1})$, Corollary~\ref{cor:intersections} gives $y=(\fu^{-1}\tau_{\la}\fu)v$ for some $\lambda\in P_n^+$ and $v\in\sB$. Thus
$$
\fu^{-1}\tau_{\nu_0}z=y=\fu^{-1}\tau_{\la}\fu v,
$$
giving $\tau_{\la}^{-1}\tau_{\nu_0}z=\fu v\in \fu\sB=\sF$, and hence the result (as $\tau_{\la}^{-1}\tau_{\nu_0}\in \sT_n$). 
\end{proof}

\subsection{Applications}\label{sec:applications1}

Before moving on to our study of the asymptotic algebra, we provide two applications of our description of $\Delta_1$ given in Theorems~\ref{thm:main1} and~\ref{thm:main2}. 

Firstly, we can give a complete classification of involutions in the cell~$\Delta_1$. 

\begin{cor}\label{cor:involutions}
Let $w\in\Delta_1$. Then $w^2=1$ if and only if $w=u^{-1}\sm_{\la}u$ for some $u\in\sB$ and $\la\in P_n^+$. In particular, each diagonal intersection $\Ga_u\cap\Ga_u^{-1}$ consists entirely of involutions, and each off-diagonal intersection $\Ga_u\cap\Ga_v^{-1}$ with $u\neq v$ contains no involution. 
\end{cor}

\begin{proof}
If $w\in\Delta_1$ then by Theorem~\ref{thm:main1} we have $w=u^{-1}\sm_{\la}v$ for some $u,v\in\sB$ and $\la\in P_n^+$, and by Proposition~\ref{prop:dominance} we have $\sm_{\la}^{-1}=\sm_{\la}$, and so $w^{-1}=v^{-1}\sm_{\la}u$. By Corollary~\ref{cor:main3} each element of $\Delta_1$ has a unique expression in the form $u^{-1}\sm_{\la}v$, and so if $w=w^{-1}$ we have $u=v$ as required. By the description of $\Ga_u\cap\Ga_v^{-1}$ in Theorem~\ref{thm:main1} it follows that every element of $\Ga_u\cap\Ga_u^{-1}$ is an involution, and no element of $\Ga_u\cap\Ga_v^{-1}$ is an involution if $u\neq v$. 
\end{proof}

Secondly, we can explicitly compute the rational length generating functions for $\Delta_1$, each right cell $\Ga_u$, and the set of involutions in~$\Delta_1$. 

\begin{cor}\label{cor:generating}
Let $\sB(t)=\sum_{u\in\sB}t^{\ell(u)}$. We have 
$$
\sum_{w\in\Delta_1}t^{\ell(w)}=\frac{2t^{n(n-1)}\sB(t)^2}{\prod_{i=1}^{n-1}(1-t^{i(2n-i)})}\quad \text{and}\quad \sum_{w\in \Ga_u}t^{\ell(w)}=\frac{2t^{n(n-1)+\ell(u)}\sB(t)}{\prod_{i=1}^{n-1}(1-t^{i(2n-i)})}. 
$$
Moreover, we have
$$
\sum_{w\in\Delta_1,\,w^2=1}t^{\ell(w)}=\frac{2t^{n(n-1)}\sB(t^2)}{\prod_{i=1}^{n-1}(1-t^{i(2n-i)})}. 
$$
\end{cor}

\begin{proof}
These are straightforward consequences of the description of $\Delta_1$ and $\Ga_u$ given in Theorem~\ref{thm:main1}, the uniqueness of expressions from Corollary~\ref{cor:main3}, and the explicit length formula Corollary~\ref{cor:mlength} (using also Corollary~\ref{cor:involutions}). 
\end{proof}

We note that it appears that there is no simple product formula for the polynomial~$\sB(t)$. For example, for $n=2,3$ we have $\sB(t)=1+t+2t^2$ and $\sB(t)=1+t+t^2+3t^3+3t^4+3t^5+3t^6+3t^7+2t^8+2t^9+2t^{10}$. Generally we have $\sB(1)=2^{n-1}n!$ and $\deg\sB(t)=\ell(\sb)=\frac{1}{3}n(n-1)(2n-1)$.

\section{A factorisation formula in $\cJ_{\Delta_1}$}\label{sec:Jfact}

We now turn our attention to the asymptotic algebra $\cJ_{\Delta_1}$. In this section we lift the cell factorisation from Corollary~\ref{cor:main3} to a factorisation theorem in $\cJ_{\Delta_1}$, proving Theorem~\ref{thm:main4}. We begin with background sections on the extended affine Hecke algebra and further Kazhdan-Lusztig theory.

\subsection{The affine Hecke algebra}

Let $\sq$ be an invertible indeterminate, and let $\sR=\ZZ[\sq,\sq^{-1}]$. 

\begin{defn}
The (equal parameter) extended affine Hecke algebra of type $\tilde{\sB}_n$ is the $\sR$-algebra $\Hext$ with basis $\{T_w\mid w\in\Wext\}$ and multiplication given by (for $w,v\in\Wext$ and $s\in\Saff$)
\begin{align*}
T_wT_v&=T_{wv}&&\text{if $\ell(wv)=\ell(w)+\ell(v)$}\\
T_wT_s&=T_{ws}+(\sq-\sq^{-1})T_w&&\text{if $\ell(ws)=\ell(w)-1$}.
\end{align*}
\end{defn}

The algebra $\Hext$ admits another basis $\{X_v\mid v\in\Wext\}$ defined as follows. For each hyperplane $H_{\alpha+k\delta}$ let $H_{\alpha+k\delta}^+$ denote the closed half space bounded by $H_{\alpha+k\delta}$ containing points $N(\omega_1+\cdots+\omega_n)$ for all sufficiently large $N>0$, and let $H_{\alpha+k\delta}^-$ denote the other closed half space bounded by $H_{\alpha+k\delta}$. Let $v=s_{i_1}\cdots s_{i_r}\sigma^k$ be any expression (not necessarily reduced). Let $v_0=1$ and $v_p=s_{i_1}\cdots s_{i_p}$ for $1\leq p\leq r$. Let $A_p=v_pA_0$, and let $\epsilon_1,\ldots,\epsilon_r\in\{-1,1\}$ be the sequence defined by $\epsilon_p=1$ (respectively $\epsilon_p=-1$) if $A_{p}$ is on the positive (respectively negative) side of the hyperplane separating $A_{p-1}$ from $A_p$. The element 
$$
X_v=T_{s_{i_1}}^{\epsilon_1}\cdots T_{s_{i_r}}^{\epsilon_r}T_{\sigma^k}
$$
does not depend on the particular expression for $v$ chosen, and $\{X_v\mid v\in \Wext\}$ is a basis of~$\Hext$. For $\la\in P$ let 
$
X^{\la}=X_{t_{\la}}. 
$
Then $X^{\la}X^{\mu}=X^{\la+\mu}=X^{\mu}X^{\la}$ for all $\la,\mu\in P$. We have $T_{s_0}=X^{\varphi^{\vee}}T_{s_{\varphi}}^{-1}$. 

The \textit{Bernstein relation} is 
$$
T_{s_i}X^{\lambda}=X^{s_i\lambda}T_{s_i}+(\sq-\sq^{-1})\frac{X^{\lambda}-X^{s_i\lambda}}{1-X^{-\alpha_i^{\vee}}}\quad\text{for $1\leq i\leq n$ and $\lambda\in P$}.
$$

\subsection{Further Kazhdan-Lusztig theory}\label{sec:KL2}

The \textit{bar involution} $\overline{\sq}=\sq^{-1}$ on $\sR$ extends to an involution on $\Hext$ by setting 
$$
\overline{\sum_{w\in\Wext}a_wT_w}=\sum_{w\in\Wext}\overline{a_w}T_{w^{-1}}^{-1}. 
$$
In \cite{KL:79} Kazhdan and Lusztig proved that for each $w\in \Wext$ there is a unique element $C_w\in\Hext$ such that 
$$
\overline{C_w}=C_w\quad\text{and}\quad C_w=T_w+\sum_{v<w}P_{v,w}T_v\quad\text{with $P_{v,w}\in\sq^{-1}\ZZ[\sq^{-1}]$},
$$
and by triangularity $\{C_w\mid w\in\Wext\}$ is a basis of $\Hext$ (the \textit{Kazhdan-Lusztig basis}). For $x,y\in\Wext$ write 
$$
C_xC_y=\sum_{z\in\Wext}h_{x,y,z}C_z
$$
where $h_{x,y,z}\in\sR$ (necessarily $\overline{h_{x,y,z}}=h_{x,y,z}$). 

Let $\deg:\sR\to \ZZ\cup\{-\infty\}$ denote degree in $\sq$ (for example, $\deg(\sq^{-3})=-3$ and $\deg(0)=-\infty$). Since $h_{x,y,z}$ is bar-invariant we have $\deg h_{x,y,z}\geq 0$ whenever $h_{x,y,z}\neq 0$. Lusztig's $\ba$-function is the function $\ba:\Wext\to\ZZ_{\geq 0}$ defined by 
$$
\ba(z)=\max\{\deg h_{x,y,z}\mid x,y\in\Wext\}. 
$$
It is known that $\ba(z)\leq \ell(\sw_0)$ for all $z\in\Wext$. By \cite{ShiBn} we have
$$
\ba(\Delta_1)=n(n-1). 
$$

Following Lusztig~\cite{Lus:03}, define $\gamma_{x,y,z^{-1}}\in\ZZ$ by 
$$
h_{x,y,z}=\gamma_{x,y,z^{-1}}\sq^{\ba(z)}+(\text{terms of degree $<\ba(z)$}). 
$$
Let $\cD=\{z\in \Wext\mid \ba(z)=-\deg P_{1,z}\}$ (the set of \textit{distinguished involutions}). 

In \cite{Lus:03} Lusztig gave a series of statements (known as $\mathsf{P1}$--$\mathsf{P15}$) that capture the essential properties of cells, Lusztig's $\ba$-function, the $\gamma$-coefficients, and the distinguished involutions. In the equal parameter setting of this paper, all properties $\mathsf{P1}$--$\mathsf{P15}$ are known to hold (see \cite[Section 15]{Lus:03}). We record some of these properties for further use. 
\begin{compactenum}
\item[$\mathsf{P2}$.] If $d\in\cD$ and $x,y\in \Wext$ satisfy $\gamma_{x,y,d}\neq 0$ then $y=x^{-1}$.
\item[$\mathsf{P3}$.] If $y\in\Wext$ then there exists a unique $d\in\cD$ such that $\gamma_{y,y^{-1},d}\neq 0$.
\item[$\mathsf{P4}$.] If $z'\leq _{LR}z$ then $\ba(z')\geq \ba(z)$. In particular, $\ba$ is constant on two-sided cells. 
\item[$\mathsf{P5}$.] If $d\in\cD$, $y\in\Wext$, and $\gamma_{y,y^{-1},d}\neq 0$, then $\gamma_{y,y^{-1},d}=1$.
\item[$\mathsf{P6}$.] If $d\in\cD$ then $d^2=1$. 
\item[$\mathsf{P7}$.] For $x,y,z\in \Wext$ we have $\gamma_{x,y,z}=\gamma_{y,z,x}$. 
\item[$\mathsf{P8}$.] If $\gamma_{x,y,z^{-1}}\neq 0$ then $z\sim_R x$, $z\sim_L y$, and $y\sim_R x^{-1}$. 
\item[$\mathsf{P11}$.] If $z'\leq_{LR}z$ and $\ba(z')=\ba(z)$ then $z'\sim_{LR}z$. 
\item[$\mathsf{P13}$.] Each right cell $\Ga$ contains a unique element $d\in\cD$, and $\gamma_{x,x^{-1},d}\neq 0$ for all $x\in\Ga$. 
\end{compactenum}


The \textit{asymptotic algebra} (also called the \textit{based ring}) is the $\ZZ$-algebra $\cJ$ with basis $\{\st_w\mid w\in\Wext\}$ and multiplication given by 
$$
\st_x\st_y=\sum_{z\in\Wext}\gamma_{x,y,z^{-1}}\st_z,
$$
and by $\mathsf{P1}$--$\mathsf{P15}$ it is known that $\cJ$ is associative. For each left or right cell $\Ga$ the submodule $\cJ_{\Ga}$ spanned by $\{\st_w\mid w\in\Ga\}$ is a subring of~$\cJ$. Let 
$$
\cJ_{\Delta_1}=\mathsf{span}_{\ZZ}\{\st_w\mid w\in\Delta_1\}. 
$$

\begin{prop}[{\cite{ShiBn}}]
For $u\in\sB$, the distinguished involution of $\Ga_u$ is $\sd(u)=u^{-1}\sw_nu$.
\end{prop}

The \textit{canonical trace} on $\Hext$ is the $\sR$-linear function $\Tr:\Hext\to \sR$ with $\Tr(T_w)=\delta_{w,e}$ (see \cite{GP:00}). Define an anti-involution on $\Hext$ by 
$$\left(\sum a_wT_w\right)^*=\sum a_w T_{w^{-1}}.$$ 
We have $\Tr(T_uT_v^*)=\delta_{u,v}$ and hence $\Tr(h_1h_2)=\Tr(h_2h_1)$ for $h_1,h_2\in\Hext$. Thus $\langle h_1,h_2\rangle=\Tr(h_1h_2^*)$ defines a nondegenerate symmetric bilinear form on $\Hext$, with $\langle T_u,T_v\rangle=\delta_{u,v}$. 

Lusztig has also defined a canonical trace on the asymptotic algebra $\cJ$ (see \cite[Section 20.1(b)]{Lus:03} and \cite[Theorem 5.2]{CGLP:24}). Define a $\ZZ$-linear map $\Tr^{\infty}:\cJ\to\ZZ$ and a bilinear form $\langle\cdot,\cdot\rangle^{\infty}:\cJ\times\cJ\to \ZZ$ by 
$$
\Tr^{\infty}\bigg(\sum_{w\in\Wext} a_w\st_w\bigg)=\sum_{d\in\cD}a_d\quad\text{and}\quad \langle A,B\rangle^{\infty}=\Tr^{\infty}(AB^*),
$$
where $\left(\sum a_w\st_w\right)^*=\sum a_w\st_{w^{-1}}$. Then $\Tr^{\infty}$ is a trace functional on $\cJ$ and $(\st_w)_{w\in \Wext}$ is an orthonormal basis of~$\cJ$. That is, $\langle \st_x,\st_y\rangle^{\infty}=\delta_{x,y}$. 

For $A\in\cJ$ let 
$$
\|A\|^2=\Tr^{\infty}(AA^*)=\langle A,A\rangle^{\infty}.
$$
Explicitly, if $A=\sum_w a_w\st_w$ then $\|A\|^2=\sum_w a_w^2$. In particular, we have the following obvious lemma. 

\begin{lemma}\label{lem:norm}
If $A\in\cJ$ with $\|A\|^2=1$ then $A=\pm \st_w$ for some $w\in\Wext$. 
\end{lemma}
%

\subsection{A factorisation formula in $\cJ_{\Delta_1}$}

Recall the definition of $h_{x,y,z}$ from Section~\ref{sec:KL2}, and define $f_{x,y,z}\in\sR$ by 
$$
T_xT_y=\sum_{z\in \Wext}f_{x,y,z}T_z. 
$$
The following lemma gives a useful way of computing certain coefficients of $h_{x,y,z}$ from the much simpler polynomials $f_{x,y,z}$. Define $g_{x,y,z}\in\sR$ by $C_xC_y=\sum_z g_{x,y,z}T_z$.

\begin{lemma}\label{lem:filtration}
For all $x,y,z\in \Wext$ we have
\begin{compactenum}[$(1)$]
\item $\deg f_{x,y,z},\deg g_{x,y,z},\deg h_{x,y,z}\leq \ell(x)+\ell(y)-\ell(z)$;
\item the coefficients of both $\sq^{\ell(x)+\ell(y)-\ell(z)}$ and $\sq^{\ell(x)+\ell(y)-\ell(z)-1}$ in $f_{x,y,z}$, $g_{x,y,z}$ and $h_{x,y,z}$ are equal.
\end{compactenum}
\end{lemma}

\begin{proof}
We have $\deg f_{x,y,z}\leq \ell(x)+\ell(y)-\ell(z)$ and $\deg h_{x,y,z}\leq \ell(x)+\ell(y)-\ell(z)$ by \cite[Section~13.1]{Lus:03}. Writing $C_x=T_x+\sum_{x'<x}P_{x',x}T_{x'}$, and similarly for $C_y$, and using these in $C_xC_y=\sum_z g_{x,y,z}T_z$ gives 
$$
g_{x,y,z}=f_{x,y,z}+\sum_{y'<y}P_{y',y}f_{x,y',z}+\sum_{x'<x}P_{x',x}f_{x',y,z}+\sum_{x'<x, y'<y}P_{x',x}P_{y',y}f_{x',y',z}. 
$$
In particular $\deg g_{x,y,z}\leq \ell(x)+\ell(y)-\ell(z)$ (proving (1)), and 
$$
\deg (g_{x,y,z}-f_{x,y,z})\leq \ell(x)+\ell(y)-\ell(z)-2. 
$$

On the other hand, we have $T_z=C_z+\sum_{z'<z}R_{z',z}C_{z'}$ where $R_{z',z}\in\sq^{-1}\ZZ[\sq^{-1}]$. Using this expression in $C_xC_y=\sum_z g_{x,y,z}T_z$, and comparing with $C_xC_y=\sum_z h_{x,y,z}C_z$, it follows that 
$$
h_{x,y,z}=g_{x,y,z}+\sum_{z'>z}g_{x,y,z'}R_{z,z'}. 
$$
This gives $\deg (h_{x,y,z}-g_{x,y,z})\leq \ell(x)+\ell(y)-\ell(z)-2$. Hence the coefficients of $\sq^{\ell(x)+\ell(y)-\ell(z)}$, and also of $\sq^{\ell(x)+\ell(y)-\ell(z)-1}$, in $f_{x,y,z}$, $g_{x,y,z}$, and $h_{x,y,z}$ are equal. 
\end{proof}

\begin{lemma}\label{lem:zbound}
Let $x,y\in\Delta_1$. If $\gamma_{x,y,z^{-1}}\neq 0$ then 
$$
\ell(z)\leq \ell(x)+\ell(y)-n(n-1). 
$$
\end{lemma}

\begin{proof}
If $\gamma_{x,y,z^{-1}}\neq 0$ then $z\in\Delta_1$ (by $\mathsf{P8}$) and so $\deg h_{x,y,z}=\ba(z)=n(n-1)$, and the result follows from Lemma~\ref{lem:filtration}(1). 
\end{proof}

Recall that the distinguished involution of $\Ga_u$ is $\sd(u)=u^{-1}\sw_n u$ for $u\in \sB$.

\begin{lemma}\label{lem:duflobasics}
For $u,v\in\sB$ we have
\begin{compactenum}[$(1)$]
\item $\st_{\sd(u)}\st_w=\st_w$ if $w\in\Ga_u$, and $\st_{\sd(u)}\st_w=0$ if $w\notin\Ga_u$. 
\item $\st_w\st_{\sd(v)}=\st_w$ if $w\in\Ga_v^{-1}$, and $\st_w\st_{\sd(v)}=0$ if $w\notin\Ga_v^{-1}$. 
\item $\st_{\sd(u)}\st_w\st_{d(v)}=\st_w$ if $w\in\Ga_u\cap\Ga_v^{-1}$ and $\st_{\sd(u)}\st_w\st_{d(v)}=0$ if $w\notin\Ga_u\cap\Ga_v^{-1}$. 
\item $\st_{\sd(u)}\cJ_{\Delta_1}\st_{\sd(v)}=\mathsf{span}_{\ZZ}\{\st_w\mid w\in \Ga_u\cap\Ga_v^{-1}\}=\mathsf{span}_{\ZZ}\{\st_{u^{-1}\sm_{\la}v}\mid \la\in P_n^+\}$. 
\item If $x\in\Ga_u$ and $y\in\Ga_v$ then $\st_x\st_y^*\in\st_{\sd(u)}\cJ_{\Delta_1}\st_{\sd(v)}$. 
\end{compactenum}
\end{lemma}

\begin{proof}
(1) By $\mathsf{P7}$ and $\mathsf{P2}$ we have $\st_{\sd(u)}\st_w=\sum_z\gamma_{\sd(u),w,z^{-1}}\st_z=\sum_z\gamma_{w,z^{-1},\sd(u)}\st_z=\gamma_{w,w^{-1},\sd(u)}\st_w$ and the result follows from $\mathsf{P5}$ and $\mathsf{P13}$. Statements (2) and (3) follow. 

(4) From (3) we have $\st_{\sd(u)}\cJ_{\Delta_1}\st_{\sd(v)}=\mathsf{span}_{\ZZ}\{\st_w\mid w\in \Ga_u\cap\Ga_v^{-1}\}$, and the result follows from Theorem~\ref{thm:main1}. For (5), we have $\st_x\st_y^*=\sum_z\gamma_{x,y^{-1},z^{-1}}\st_z$, and if $\gamma_{x,y^{-1},z^{-1}}\neq 0$ then by $\mathsf{P7}$ and $\mathsf{P8}$ we have $z\in\Ga_u\cap\Ga_v^{-1}$, and apply (4). 
\end{proof}

We now prove Theorem~\ref{thm:main4}, which we restate for convenience. 

\begin{thm}
For $u,v\in \sB$ and $\la\in P_n^+$ we have
$
\st_{u^{-1}\sw_n}\st_{\sm_{\la}}\st_{\sw_nv}=\st_{u^{-1}\sm_{\la}v}. 
$
\end{thm}

\begin{proof}
Write $a_u=\st_{u^{-1}\sw_n}$. We claim the following:
\begin{align}
\label{eq:first}a_ua_u^*&=\st_{\sd(u)}\\
\label{eq:second}a_u^*a_u&=\st_{\sw_n}.
\end{align}

We first prove~(\ref{eq:first}). Since $u^{-1}\sw_n\in \Ga_u$ and $\sw_nu\in \Ga_u^{-1}$, by Lemma~\ref{lem:duflobasics} parts (4) and (5) we have 
 $$
 a_ua_u^*=\sum_{\lambda\in P_n^+}c_{\lambda}\st_{u^{-1}\sm_{\la}u}
 $$
 where $c_{\la}=\gamma_{u^{-1}\sw_n,\sw_nu,(u^{-1}\sm_{\la}u)^{-1}}$. By Lemma~\ref{lem:zbound}, if $c_{\la}\neq 0$ we have 
 $$
 \ell(u^{-1}\sm_{\la}u)\leq \ell(u^{-1}\sw_n)+\ell(\sw_nu)-n(n-1),
 $$
 and so by Theorem~\ref{thm:main1} we have $\ell(\sm_{\la})\leq n(n-1)$. By Corollary~\ref{cor:mlength} this forces $\la=0$ or $\la=e_n$. The latter case is impossible because $u^{-1}\sw_n,\sw_nu\in \Waff$ and $u^{-1}\sm_{e_n}u\in\Wext\backslash \Waff$. Hence $a_ua_u^*=k\st_{\sd(u)}$ for some integer $k$. Taking traces gives $k=1$, hence~(\ref{eq:first}). 
 
 We now prove~(\ref{eq:second}). Let $p=a_u^*a_u$. By~(\ref{eq:first}) and Lemma~\ref{lem:duflobasics}(1) we have $$p^2=a_u^*(a_ua_u^*)a_u=a_u^*\st_{\sd(u)}a_u=a_u^*a_u=p.$$
Moreover, since $p^*=p$ we compute (using Lemma~\ref{lem:duflobasics})
$$
(\st_{\sw_n}-p)(\st_{\sw_n}-p)^*=(\st_{\sw_n}-p)^2=\st_{\sw_n}^2-(\st_{\sw_n}a_u^*)a_u-a_u^*(a_u\st_{\sw_n})+p^2=\st_{\sw_n}-2p+p=\st_{\sw_n}-p. 
$$
Thus $\Tr^{\infty}((\st_{\sw_n}-p)(\st_{\sw_n}-p)^*)=\Tr^{\infty}(\st_{\sw_n}-p)=1-1=0$. Writing $\st_{\sw_n}-p=\sum_w c_w\st_w$ we have $\Tr^{\infty}((\st_{\sw_n}-p)(\st_{\sw_n}-p)^*)=\sum_w c_w^2$, and hence $p=\st_{\sw_n}$, proving~(\ref{eq:second}). 

We now complete the proof of the theorem. By~(\ref{eq:second}) we have 
$$\|a_u\st_{\sm_{\la}}\|^2=\Tr^{\infty}(a_u\st_{\sm_{\la}}\st_{\sm_{\la}}^*a_u^*)=\Tr^{\infty}(\st_{\sm_{\la}}\st_{\sm_{\la}}^*a_u^*a_u)=\Tr^{\infty}(\st_{\sm_{\la}}\st_{\sm_{\la}}^*\st_{\sw_n})=\Tr^{\infty}(\st_{\sm_{\la}}\st_{\sm_{\la}}^*)=1,$$
and so by Lemma~\ref{lem:norm} we have $a_u\st_{\sm_{\la}}=\pm \st_w$ for some $w\in\Wext$. 

To compute $w$ (and the sign), write $\sm_{\la}=\sw_nx_{\la}$ with $x_{\la}=\fu^{-1}\tau_{\la}\fu$. It is elementary that $T_{\sw_n}^2=\sq^{n(n-1)}(1+f)T_{\sw_n}+\sum_{y<\sw_n}c_yT_y$ with $f\in\sq^{-1}\ZZ[\sq^{-1}]$, and hence 
$$
T_{u^{-1}\sw_n}T_{\sm_{\la}}=T_{u^{-1}}T_{\sw_n}^2T_{x_{\la}}=\sq^{n(n-1)}(1+f)T_{u^{-1}\sm_{\la}}+\text{(terms $T_y$ with $\ell(y)<\ell(u^{-1}\sm_{\la})$)}.
$$
Thus $f_{u^{-1}\sw_n,\sm_{\la},u^{-1}\sm_{\la}}=\sq^{n(n-1)}(1+f)$, and since $n(n-1)=\ell(u^{-1}\sw_n)+\ell(\sm_{\la})-\ell(u^{-1}\sm_{\la})$ we have $h_{u^{-1}\sw_n,\sm_{\la},u^{-1}\sm_{\la}}=\sq^{n(n-1)}+\text{(strictly lower powers)}$ by Lemma~\ref{lem:filtration}. Since $n(n-1)=\ba(u^{-1}\sm_{\la})$ this gives $\gamma_{u^{-1}\sw_n,\sm_{\la},(u^{-1}\sm_{\la})^{-1}}=1$. Hence we have shown that $a_u\st_{\sm_{\la}}=\st_{u^{-1}\sm_{\la}}$. 

We now apply a similar argument on the right. We have
$$
\|\st_{u^{-1}\sm_{\la}}a_v^*\|^2=\Tr^{\infty}(\st_{u^{-1}\sm_{\la}}a_v^*a_v\st_{u^{-1}\sm_{\la}}^*)=\Tr^{\infty}(\st_{u^{-1}\sm_{\la}}\st_{\sw_n}\st_{u^{-1}\sm_{\la}}^*)=\Tr^{\infty}(\st_{u^{-1}\sm_{\la}}\st_{u^{-1}\sm_{\la}}^*)=1,
$$
and hence $\st_{u^{-1}\sm_{\la}}a_v^*=\pm \st_w$ for some $w\in\Wext$. Using the fact that $\sm_{\la}$ is an involution, we have $\sm_{\la}=x_{\la}^{-1}\sw_n$, and hence 
$$
T_{u^{-1}\sm_{\la}}T_{\sw_n v}=T_{u^{-1}}T_{x_{\la}^{-1}}T_{\sw_n}^2T_v=\sq^{n(n-1)}(1+f)T_{u^{-1}\sm_{\la}v}+(\text{terms $T_y$ with $\ell(y)<\ell(u^{-1}\sm_{\la}v)$}). 
$$
Since $n(n-1)=\ell(u^{-1}\sm_{\la})+\ell(\sw_nv)-\ell(u^{-1}\sm_{\la}v)$ we again use Lemma~\ref{lem:filtration} to conclude that $a_u\st_{\sm_{\la}}a_v^*=\st_{u^{-1}\sm_{\la}v}$. 
\end{proof}

\section{Leading matrices and the asymptotic algebra~$\cJ_{\Delta_1}$}\label{sec:J}

In this section we determine the asymptotic algebra $\cJ_{\Delta_1}$, proving Theorems~\ref{thm:main5}, \ref{thm:main6}, and Corollary~\ref{cor:main7}. As outlined in the introduction, the general approach centres around a careful study of a certain finite dimensional representation. We introduce this representation $\pi$ in Section~\ref{sec:pi}, and recall combinatorics of $\pi$ in Section~\ref{sec:paths}. In Sections~\ref{sec:killing} and~\ref{sec:bounded} we show that $\pi$ satisfies important ``killing'' and ``boundedness'' properties, allowing for the definition of \textit{leading matrices} $\fc(w)$ in Section~\ref{sec:leading}. We ultimately compute the leading matrices in terms of Schur functions of type $\sC_{n-1}$ in Section~\ref{sec:asymptoticfinal}. A crucial ingredient for this calculation is the identification of the \textit{asymptotic Plancherel density} associated to the cell in Section~\ref{sec:plancherel}. Finally we give applications, and verify Lusztig's conjecture \cite[Conjecture~10.5]{Lus:IV} for the cell, in Section~\ref{sec:applications2}.

\subsection{The representation $\pi=\pi_z$}\label{sec:pi}

Let $z_1,\ldots,z_{n-1},z_n$ be commuting invertible indeterminates, and write
$$
\sR[z]=\sR[z_1^{\pm 1},\ldots,z_{n}^{\pm 1}]/(z_n^2-1)\quad\text{and}\quad \ZZ[z]=\ZZ[z_1^{\pm 1},\ldots,z_{n}^{\pm 1}]/(z_n^2-1).
$$
Let $\epsilon$ denote the image of $z_n$ in $\sR[z]$, and so $\epsilon^2=1$. For $\lambda=\lambda_1e_1+\cdots+\lambda_{n}e_n$ write $z^{\lambda}=z_1^{\lambda_1}\cdots z_{n-1}^{\lambda_{n-1}}\epsilon^{\lambda_n}$. Since $\epsilon^2=1$ we have $z^{\lambda}=z^{\mu}$ whenever $\lambda-\mu\in \ZZ\alpha_n^{\vee}$. Thus each element $f\in\sR[z]$ has a unique expression as 
$$
f=\sum_{\la\in P_n}a_{\lambda}z^{\lambda}\quad\text{with $a_{\lambda}\in \sR$}. 
$$ 
Write $[f]_{\mathsf{ct}}=a_0$ for the \textit{constant term} of~$f$ in the above expression. Define an involution $*$ on $\sR[z]$ by $z_i^*=z_i^{-1}$ for $1\leq i\leq n-1$ and $\epsilon^*=\epsilon$.

Let $\cL_n=\langle T_n,X^{\lambda}\mid \lambda\in P\rangle$ be the \textit{Levi subalgebra} of $\Hext$ generated by $T_n$ and $X^{\lambda}$, with $\lambda\in P$. Consider a $1$-dimensional right $\cL_n$-module over $\sR[z]$ generated by an element $\xi_z$ with (for $1\leq i\leq n-1$)
\begin{align*}
\xi_z\cdot T_n&=-\sq^{-1}\xi_z\\
\xi_z\cdot X^{e_i}&=z_i\xi_z\\
\xi_z\cdot X^{e_n}&=-\epsilon\sq^{-1} \xi_z.
\end{align*}
(it is straightforward to verify that this defines a $\cL_n$-action using the Bernstein relation). Let $\psi_z:\langle X^{\la}\mid \la\in P\rangle\to \sR[z]$ be the evaluation map with $\psi_z(X^{e_i})=z_i$ for $1\leq i\leq n-1$ and $\psi_z(X^{e_n})=-\epsilon\sq^{-1}$ (thus $\psi_z$ gives the action of the element $X^{\la}$ on $\xi_z$).

Let $\pi_z=\mathsf{Ind}_{\cL_n}^{\Hext}(\xi_z)$ be the induced representation of $\Hext$ with associated right $\Hext$-module 
$$
\mathsf{M}_z=(\sR[z]\xi_z)\otimes_{\cL_n}\Hext.
$$
Let $\chi_z$ be the character of $\pi_z$. We will typically write $\pi=\pi_z$, $\chi=\chi_z$, and $\mathsf{M}=\mathsf{M}_z$, unless the dependence on~$z$ is explicitly required. 

By \cite[Section~5.5]{GLP:23}, for any choice of fundamental domain $F$ for the action of $\sT_n$ on $\mathbb{X}_n$ the set 
$$\{\xi_z\otimes X_u\mid u \in F\}$$ is a basis of $\mathsf{M}$, and we write $\pi(h;F)$ for the matrix of $\pi(h)$ with respect to this basis. Moreover if $F'$ is another fundamental domain then the change of basis matrix converting from the $F$-basis to the $F'$-basis is a monomial matrix with each nonzero entry of the form $z^{\lambda}$ with $\lambda\in P_n$.

\subsection{$n$-folded alcove paths and the path formula}\label{sec:paths}

In \cite{GLP:23} we developed a combinatorial formula for the matrix entries of $\pi(T_w;F)$ in terms of \textit{$J$-folded alcove paths}. In the present situation the relevant set $J$ is $J=\{n\}$, and we give a brief description of \textit{$n$-folded alcove paths} below (see \cite{GLP:23} for further information and illustrations). 

\begin{defn}
Let $\vec{w}=s_{i_1}\cdots s_{i_k}\sigma^a$ be a reduced expression for $w\in\Wext$. An \textit{$n$-folded alcove path of type $\vec{w}$} is a sequence $p=(u_0,\ldots,u_k,u_k\sigma^a)$ with $u_0,\ldots,u_k\in\Wext$ such that $u_0A_0\subseteq \cA_n$, and for $1\leq r\leq k$,
\begin{compactenum}[$(1)$]
\item $u_rA_0\subseteq \cA_n$ (that is, $u_r\in\mathbb{X}_n$),
\item $u_r\in\{u_{r-1},u_{r-1}s_{i_r}\}$, and if $u_{r-1}=u_r$ then either:
\begin{compactenum}[$(a)$]
\item $u_{r-1}s_{i_r}A_0\not\subseteq \cA_n$, or 
\item $u_{r-1}s_{i_r}A_0\subseteq \cA_n$, and the alcove $u_{r-1}A_0$ is on the positive side of the hyperplane separating the alcoves $u_{r-1}A_0$ and $u_{r-1}s_{i_r}A_0$. 
\end{compactenum}
\end{compactenum}
\end{defn}

If $p=(u_0,\ldots,u_k,u_k\sigma^a)$ is an $n$-folded alcove path, we define the \textit{start} of $p$ to be $\mathsf{start}(p)=u_0$, and the \textit{end} of $p$ to be $\mathsf{end}(p)=u_k\sigma^a$. Moreover:
\begin{compactenum}
\item[--] If $u_r=u_{r-1}s_{i_r}$ then the index $r$ is called a \textit{crossing};
\item[--] If $u_r=u_{r-1}$ and $u_{r-1}s_{i_r}A_0\subseteq\cA_n$ then the index $r$ is called a \textit{fold};
\item[--] If $u_r=u_{r-1}$ and $u_{r-1}s_{i_r}A_0\not\subseteq \cA_n$ then the index $r$ is called a \textit{bounce}. 
\end{compactenum}
We write 
$$
f(p)=\#(\text{folds in $p$})\quad\text{and}\quad b(p)=\#(\text{bounces in $p$}). 
$$

If $F$ is a fundamental domain for the action of $\sT_n$ on $\mathbb{X}_n$, then each $w\in \mathbb{X}_n$ can be written in a unique way as $w=\tau_{\lambda}u$ with $\lambda\in P_n$ and $u\in F$. We define the $F$-\textit{weight} and the $F$-\textit{final position} of $w=\tau_{\lambda}u$ as
$$
\mathsf{wt}_{F}(w)=\lambda\in P_n\quad\text{and}\quad \theta_{F}(w)=u\in F. 
$$
Note that the $F$-weight and $F$-final position typically change when one changes the fundamental domain~$F$. 

If $p$ is an $n$-folded alcove path we define $\mathsf{wt}_F(p)=\mathsf{wt}_{\sF}(\mathsf{end}(p))$ and $\theta_F(p)=\theta_F(\mathsf{end}(p))$. For $u\in \mathbb{X}_n$ and $\vec{w}$ a reduced expression for $w$, let 
$$
\cP_n(\vec{w},u)=\{\text{all $n$-folded alcove paths of type $\vec{w}$ starting at $u$}\}. 
$$

\begin{thm}[Path Formula]\label{thm:pathformula}
Let $F$ be a fundamental domain for the action of $\sT_n$ on $\mathbb{X}_n$, and let $(\xi_z\otimes X_u)_{u\in F}$ be an associated ordered basis. The matrix entries of $\pi(T_w)$ with respect to this basis are given by
$$
[\pi(T_w;F)]_{u,v}=\sum_{\{p\in\cP_n(\vec{w},u)\mid \theta_F(p)=v\}}(-\sq)^{-b(p)}(\sq-\sq^{-1})^{f(p)}z^{\mathsf{wt}_F(p)}.
$$
\end{thm}

Define a $*$-anti-involution on matrices $A\in\mathsf{Mat}_{|\sB|}(\sR[z])$ by transposition composed with applying $z_i^*=z_i^{-1}$ entrywise (with $\epsilon^*=\epsilon$). The path formula in Theorem~\ref{thm:pathformula} implies the following important property.

\begin{prop}\label{prop:*property}
For any choice of fundamental domain $F$ we have 
$$
\pi(T_w^*;F)=\pi(T_w;F)^*.
$$
\end{prop}

\begin{proof}
The proof in \cite[Lemma~4.6]{CGLP:24} applies verbatim for the fundamental domain $F={^n}\Wfin$, and changing fundamental domain results in a change of basis matrix that is monomial with nonzero entries of the form $z^{\la}$ for $\la\in P_n$, and the result follows. 
\end{proof}

\subsection{The killing property}\label{sec:killing}

For each $1\leq i\leq n$ define an \textit{intertwiner}
$$
U_i=T_i-\frac{\sq-\sq^{-1}}{1-X^{-\alpha_i^{\vee}}}. 
$$
Note that $U_i$ is not an element of the Hecke algebra, however it can be considered as an operator acting on any $\Hext$-module for which each $1-X^{-\alpha_i^{\vee}}$ acts invertibly (after an appropriate extension of scalars; the module $\mathsf{M}_z$ constructed in Section~\ref{sec:pi} will satisfy this property). We record the following well known consequences of the Bernstein relation (see \cite[Proposition~2.31]{CGLP:24}).

\begin{prop}\label{prop:inter}
We have the following. 
\begin{compactenum}[$(1)$]
\item The elements $U_i$ satisfy the braid relations, and hence the element $U_w=U_{i_1}\cdots U_{i_k}$ is independent of the particular reduced expression $w=s_{i_1}\cdots s_{i_k}$. 
\item $U_wX^{\lambda}=X^{w\lambda}U_w$ for all $\lambda\in P$ and $w\in \Wfin$. 
\item We have 
$$
U_i^2=\sq^2\frac{(1-\sq^{-2}X^{-\alpha_i^{\vee}})(1-\sq^{-2}X^{\alpha_i^{\vee}})}{(1-X^{-\alpha_i^{\vee}})(1-X^{\alpha_i^{\vee}})}.
$$
\item If $u,v\in \Wfin$ then $U_uU_v=b(X)U_{uv}$ for a rational function $b(X)$. 
\end{compactenum}
\end{prop}

Triangularity between $T_w$ and $U_w$ implies that $\mathsf{M}$ has ``basis'' $\{\xi_z\otimes U_w\mid w\in{^n}\Wfin\}$, where one must extend scalars to rational functions in~$z_1,\ldots,z_{n-1}$. Since $\xi_z\cdot T_n=-\sq^{-1}\xi_z$ and $\xi_z\cdot X^{\alpha_n^{\vee}}=\sq^{-2}\xi_z$ we have $\xi_z\cdot U_n=0$. It follows, using Proposition~\ref{prop:inter}, that with respect to the basis $\{\xi_z\otimes U_u\mid u\in {^n}\Wfin\}$, 
\begin{compactenum}[$(1)$]
\item the matrix $\pi(X^{\lambda})$ is diagonal for each $\lambda\in P$.
\item the matrix for $\pi(U_w)$, $w\in\Wfin$, has at most one nonzero entry in each row and column. Indeed, if $u\in {^n}\Wfin$ then the $u$th row of $\pi(U_w)$ is zero if $uw\notin {^n}\Wfin$, and has a nonzero entry only in the $(u,uw)$-position if $uw\in {^n}\Wfin$. 
\end{compactenum}

\begin{thm}[Killing Property]\label{thm:killing}
We have $\pi(C_{w})=0$ for all $w$ in the lowest two-sided cell. 
\end{thm}

\begin{proof}
It is well known that if $w\in \Delta_0$ then $C_w\in HC_{\sw_0}H$. Thus it is sufficient to show that $\pi(C_{\sw_0})=0$. It is well known that (see for example \cite[Theorem~2.32]{CGLP:24})
$$
C_{\sw_0}=\sq^{\ell(\sw_0)}\sum_{w\in W_0}\sq^{-\ell(w)}c_w(X)U_w\quad\text{where}\quad c_w(X)=\prod_{\alpha\in\Phi_0^+\backslash\Phi(w)}\frac{1-\sq^{-2}X^{-\alpha^{\vee}}}{1-X^{-\alpha^{\vee}}}.
$$
 Thus
$$
\pi(C_{\sw_0})=\sq^{\ell(\sw_0)}\sum_{w\in W_0}\sq^{-\ell(w)}\pi(c_w(X)U_w).
$$
Working in the basis of intertwiners, the matrices $\pi(c_w(X))$ are diagonal, and the matrices $\pi(U_w)$ have at most one non-zero entry in each row and column. Moreover, the matrices $\pi(U_w)$, $w\in \Wfin$, have distinct support (the places of the non-zero entries; because the $u$th row of $\pi(U_w)$ is either $0$, or has an entry only in the $uw$-column in the case that $uw\in {^n}\Wfin$, see Proposition~\ref{prop:inter}). It follows that $\pi(C_{\sw_0})=0$ if and only if $\pi(c_w(X)U_w)=0$ for all $w\in W_0$. This is equivalent to the statement 
$$
(\xi_z\otimes U_u)\cdot c_w(X)U_w=0\quad\text{for all $u\in{^n}\Wfin$ and $w\in W_0$}.
$$
We have 
$$
(\xi_z\otimes U_u)\cdot c_w(X)U_w=\bigg[\prod_{\alpha\in\Phi_0^+\backslash\Phi(w)}\frac{1-\sq^{-2}\psi_z(X^{-u\alpha^{\vee}})}{1-\psi_z(X^{-u\alpha^{\vee}})}\bigg](\xi_z\otimes U_uU_w).
$$
If $uw\notin {^n}\Wfin$ then $\xi\otimes U_uU_w=0$ (because $U_uU_w=b(X) U_{uw}$ for a rational function $b(X)$, and if $uw\notin {^n}\Wfin$ then $uw=s_ny$ with $\ell(s_ny)=\ell(y)+1$, and hence $U_{uw}=U_{n}U_y$, but $\xi_z\cdot U_n=0$). Thus suppose that $uw\in {^n}\Wfin$. Since $\ell(s_nu)=\ell(u)+1$ we have $u^{-1}\alpha_n>0$, and since $\ell(s_nuw)=\ell(uw)+1$ we have $w^{-1}(u^{-1}\alpha_n)>0$, and so $u^{-1}\alpha_n\in\Phi_0^+\backslash\Phi(w)$. Hence the factor $1-\sq^{-2}\psi_z(X^{-\alpha_n^{\vee}})$ appears in the numerator of the above product, and since $\psi_z(X^{-\alpha_n^{\vee}})=\sq^2$ we have $\pi(C_{\sw_0})=0$ as required. 
\end{proof}

\subsection{Boundedness}\label{sec:bounded}

The following result links the degree of matrix entries of $\pi(T_w)$ to $\ba(\Delta_1)=n(n-1)$. 
\begin{thm}[Boundedness]\label{thm:bounded}
Let $F$ be any ordered fundamental domain. Then 
$$\deg[\pi(T_w;F)]_{u,v}\leq n(n-1)\quad\text{for all $w\in\Wext$ and $u,v\in F$},
$$ and if equality holds for some $u,v\in F$ then $w\in \Delta_1$.
\end{thm}

\begin{proof}
By \cite[Theorem~6.9]{GLP:23} there is a global bound $\deg[\pi(T_{w};F)]_{u,v}\leq N$ for all $w\in\Wext$ and $u,v\in F$. Let $\mathbf{b}$ be the minimum such integer~$N$ (by \cite[Proposition~6.5]{GLP:23} the integer $\mathbf{b}$ is independent of the particular fundamental domain chosen, and so we will write $\pi(T_w)$ in place of $\pi(T_w;F)$ in the proof below). 

Let $w\in \Wext$ be such that $\pi(T_w)$ attains the degree bound~$\bb$, and suppose that this degree is attained on the $u$th row of $\pi(T_w)$. Let $v_1,\ldots,v_k\in F$ index the columns with $\deg[\pi(T_w)]_{u,v_r}=\bb$ for $1\leq r\leq k$. Write 
$$
[\pi(T_w)]_{u,v_r}=a_r\sq^{\bb}+(\textrm{terms of strictly lower degree}), \text{ with $a_r\in \ZZ[z]$}.
$$
By Proposition~\ref{prop:*property} we have $
\pi(T_{w^{-1}})=\pi(T_w)^*
$ and so
$$
[\pi(T_{w^{-1}})]_{v_r,u}=a_r^*\sq^{\bb}+(\textrm{terms of strictly lower degree}).
$$
By triangularity it follows that $[\pi(C_w)]_{u,v_r}=a_r\sq^{\bb}+(\textrm{terms of strictly lower degree})$ and that $[\pi(C_{w^{-1}})]_{v_r,u}=a_r^*\sq^{\bb}+(\textrm{terms of strictly lower degree})$. Thus
\begin{align*}
[\pi(C_w)\pi(C_{w^{-1}})]_{u,u}&=\bigg(\sum_{r=1}^ka_ra_r^*\bigg)\sq^{2\bb}+(\textrm{terms of strictly lower degree}).
\end{align*}
Writing $a_r=\sum_{\la\in P_n}a_{r,\la}z^{\la}$ with $a_{r,\la}\in\ZZ$ we have $[a_ra_r^*]_{\mathsf{ct}}=\sum_{\la\in P_n}a_{r,\la}^2>0$, and hence the coefficient of $\sq^{2\bb}$ in the above expression for $[\pi(C_w)\pi(C_{w^{-1}})]_{u,u}$ does not vanish. Thus the matrix $\pi(C_w)\pi(C_{w^{-1}})$ attains degree $2\bb$ in the $(u,u)$-entry.

On the other hand,
\begin{align}\label{eq:applypi}
\pi(C_w)\pi(C_{w^{-1}})=\sum_{z}h_{w,w^{-1},z}\pi(C_z).
\end{align}
By Theorem~\ref{thm:killing} the sum is over $z$ in two-sided cells $\Delta$ with $\Delta\geq_{LR} \Delta_1$, and hence by $\mathsf{P4}$ we have $\ba(z)\leq n(n-1)$ for all such $z$. By triangularity and boundedness we have $\deg\pi(C_z)\leq \bb$, and since $\deg h_{w,w^{-1},z}\leq \ba(z)$ it follows that the maximum degree on the right hand side of~(\ref{eq:applypi}) is $n(n-1)+\bb$. Thus $2\bb\leq n(n-1)+\bb$, and hence $\bb\leq n(n-1)$.

Suppose now that $\pi(T_w)$ attains degree~$n(n-1)$, in position $(u,v)$, say. As explained in the first paragraph of the proof, $\pi(C_w)\pi(C_{w^{-1}})$ attains degree $2n(n-1)$ in position $(u,u)$, and then~(\ref{eq:applypi}) implies that there exists $z\in\Delta$ with $\Delta\geq_{LR}\Delta_1$ such that $\deg h_{w,w^{-1},z}=n(n-1)=\deg[\pi(C_z)]_{u,u}$. Thus $\ba(z)=n(n-1)$ (because $\ba(z)\leq n(n-1)$) and hence $z\in\Delta_{1}$ by $\mathsf{P11}$. But then $\mathsf{P8}$ gives $w\sim_R z$, and so $w\in \Delta_1$ as required.
\end{proof}

\begin{example}\label{ex:pathscancel}
Comparing the path formula Theorem~\ref{thm:pathformula} with the boundedness result Theorem~\ref{thm:bounded}, it is natural to ask if the contribution $(-\sq)^{-b(p)}(\sq-\sq^{-1})^{f(p)}$ of each path $p\in\cP_n(\vec{w},u)$ is bounded by $n(n-1)$. This is not the case, and in general it appears to be a complicated problem to give a combinatorial proof of Theorem~\ref{thm:bounded} using the path formula. For example, consider $n=4$, and so $n(n-1)=12$. Let $\vec{w}=13213432434123024343202$ (a reduced expression) and $v=t_{e_1+e_4}s_{e_1}$. The paths $p_i$, $1\leq i\leq 4$, given below are $n$-folded alcove paths of type $\vec{w}$ starting at the identity and ending at $v$ (here $\hat{i}$ indicates a fold, and $\check{i}$ indicates a bounce):
\begin{align*}
p_1&=\hat{1}\hat{3}\hat{2}\hat{1}34\hat{3}24\hat{3}\hat{4}1\hat{2}\hat{3}024\hat{3}43\hat{2}\hat{0}\hat{2}&&\text{($13$ folds, $0$ bounces)}\\
p_2&=\hat{1}\hat{3}\hat{2}\hat{1}34\hat{3}24\hat{3}\hat{4}1\hat{2}\hat{3}02\hat{4}3\check{4}\hat{3}\hat{2}\hat{0}\hat{2}&&\text{($14$ folds, $1$ bounce)}\\
p_3&=\hat{1}\hat{3}\hat{2}\hat{1}\hat{3}\check{4}32\hat{4}\hat{3}\hat{4}1\hat{2}\hat{3}024\hat{3}43\hat{2}\hat{0}\hat{2}&&\text{($14$ folds, $1$ bounce)}\\
p_4&=\hat{1}\hat{3}\hat{2}\hat{1}\hat{3}\check{4}32\hat{4}\hat{3}\hat{4}1\hat{2}\hat{3}02\hat{4}3\check{4}\hat{3}\hat{2}\hat{0}\hat{2}&&\text{($15$ folds, $2$ bounces)}.
\end{align*}
In particular, each of these paths $p_i$ has degree contribution exceeding $12$ (indeed it turns out that these $4$ paths are precisely the $n$-folded paths of type $\vec{w}$ starting at the identity, ending at~$v$, with degree contribution at least~$12$). Their respective contributions are $(\sq-\sq^{-1})^{13}$, $-\sq^{-1}(\sq-\sq^{-1})^{14}$, $-\sq^{-1}(\sq-\sq^{-1})^{14}$, and $\sq^{-2}(\sq-\sq^{-1})^{15}$ (with each multiplied by $z^{e_1+e_4}=\epsilon z_1$). The combined contribution is thus $(1-2(1-\sq^{-2})+(1-\sq^{-2})^2)(\sq-\sq^{-1})^{13}=\sq^{-4}(\sq-\sq^{-1})^{13}$ which has degree only~$9$ (in agreement with Theorem~\ref{thm:bounded}). 

Thus we see that the minus signs introduced by bounces can lead to delicate cancellations of leading terms in the path formula, making a ``path-by-path'' bounding argument difficult. Of particular note, the contribution from path $p_1$ exceeds $n(n-1)$ without any bounces, however the leading term is cancelled by paths with more folds and with bounces, having the same degree contribution. 
\end{example}

\subsection{The matrix $\pi(C_{\sw_n})$ with respect to the $W_n$-fundamental domain}\label{sec:Wn}

It is well known that the Kazhdan-Lusztig basis element for the longest element $\sw_J$ of a spherical Coxeter group $W_J$ is 
$$
C_{\sw_J}=\sq^{-\ell(\sw_J)}\sum_{w\in W_J}\sq^{\ell(w)}T_w=\sq^{\ell(\sw_J)}\sum_{w\in W_J}\sq^{-\ell(w)}T_{w^{-1}}^{-1}. 
$$
In particular, note that $C_{\sw_J}^2=\sq^{-\ell(\sw_J)}W_J(\sq^2)C_{\sw_J}$, where $W_J(\sq^2)=\sum_{w\in W_J}\sq^{2\ell(w)}$.

The formula in Lemma~\ref{lem:CJp} below gives another expression for $C_{\sw_J}$ for any proper subset $J\subset\Saff$ in terms of signed lengths, where the \textit{signed length} of $w\in\Wext$ is
$$
\veps(w)=\#\{\beta\in\Phi(w)\mid wA_0\subseteq H_{\beta}^+\}-\#\{\beta\in\Phi(w)\mid wA_0\subseteq H_{\beta}^-\}. 
$$

\begin{lemma}\label{lem:CJp} If $J\subset\Saff$ is a proper subset, then
$$
C_{\sw_J}=\sq^{-\veps(\sw_J)}\sum_{w\in W_{J}}\sq^{\veps(w)}X_w.
$$
\end{lemma}

\begin{proof}
Note that $C_{\sw_J}$ is the unique element $A$ of $H$ of the form $A=\sum_{v\in W_J}a_vT_v$ with $a_{\sw_J}=1$ and $AT_{s_j}=\sq A$ for all $j\in J$ (to see this, use the formula $AT_{s_j}=\sq A$ to see that $a_{vs_j}=\sq^{-1}a_v$ whenever $vs_j<v$, which determines all coefficients $a_v$ given that $a_{\sw_J}=1$). It is straightforward to check that the element $A=\sq^{-\veps(\sw_J)}\sum_{w\in W_J}\sq^{\veps(w)}X_w$ satisfies these properties. 
%
\end{proof}

\begin{lemma}\label{lem:etalebasics1}
We have $\veps(\sw_{n})=0$, and 
$$\veps(w)\leq\ell(\fu)\quad\text{ for all $w\in W_{n}$}
$$
with equality if and only if~$w=\fu$. 
\end{lemma}

\begin{proof}
We have  $\sw_{n}=\sw_{0,n} \fu$ with $\ell(\sw_n)=\ell(\sw_{0,n})+\ell(\fu)$ (see Lemma~\ref{lem:etalebasics}) and so $\Phi(\sw_n)=\Phi_n^+=\Phi(\sw_{0,n})\sqcup \sw_{0,n}\star \Phi(\fu)$. We have $\Phi(\sw_{0,n})=\{e_i-e_j\mid 1\leq i<j\leq n\}$ and $\sw_{0,n}\star\Phi(\fu)=\Phi(\fu)=\{-(e_i+e_j)+\delta\mid 1\leq i<j\leq n\}$ (see Lemma~\ref{lem:etalebasics}). The roots $\beta\in\Phi(\sw_{0,n})=\Phi_{0,n}^+$ have $\sw_nA_0\subseteq H_{\beta}^-$, and the roots $\beta\in \sw_{0,n}\star \Phi(\fu)$ have $\sw_nA_0\subseteq H_{\beta}^+$, and hence $\veps(\sw_n)=-\ell(\sw_{0,n})+\ell(\fu)=0$. Moreover, if $w$ is any element of $W_{n}$ then the hyperplanes crossed in a reduced expression for $w$ are contained in $\Phi_{n}^+=\Phi(\sw_{0,n})\sqcup\Phi(\fu)$, and hence $\veps(w)\leq \veps(\fu)$ with equality if and only if $w=\fu$. 
\end{proof}

\begin{cor}\label{cor:CJp}
We have 
$$
\CwJp=\sum_{w\in W_{n}}\sq^{\veps(w)}X_w. 
$$
\end{cor}

\begin{proof}
This follows from Lemmas~\ref{lem:CJp} and~\ref{lem:etalebasics1}.
\end{proof}

\begin{lemma}\label{lem:Cform}
With respect to the $W_n$-fundamental domain, we have 
$$
[\pi(\CwJp;W_n)]_{u,v}=\sq^{\veps(u)+\veps(v)}\quad\text{for $u,v\in W_n$}.
$$
Moreover $\chi(C_{\sw_n})=\sq^{-\ell(\sw_n)}W_n(\sq^2)$. 
\end{lemma}

\begin{proof}
For $u\in W_n$ we have
$
(\xi\otimes X_u)\cdot \CwJp=\sq^{\veps(u)}(\xi\otimes \CwJp),
$
and the first statement follows using the formula in Corollary~\ref{cor:CJp}. For the trace formula, we have $\chi(C_{\sw_n})=\sum_{u\in W_n}\sq^{2\veps(u)}$. But the formula in Corollary~\ref{cor:CJp} implies that $C_{\sw_n}^2=\big(\sum_{u\in W_n}\sq^{2\veps(u)}\big)C_{\sw_n}$ and hence $\sum_{u\in W_n}\sq^{2\veps(u)}=\sq^{-\ell(\sw_n)}W_n(\sq^2)$. 
\end{proof}

\begin{cor}\label{cor:Cmat}
The matrix $\pi(\CwJp;W_n)$ is bounded in degree by $n(n-1)$. Moreover, this degree is attained only in the $(\fu,\fu)$-position, where the entry is $\sq^{n(n-1)}$.
\end{cor} 

\begin{proof}
The general degree bound follows from Theorem~\ref{thm:bounded}, however we see it directly here using Lemma~\ref{lem:Cform} and Lemma~\ref{lem:etalebasics} because $\deg[\pi(\CwJp)]_{u,v}\leq 2\ell(\fu)=\ell(\sw_n)=n(n-1)$ with equality if and only if $u=v=\fu$. 
\end{proof}

\begin{remark}
Using the $W_n$-fundamental domain leads to the remarkably simple formula for $\pi(C_{\sw_n})$ in Lemma~\ref{lem:Cform}. We note that other choices of fundamental domain introduce monomials $z^{\la}$ into the matrix entries. 
\end{remark}

%

\subsection{Leading matrices and the asymptotic representation}\label{sec:leading}

%
%
%

\begin{defn}\label{defn:leading}
The \textit{leading matrix} of $w\in\Wext$ is 
$$
\fc(w)=\mathsf{sp}_{\sq^{-1}=0}\big(\sq^{-n(n-1)}\pi(C_w;\sF)\big)=\mathsf{sp}_{\sq^{-1}=0}\big(\sq^{-n(n-1)}\pi(T_w;\sF)\big)\in\mathsf{Mat}_{2^{n-1}n!}(\ZZ[z]),
$$
where $\mathsf{sp}_{\sq^{-1}=0}$ is specialisation at $\sq^{-1}=0$. By Theorem~\ref{thm:bounded} this specialisation is well defined, and if $\fc(w)\neq 0$ then $w\in\Delta_1$. Moreover, one can use either $C_w$ or $T_w$ in the above definition, by triangularity between these bases. 
\end{defn}

\begin{remark}
Note that in the definition of the leading matrices we are fixing the choice of fundamental domain $\sF=\fu\sB$. This is the choice of fundamental domain that ultimately leads to the cleanest description of the asymptotic algebra.  
\end{remark}

We will index rows and columns of our matrices by $\sB$. Thus the $(u,v)$-entry corresponds to the row and column of $\fu u$ and $\fu v$. For $u,v\in\sB$ let $E_{u,v}$ denote the $|\sB|\times|\sB|$ matrix with $(u,v)$-entry equal to~$1$ and all other entries equal to~$0$. By Corollary~\ref{cor:Cmat}, and the fact that $\fu\in W_n\cap \sF$, we have
\begin{align}\label{eq:Cwn}
\fc(\sw_n)=E_{11}. 
\end{align}

We also have the following important property. 

\begin{cor}\label{cor:starinv}
We have $\fc(w^{-1})=\fc(w)^*$ for all $w\in\Wext$.
\end{cor}

\begin{proof}
This follows from the formulae for $[\pi(T_w)]_{u,v_r}$ and $[\pi(T_{w^{-1}})]_{v_r,u}$ in the proof of Theorem~\ref{thm:bounded}. 
\end{proof}

The following theorem shows that the leading matrices give a representation of the asymptotic algebra (ultimately we will prove that this representation is faithful, but this is far from obvious at this stage).

\begin{thm}[The asymptotic representation]\label{thm:asymptoticrepresentation}
The $\ZZ$-linear map 
$$
\pi^{\infty}:\mathcal{J}_{\Delta_1}\to \mathsf{Mat}_{|\sB|}(\ZZ[z])\quad\text{with}\quad \st_w\mapsto \fc(w)
$$
is a $\ZZ$-algebra homomorphism. 
\end{thm}

\begin{proof}
For $x,y\in\Delta_1$ we have 
$$
[\sq^{-n(n-1)}\pi(C_x;\sF)][\sq^{-n(n-1)}\pi(C_y;\sF)]=\sum_{z\in\Wext}[\sq^{-n(n-1)}h_{x,y,z}][\sq^{-n(n-1)}\pi(C_z;\sF)].
$$
The left hand side specialises at $\sq^{-1}=0$ to $\fc(x)\fc(y)$. By Theorem~\ref{thm:killing} the sum on the right is over $z\in\Delta$ with $\Delta\geq \Delta_1$, and hence $\ba(z)\leq n(n-1)$ by $\mathsf{P4}$. Hence $\deg \sq^{-n(n-1)}h_{x,y,z}\leq 0$ for all terms in the sum. Thus each nonzero term $[\sq^{-n(n-1)}h_{x,y,z}][\sq^{-n(n-1)}\pi(C_z;\sF)]$ specialises at $\sq^{-1}=0$, and if this specialisation is non-zero then $\fc(z)\neq 0$ (and so $z\in\Delta_1$, see Definition~\ref{defn:leading}), and $\deg h_{x,y,z}=n(n-1)=\ba(z)$ (and so $\sq^{-n(n-1)}h_{x,y,z}$ specialises to $\gamma_{x,y,z^{-1}}$). Hence
$$
\fc(x)\fc(y)=\sum_{z\in\Delta_1}\gamma_{x,y,z^{-1}}\fc(z)
$$
as required. 
\end{proof}

\begin{cor}\label{cor:partial}
We have the following.
\begin{compactenum}[$(1)$]
\item If $w\in\Ga\cap\Ga^{-1}$ and $\fc(w)\neq 0$ then $\fc(w)$ has a unique nonzero entry in the $(1,1)$-position. 
\item If $\fc(w)$ has a nonzero entry in the $(1,1)$-position then $w\in\Ga\cap\Ga^{-1}$ and the nonzero entry is unique. 
\end{compactenum}
\end{cor}

\begin{proof}
(1) If $w\in\Ga\cap\Ga^{-1}$ then $\st_{\sw_n}\st_w=\st_w$ and $\st_w\st_{\sw_n}=\st_w$ (see Lemma~\ref{lem:duflobasics}). Thus by Theorem~\ref{thm:asymptoticrepresentation} we have $\fc(\sw_n)\fc(w)=\fc(w)$ and $\fc(w)\fc(\sw_n)=\fc(w)$. By~(\ref{eq:Cwn}) we have $\fc(\sw_n)=E_{11}$ and it follows that if $\fc(w)$ is nonzero, then there is a unique nonzero entry, and this entry is in the $(1,1)$-position. 

(2) If $\fc(w)$ has a nonzero entry in the $(1,1)$-position then $\fc(\sw_n)\fc(w)\neq 0$, and so by Theorem~\ref{thm:asymptoticrepresentation} we have $\st_{\sw_n}\st_w\neq 0$ and so $w\sim_R\sw_n$. Similarly since $\fc(w)\fc(\sw_n)\neq 0$ we have $w\sim_L\sw_n$, and so $w\in\Ga\cap\Ga^{-1}$. 
\end{proof}

\subsection{$G_n$-invariance of $\chi$}

The group $G_n$ acts on $\sR[z]$ in the natural way as signed permutations of $z_1,\ldots,z_{n-1}$ and fixing~$\epsilon$. Write 
$$
\sR[z]^{G_n}=\{f(z)\in\sR[z]\mid g\cdot f(z)=f(z)\text{ for all $g\in G_n$}\},
$$
and similarly $\ZZ[z]^{G_n}=\{f(z)\in\ZZ[z]\mid g\cdot f(z)=f(z)\text{ for all $g\in G_n$}\}$.

\begin{thm}\label{thm:invariance}
Working over $\QQ(\sq,z_1,\ldots,z_{n-1})[z_n]/(z_n^2-1)$ we have $\pi_z\cong \pi_{gz}$ for all $g\in G_n$. 
\end{thm}

\begin{proof}
For $1\leq i\leq n-2$ let $U_i'=U_i$, and let $U_{n-1}'=U_{n-1}U_nU_{n-1}$. For $1\leq i\leq n-1$ we have $U_i'T_n=T_nU_i'$. We claim that $U_{n-1}'T_n=T_nU_{n-1}'$. Since $T_n=U_n+(\sq-\sq^{-1})/(1-X^{-2e_n})$ it is sufficient to prove that $U_{n-1}'$ commutes with both $U_n$ and $X^{e_n}$. By the braid relation we have $U_{n-1}'U_n=U_{n-1}U_nU_{n-1}U_n=U_nU_{n-1}U_nU_{n-1}=U_nU_{n-1}'$, and since $s_{n-1}s_ns_{n-1}e_n=e_n$ we have $U_{n-1}'X^{e_n}=X^{e_n}U_{n-1}'$. Hence the claim. 

Thus for each generator $s_i'$ of $G_n$ (where $s_i'=s_i$ for $1\leq i\leq n-2$ and $s_{n-1}'=s_{n-1}s_ns_{n-1}$) there is an intertwiner $U_i'$ commuting with $T_n$ and with $U_i'p=(s_i'\cdot p)U_i'$ for $p\in \sR[P]$. These operators satisfy the $\sB_{n-1}$ braid relations. Indeed, the only nontrivial case to check is the $4$-term braid relation for $U_{n-2}'$ and $U_{n-1}'$:
\begin{align*}
U_{n-2}'U_{n-1}'U_{n-2}'U_{n-1}'&=U_{n-2}U_{n-1}U_n(U_{n-1}U_{n-2}U_{n-1})U_nU_{n-1}\\
&=U_{n-2}U_{n-1}U_n(U_{n-2}U_{n-1}U_{n-2})U_nU_{n-1}\\
&=(U_{n-2}U_{n-1}U_{n-2})U_nU_{n-1}U_nU_{n-2}U_{n-1}\\
&=U_{n-1}U_{n-2}(U_{n-1}U_nU_{n-1}U_n)U_{n-2}U_{n-1}\\
&=U_{n-1}U_{n-2}U_{n}U_{n-1}U_{n}(U_{n-1}U_{n-2}U_{n-1})\\
&=U_{n-1}U_{n}U_{n-2}U_{n-1}U_{n}U_{n-2}U_{n-1}U_{n-2}\\
&=U_{n-1}U_{n}(U_{n-2}U_{n-1}U_{n-2})U_{n}U_{n-1}U_{n-2}\\
&=(U_{n-1}U_{n}U_{n-1})U_{n-2}(U_{n-1}U_{n}U_{n-1})U_{n-2}\\
&=U_{n-1}'U_{n-2}'U_{n-1}'U_{n-2}'. 
\end{align*}
Hence for each $g\in G_n$ there is a well defined intertwiner $U_g'=U_{i_1}'\cdots U_{i_k}'$ whenever $g=s_{i_1}'\cdots s_{i_k}'$ is a reduced expression in the Coxeter group $G_n$ generated by $s_1',\ldots,s_{n-1}'$. 

For each $g\in G_n$ define a right module intertwiner $I_g(z):\mathsf{M}_z\to \mathsf{M}_{gz}$ by $\xi_z\otimes h\mapsto \xi_{gz}\otimes U_g'h$. To check that this is well defined, note that:
$$
(\xi_{gz}\otimes U_g')X^{\lambda}=\xi_{gz}\otimes X^{g\lambda}U_g'=\psi_z(X^{\la})(\xi_{gz}\otimes U_g'),
$$
and 
$$
(\xi_{gz}\otimes U_g')T_n=\xi_{gz}\otimes T_nU_g'=-\sq^{-1}(\xi_{gz}\otimes U_g'). 
$$
To prove that $\pi_z\cong \pi_{gz}$ over $\QQ(\sq,z_1,\ldots,z_{n-1})[z_n]/(z_n^2-1)$ we must show that each $I_g(z)$ is invertible, and it suffices to prove this for the generators $g$ of $G_n$. For $1\leq i\leq n-2$ we have (by Proposition~\ref{prop:inter})
$$
I_{s_i'}(s_i'z)\circ I_{s_i'}(z)(\xi_z\otimes 1)=(\xi_z\otimes U_i^2)=\sq^2\frac{(1-\sq^{-2}z_i^{-1}z_{i+1})(1-\sq^{-2}z_iz_{i+1}^{-1})}{(1-z_i^{-1}z_{i+1})(1-z_iz_{i+1}^{-1})}(\xi_z\otimes 1),
$$
and since $\xi_z\otimes 1$ is cyclic this gives $I_{s_i'}(s_i'z)\circ I_{s_i'}(z)$ is a nonzero multiple of the identity. For the generator $s_{n-1}'$, a computation using Proposition~\ref{prop:inter} shows that $U_{n-1}'^2=\kappa(X_{n-1},X_n)$, where $\kappa(X,Y)$ equals
\begin{align*}
\sq^6\frac{(1-\sq^{-2}X^{-1}Y^{-1})(1-\sq^{-2}XY)(1-\sq^{-2}X^{-2})(1-\sq^{-2}X^2)(1-\sq^{-2}X^{-1}Y)(1-\sq^{-2}XY^{-1})}{(1-X^{-1}Y^{-1})(1-XY)(1-X^{-2})(1-X^2)(1-X^{-1}Y)(1-XY^{-1})}.
\end{align*}
Thus
$$
I_{s_{n-1}'}(s_{n-1}'z)\circ I_{s_{n-1}'}(z)(\xi_z\otimes 1)=\kappa(z_{n-1},-\epsilon\sq^{-1})(\xi_z\otimes 1)
$$ 
In particular, both $\kappa(z_{n-1},-\sq^{-1})\neq 0$ and $\kappa(z_{n-1},\sq^{-1})\neq 0$ (as element of the field $\mathbb{Q}(\sq,z_1,\ldots,z_{n-1})$), and it follows that the intertwiners $I_g(z)$ are invertible (over rational functions), and hence $\pi_z\cong \pi_{gz}$ for all $g\in G_n$. 
\end{proof}

\begin{cor}[$G_n$-invariance]\label{cor:inv}
We have $\chi(h)\in \sR[z]^{G_n}$ for all $h\in\Hext$. 
\end{cor}

\begin{proof}
This follows from Theorem~\ref{thm:invariance} by taking traces. 
\end{proof}

\subsection{The leading matrix $\fc(\sm_{\la})$}\label{sec:leading}

The \textit{$n$-dominance partial order} on $P_n^+$ is given by 
$$\mu\leq_n \la\quad\text{if and only if}\quad  \lambda-\mu\in\sum_{i=1}^{n-2}\ZZ_{\geq 0}(e_i-e_{i+1})+\ZZ_{\geq 0}(2e_{n-1}).
$$

\begin{lemma}\label{lem:monotone}
If $\la,\mu\in P_n^+$ with $\mu\leq_n\la$ then $\ell(\sm_{\mu})\leq\ell(\sm_{\la})$ with equality if and only if $\mu=\la$. 
\end{lemma}

\begin{proof} By Corollary~\ref{cor:mlength} we have $\ell(\sm_{\la})=n(n-1)+H(\lambda)$, where $H$ is the linear function 
$
H(\lambda)=\sum_{i=1}^{n-1}(2n-2i+1)\lambda_i
$. Since $\mu\leq_n\lambda$ we have $\lambda-\mu=\sum_{i=1}^{n-2}a_i(e_i-e_{i+1})+a_{n-1}(2e_{n-1})$ with $a_1,\ldots,a_{n-1}\geq 0$. We have $H(e_i-e_{i+1})=2$ and $H(2e_{n-1})=6$, and hence $\ell(\sm_{\la})-\ell(\sm_{\mu})=2\sum_{i=1}^{n-2}a_i+6a_{n-1}\geq 0$. Moreover, equality holds if and only if $a_i=0$ for all $i=1,\ldots,n-1$, and so $\lambda=\mu$. 
\end{proof}

\begin{lemma}\label{lem:fullfold}
If $0\leq i\leq n-1$ then $\fu s_i\in \mathbb{X}_n$ and $\fu s_iA_0$ is on the negative side of the hyperplane separating $\fu A_0$ and $\fu s_iA_0$.  
\end{lemma}

\begin{proof}
Since $\fu\in W_n$ we have $\fu s_i\in W_n$, and hence $\fu s_i\in\mathbb{X}_n$ (as $W_n\subseteq \mathbb{X}_n$ by Lemma~\ref{lem:Dnfunddom}). A direct calculation, using $\fu=\sw_{0,n}\sw_n$, shows that $\fu\star \alpha_i\in \Phi_0^++\ZZ\delta$ for all $i\in\{0,1,\ldots,n-1\}$, and this is equivalent to the statement that $\fu s_iA_0$ is on the negative side of the hyperplane separating $\fu A_0$ from $\fu s_i A_0$. 
\end{proof}

For $\la\in P_n^+$ let $M_{\la}(z)\in\ZZ[z]^{G_n}$ be the \textit{monomial symmetric function}
$$
M_{\la}(z)=\frac{1}{|\mathsf{Stab}_{G_n}(\la)|}\sum_{g\in G_n}z^{g\la}.
$$

\begin{thm}\label{thm:paththm}
For each $\la\in P_n^+$ we have 
$$
\fc(\sm_{\la})=f_{\lambda}(z)E_{1,1}\quad\text{for some $f_{\lambda}(z)\in\ZZ[z]^{G_n}$}.
$$
Moreover, writing $f_{\la}(z)$ in the basis of monomial symmetric functions 
$$
f_{\lambda}(z)=cM_{\lambda}(z)+\sum_{\mu\in P_n^+\backslash \{\lambda\}} c_{\mu}M_{\mu}(z),
$$
we have $c\in\ZZ_{>0}$, and if $c_{\mu}\neq 0$ then $\lambda\not<_n\mu$.
\end{thm}

\begin{proof} 
We will use the path formula (Theorem~\ref{thm:pathformula}) to compute a leading term in $\pi(T_{\sm_{\la}};\sF)$ of degree $n(n-1)$ (which in particular shows that the $(1,1)$-entry of $\fc(\sm_{\la})$ is nonzero). From Theorem~\ref{thm:main2} we have $\sm_{\la}=\sw_n\cdot x_{\la}$ with $x_{\la}=\fu^{-1}\tau_{\la}\fu$. Let $p$ be the path of type $\vec{\sm}_{\la}=\vec{\sw}_n\cdot\vec{x}_{\la}$ starting at $\fu\in\sF$ with the $\vec{\sw}_n$ part being all folds, and the remaining $\vec{x}_{\la}$ consisting of crossings. We must show that this is an $n$-folded alcove path. The $\ell(\sw_n)$ folds are permitted by Lemma~\ref{lem:fullfold}. The remaining crossings start at $\fu$ and end at $\mathsf{end}(p)=\fu x_{\la}=\tau_{\la}\fu\in \mathbb{X}_n$ (because $\sF$ is a fundamental domain). Thus the start and end of this part of the path lie in $\mathbb{X}_n$, and since the path is of reduced type convexity implies that the entire path lies in $\mathbb{X}_n$, and so indeed $p$ is a valid $n$-folded alcove path. We have $\mathsf{wt}_{\sF}(p)=\la$ and $\theta_{\sF}(p)=\fu$, and so by Theorem~\ref{thm:pathformula} the path contributes $(\sq-\sq^{-1})^{n(n-1)}z^{\la}$ to the $(1,1)$-entry of the matrix $\pi(T_{\sm_{\la}};\sF)$. 

We must show that this leading contribution $\sq^{n(n-1)}z^{\la}$ from $p$ cannot be cancelled by contributions from other paths. Thus suppose that $p'$ is another path of type $\vec{\sm}_{\la}$ starting at $\fu$ and ending at $\tau_{\la}\fu$ (hence contributing a multiple of $z^{\la}$ to the matrix entry). Thus $\fu^{-1}\tau_{\la}\fu$ is obtained from $\sw_n(\fu^{-1}\tau_{\la}\fu)$ by deleting the generators corresponding to the folds and bounces in~$p'$. Moreover, this path $p'$ must attain at least degree $n(n-1)$ to potentially cancel the leading contribution $\sq^{n(n-1)}z^{\la}$ from $p$, and so $f(p')-b(p')\geq \ell(\sw_n)$. Putting this together, 
\begin{align*}
\ell(\sm_{\la})&=\ell(\sw_n)+\ell(\fu^{-1}\tau_{\la}\fu)\\
&\leq \ell(\sw_n)+\ell(\sw_n\fu^{-1}\tau_{\la}\fu)-f(p')-b(p')\\
&\leq \ell(\sw_n)+\ell(\sm_{\la})-\ell(\sw_n)-b(p')-b(p')\\
&=\ell(\sm_{\la})-2b(p').
\end{align*}
Hence $b(p')=0$ and $f(p')=n(n-1)$, which raises the leading contribution to $2\sq^{n(n-1)}z^{\la}$ (assuming $p'\neq p$). Hence the leading term does not cancel, and moreover it proves that $c>0$ (with $c$ as in the statement of the theorem).

Thus $[\fc(\sm_{\la})]_{1,1}\neq 0$, and so by Corollary~\ref{cor:partial} we have that $\fc(\sm_{\la})=f_{\la}(z)E_{1,1}$ for some $f_{\la}(z)\in\ZZ[z]$. Moreover, since $\chi(T_{\sm_{\la}})$ is $G_n$-invariant (by Corollary~\ref{cor:inv}) and $f_{\la}(z)$ is the coefficient of $\sq^{n(n-1)}$ in this trace (as $\pi(T_{\sm_{\la}};\sF)=\sq^{n(n-1)}f_{\la}(z)E_{1,1}+\text{(terms of strictly lower degree)}$) we have $f_{\la}(z)\in\ZZ[z]^{G_n}$. Thus $f_{\la}(z)$ can be expressed as a sum of monomial symmetric functions. 
%

Finally we show that if $M_{\mu}(z)$ with $\mu\in P_n^+$ occurs with nonzero coefficient in $f_{\la}(z)$ then $\la\not<_n\mu$. The argument is similar to the argument in the second paragraph of the proof -- now there is a path $p''$ of type $\vec{\sm}_{\la}$ starting at $\fu$ and ending at $\tau_{\mu}\fu$, with $f(p'')-b(p'')\geq \ell(\sw_n)$. Thus $\fu^{-1}\tau_{\mu}\fu$ is obtained from $\sw_n(\fu^{-1}\tau_{\la}\fu)$ by deleting the generators corresponding to the folds and bounces, and so as above we have $\ell(\sm_{\mu})=\ell(\sw_n)+\ell(\fu^{-1}\tau_{\mu}\fu)\leq \ell(\sm_{\la})-2b(p'')$. 
If $\ell(\sm_{\mu})<\ell(\sm_{\la})$ then by Lemma~\ref{lem:monotone} we have $\lambda\not<_n\mu$. If $\ell(\sm_{\mu})=\ell(\sm_{\lambda})$ with $\lambda\leq_n\mu$ then Lemma~\ref{lem:monotone} gives $\lambda=\mu$, completing the proof. 
\end{proof}

\begin{remark}
Ultimately we will see that $f_{\lambda}(z)$ is a single Schur function, and in particular $c=1$ and the terms $M_{\mu}(z)$ appearing in the sum have $\mu<_n\lambda$. 
\end{remark}

\subsection{The asymptotic Plancherel density}\label{sec:plancherel}

The \textit{Plancherel Theorem} of Opdam \cite{Opd:04} is a spectral decomposition of the canonical trace functional $\Tr:H(q)\to \RR$, where $H(q)$ denotes the algebra $\Hext$ with parameter $\sq$ specialised to a real number $q>1$. The theorem expresses $\Tr(h)$ as an integral over the tempered dual of $\Hext(q)$:
$$
\Tr(h)=\int_{\widehat{H(q)}_{\mathsf{temp}}}\chi_{\pi}(h)\,d\mu_q(\pi). 
$$

The spectrum is organised by induction data $\xi=(P,\delta,z)$ where $P$ is a finite root system parabolic, $\delta$ is a discrete series representation of the corresponding semisimple affine Hecke algebra, and $z\in T_u^P$ is a unitary unramified parameter. Twisting $\delta$ by $z$ and inducing from the affine Levi subalgebra produces $\pi_{P,\delta,z}$. 

The representation $\pi_{z}$ from Section~\ref{sec:pi} corresponds to the parabolic $W_{\{s_n\}}=\langle s_n\rangle$. The associated semisimple affine Hecke algebra is the $\{s_n\}$-affine Hecke algebra $H_{\{s_n\}}$ generated by $T_n$ and $X^{e_n}$ (see \cite[Section 5.3]{GLP:23}) of extended type $\tilde{\sA}_1$. The discrete series representation $\delta=\delta_{\epsilon}$ of $H_{\{s_n\}}$ is the $1$-dimensional representation with $T_n$ acting by $-\sq^{-1}$ and $X^{e_n}$ acting by $-\epsilon\sq^{-1}$. The twisting by $z\in T_u^P$ corresponds to the action of the additional lattice elements $X^{e_1},\ldots,X^{e_{n-1}}$ by $z_1,\ldots,z_{n-1}$ (these elements $X^{e_i}$ with $1\leq i\leq n-1$ are in the Levi subalgebra, but not in the semisimple $\{s_n\}$-affine Hecke algebra). Here we specialise $z_1,\ldots,z_{n-1}$ to nonzero complex numbers, and $\epsilon$ to $\pm 1$. The tempered representations have $|z_i|=1$.

Let 
$$
N(t)=\frac{(1-t^{-1})(1-t)}{(1-q^{-2}t^{-1})(1-q^{-2}t)}\quad \text{and}\quad N'(t)=\frac{(1+\epsilon q^{-1}t)(1+\epsilon q^{-1}t^{-1})}{(1+\epsilon q^{-3}t)(1+\epsilon q^{-3}t^{-1})}.
$$

\begin{thm}\label{thm:planchereldensity}
For each choice $\epsilon=\pm 1$ the density of $\pi_{z,\epsilon}$ in the Plancherel Theorem is
$$
d\nu_{q,\epsilon}(z)=\frac{q^{-2n(n-1)}(1-q^{-2})}{2|G_n|(1+q^{-2})}\bigg[\prod_{1\leq i<j\leq n-1}N(z_i/z_j)N(z_iz_j)\bigg]\bigg[\prod_{k=1}^{n-1}N(z_k^2)N'(z_k)\bigg]\,dz. 
$$
\end{thm}

\begin{proof}
It follows from \cite[(4.96) and Theorem~4.43(ii)]{Opd:04} that the Plancherel mass for each representation $\pi_{z,\epsilon}$ ($\epsilon=\pm 1$) is 
$$
d\nu_{q,\epsilon}(z)=\frac{\mathsf{fdim}(\delta_{\epsilon})}{|W(\mathsf{B}_{n-1})|}m^P(t(z,\epsilon))\,dz,
$$
where $dz$ is normalised Haar measure on $\mathbb{T}^{n-1}$, with $\TT$ the group of complex numbers of modulus~$1$, and $t(z,\epsilon)=(z_1,\ldots,z_{n-1},-\epsilon q^{-1})$. Using \cite[Corollary~3.32]{Opd:04} one computes the formal dimension $\mathsf{fdim}(\delta_{\epsilon})$ as 
$$
\frac{1}{\mathsf{fdim}(\delta_{\epsilon})}=2+4\sum_{k=1}^{\infty}q^{-2k}=\frac{2(1+q^{-2})}{(1-q^{-2})}. 
$$
Moreover by \cite[Proposition~3.27(ii)]{Opd:04} it follows that 
$$
m^P(t(z,\epsilon))=q^{-2(n^2-1)}\bigg[\prod_{1\leq i<j\leq n-1}N(z_iz_j^{-1})N(z_iz_j)\bigg]\bigg[\prod_{k=1}^{n-1}N(z_k^2)N(-\epsilon qz_k)N(-\epsilon q^{-1}z_k)\bigg],
$$
where 
$$
N(t)=\frac{(1-t^{-1})(1-t)}{(1-q^{-2}t^{-1})(1-q^{-2}t)}. 
$$
We have the cancellation
$$
N(-\epsilon q z_i)N(-\epsilon q^{-1}z_i)=q^2\frac{(1+\epsilon q^{-1}z_i)(1+\epsilon q^{-1}z_i^{-1})}{(1+\epsilon q^{-3}z_i)(1+\epsilon q^{-3}z_i^{-1})}=q^2N'(z_i), 
$$
and hence the result (since $|G_n|=|W(\sB_{n-1})|$). 
\end{proof}

Let 
\begin{align}\label{eq:D}
D(z)=\bigg[\prod_{1\leq i<j\leq n-1}(1-z_i^{-1}z_j)(1-z_i^{-1}z_j^{-1})\bigg]\bigg[\prod_{k=1}^{n-1}(1-z_k^{-2})\bigg].
\end{align}
The specialisation $\sq^{-1}=0$ corresponds to the limit $q\to\infty$. The following corollary computes the ``asymptotic density'' associated to the representation $\pi_z$. 

\begin{cor}\label{cor:asymptoticplancherel}
For each choice of specialisation $\epsilon=\pm 1$ we have 
$$
\lim_{q\to\infty}q^{2n(n-1)}d\nu_{q,\epsilon}(z)=\frac{|D(z)|^2}{2|G_n|}\,dz. 
$$
\end{cor}

\begin{proof}
This is immediate from Theorem~\ref{thm:planchereldensity}. 
\end{proof}

\subsection{Schur functions}

Define a symmetric bilinear form $(\cdot,\cdot)$ in $\sR[z]$ by 
\begin{align}\label{eq:innerprod}
( f_1,f_2)=\frac{1}{|G_n|}\big[f_1f_2^*D(z)D(z)^*\big]_{\mathsf{ct}}\quad\text{for $f_1,f_2\in\sR[z]$},
\end{align}
where $D(z)$ is as in~(\ref{eq:D}) (with the $z_i$ unspecialised). 

For $\lambda\in P_n^+$ define $\fs_{\la}(z)\in \ZZ[z]^{G_n}$ by 
$$
\fs_{\lambda}(z)=\sum_{g\in G_n}g\cdot\left(\frac{z^{\lambda}}{D(z)}\right).
$$

\begin{remark}\label{rem:characters}
It is easy to see that $\fs_{\la}(z)$ is indeed an element of $\ZZ[z]$. In fact, writing $\lambda'=\lambda_1e_1+\cdots+\lambda_{n-1}e_{n-1}$ we have
$$
\fs_{\la}(z)=\epsilon^{\lambda_n}\fs^{\mathsf{C}_{n-1}}_{\la'}(z_1,\ldots,z_{n-1})
$$
where $\fs_{\la'}^{\mathsf{C}_{n-1}}$ is the usual Schur function of type $\mathsf{C}_{n-1}$ corresponding to the weight $\la'$. We will refer to the elements $\fs_{\la}(z)$ as ``Schur functions''. These Schur functions are the characters of irreducible highest weight representations of the group $\mathsf{Sp}_{2n-2}(\CC)\times C_2$ (with $C_2$ the group of order~$2$). Specifically, for $\la\in P_n^+$, writing $\la'$ as above, the Schur function $\fs_{\la}(z)$ is the character of $V_{\la}=V_{\la'}\otimes \epsilon^{\la_n}$ where $V_{\la'}$ is the irreducible highest weight representation of $\mathsf{Sp}_{2n-2}(\CC)$, and $\epsilon$ is the nontrivial irreducible representation of $C_2$. 
\end{remark}

Recall the definition of the monomial symmetric functions $M_{\la}(z)$ from Section~\ref{sec:leading}

\begin{prop}
The Schur functions $\fs_{\la}(z)$ with $\la\in P_n^+$ are the unique elements of $\ZZ[z]^{G_{n}}$ satisfying 
\begin{compactenum}[$(1)$]
\item $\fs_{\la}(z)=M_{\la}(z)+\sum_{\mu<_n\la} a_{\la,\mu}M_{\mu}(z)$ with $a_{\la,\mu}\in\ZZ$, and 
\item $(\fs_{\la}(z),\fs_{\mu}(z))=\delta_{\la,\mu}$. 
\end{compactenum}
\end{prop}

\begin{proof} This follows from classical properties of Schur functions, see \cite[Proposition~3.4]{Ram}. 
\end{proof}

\subsection{The ring $\mathcal{J}_{\Ga\cap\Ga^{-1}}$}

In this section we determine the structure of $\mathcal{J}_{\Ga\cap\Ga^{-1}}$. Recall from Theorem~\ref{thm:paththm} that $\fc(\sm_{\la})=f_{\la}(z)E_{1,1}$ with $f_{\la}(z)\in \ZZ[z]^{G_n}$ satisfying a weak positivity and triangularity property. The following crucial theorem computes $f_{\la}(z)$ precisely. 

\begin{thm}\label{thm:identifyf}
For $\la\in P_n^+$ we have $f_{\lambda}(z)=\fs_{\lambda}(z)$. Thus $\fc(\sm_{\la})=\fs_{\la}(z)E_{11}$. 
\end{thm}

\begin{proof}
The Plancherel formula is a decomposition of the canonical trace against a positive measure into characters of unitary $*$-representations (see \cite{Opd:04}). Hence every integrand evaluated at $T_w^*T_w$ is nonnegative. Retaining only the two families $\pi_{z,\epsilon}$, using Proposition~\ref{prop:*property}, and the fact that $\Tr(T_w^*T_w)=1$, gives the inequality
%
$$
\sum_{\epsilon\in\{-1,1\}}\int_{\TT^{n-1}}\mathsf{tr}(\pi_{z,\epsilon}(T_w)^*\pi_{z,\epsilon}(T_w))\,d\nu_{q,\epsilon}(z)\leq 1\quad\text{for all $w\in\Wext$.}
$$
Take $w=\sm_{\la}$ in this inequality, and so 
$$
\sum_{\epsilon\in\{-1,1\}}\int_{\TT^{n-1}}\mathsf{tr}\big([\sq^{-n(n-1)}\pi_{z,\epsilon}(T_{\sm_{\la}})]^*[\sq^{-n(n-1)}\pi_{z,\epsilon}(T_{\sm_{\la}})]\big)[\sq^{2n(n-1)}d\nu_{q,\epsilon}(z)]\leq 1. 
$$
By Theorem~\ref{thm:paththm} we have $q^{-n(n-1)}\pi_{z,\epsilon}(T_{\sm_{\la}})\to f_{\la}(z)E_{1,1}$, and so by Corollary~\ref{cor:asymptoticplancherel} and Fatou's Lemma we have 
$$
\frac{1}{2|G_n|}\sum_{\epsilon\in\{-1,1\}}\int_{\TT^{n-1}}f_{\la}(z)^*f_{\la}(z)|D(z)|^2\,dz\leq 1. 
$$
As discussed in \cite[Remarks~5.5 and 5.12]{CGLP:24}, the process of averaging over $\epsilon\in\{-1,1\}$ and integrating over $\TT^{n-1}$ agrees with the combinatorial constant term inner product defined in~(\ref{eq:innerprod}), and hence the above inequality gives $(f_{\la},f_{\la})\leq 1$ (in the specialised Hecke algebra). 

On the other hand, the Schur functions $\fs_{\la}(z)$ form an orthonormal $\ZZ$-basis of $\ZZ[z]^{G_n}$ and so $f_{\la}(z)=\sum_{\mu\in P_n^+}c_{\mu}\fs_{\mu}(z)$. Then $(f_{\la},f_{\la})=\sum_{\mu\in P_n^+}c_{\mu}^2$ (a positive integer), and the above inequality forces $(f_{\la},f_{\la})=1$. Thus $c_{\mu}=\pm 1$ for a unique $\mu\in P_n^+$, and $c_{\mu'}=0$ for all $\mu'\neq \mu$. So $f_{\la}(z)=\pm \fs_{\mu}(z)$ for some $\mu\in P_n^+$. The weak positivity and triangularity in Theorem~\ref{thm:paththm} then forces $\mu=\la$ and $f_{\la}(z)=\fs_{\la}(z)$.  
\end{proof}

\begin{cor}
We have $\cJ_{\Ga\cap\Ga^{-1}}\cong \ZZ[z]^{G_n}$. 
\end{cor}

\begin{proof}
This follows from Theorem~\ref{thm:asymptoticrepresentation}, Theorem~\ref{thm:identifyf}, and the fact that the Schur functions form a $\ZZ$-basis of~$\ZZ[z]^{G_n}$.
\end{proof}

\subsection{The asymptotic algebra $\mathcal{J}_{\Delta_1}$}\label{sec:asymptoticfinal}

We determined the leading matrices $\fc(w)$ for elements $w\in\Ga\cap\Ga^{-1}$ in Theorem~\ref{thm:identifyf}. We now use our factorisation formula (Theorem~\ref{thm:main4}) to determine all other leading matrices. 


\begin{proof}[Proof of Theorem~\ref{thm:main5}] We must show that 
$$
\fc(u^{-1}\sm_{\la}v)=\fs_{\lambda}(z)E_{u,v}.
$$
We will first prove that $\fc(u^{-1}\sw_n)=E_{u,1}$. Write $a_u=\st_{u^{-1}\sw_n}$. Since $a_u=a_u\st_{\sw_n}$, applying $\pi^{\infty}$ gives $\fc(u^{-1}\sw_n)=\fc(u^{-1}\sw_n)E_{11}$, and hence only the first column of $\fc(u^{-1}\sw_n)$ is nonzero. Denote this first column by $\mathbf{v}$. By~(\ref{eq:second}) we have $a_u^*a_u=\st_{\sw_n}$ and hence $\fc(u^{-1}\sw_n)^*\fc(u^{-1}\sw_n)=E_{11}$, giving $\mathbf{v}^*\mathbf{v}=1$. Writing $(\mathbf{v})_i=\sum_{\la}c_{i,\la}z^{\la}$ and taking the constant term of $\mathbf{v}^*\mathbf{v}$ it follows that $\sum_{i,\la}c_{i,\la}^2=1$. Hence only one entry of $\mathbf{v}$ is nonzero, and this nonzero entry is of the form $\pm z^{\la}$ for some $\lambda\in P_n$. We locate the position of this entry, and compute the entry, using paths. 

Let $p$ be the path of type $u^{-1}\sw_n$ starting at $\fu u$ such that the $u^{-1}$-steps are crossings, and the $\sw_n$-steps are all folds (as in the proof of Theorem~\ref{thm:paththm} this is an $n$-folded alcove path). We have $\mathsf{end}(p)=\fu$, and hence $\mathsf{wt}_{\sF}(p)=0$ and $\theta_{\sF}(p)=\fu$. Thus by Theorem~\ref{thm:pathformula} the path $p$ contributes $(\sq-\sq^{-1})^{n(n-1)}$ to the $(u,1)$-entry of $\pi(T_{u^{-1}\sw_{n}};\sF)$. There is no cancellation of this coefficient, because any path $p'$ of type $u^{-1}\sw_n$ with these same start and end points has at least $\ell(u)$ crossings, and if $p'$ has $k$ crossings then the degree of the contribution from $p'$ is $n(n-1)+\ell(u)-k-2b(p')\leq n(n-1)$. Equality requires $k=\ell(u)$ and $b(p')=0$, and every such path contributes positive leading coefficient. Thus we have $\fc(u^{-1}\sw_n)=E_{u,1}$ as claimed. 

By Theorem~\ref{thm:main4} we have $\st_{u^{-1}\sw_n}\st_{\sm_{\la}}\st_{\sw_nv}=\st_{u^{-1}\sm_{\la}v}$ and hence by Theorem~\ref{thm:identifyf} we conclude that
$$
\fc(u^{-1}\sm_{\la}v)=\fc(u^{-1}\sw_n)\fc(\sm_{\la})\fc(\sw_nv)=\fs_{\la}(z)E_{u,1}E_{1,1}E_{1,v}=\fs_{\la}(z)E_{u,v},
$$
completing the proof. 
\end{proof}

\begin{cor}\label{cor:asymp}
The map 
$$\pi^{\infty}:\cJ_{\Delta_1}\to \mathsf{Mat}_{2^{n-1}n!}(\ZZ[z]^{G_n})\quad\text{with}\quad \st_{u^{-1}\sm_{\la}v}\mapsto \fc(u^{-1}\sm_{\la}v)=\fs_{\la}(z)E_{u,v}$$ is an isomorphism. 
\end{cor}

\begin{proof}
By Theorem~\ref{thm:main1} each element of $\cJ_{\Delta_1}$ can be written as $A=\sum a_{u,v,\lambda}\st_{u^{-1}\sm_{\la}\sv}$ with the sum over $u,v\in\sB$ and $\lambda\in P_n^+$. If $\pi^{\infty}(A)=0$ then Theorem~\ref{thm:main5} implies that for each $u,v\in\sB$ we have $\sum_{\la\in P_n^+}a_{u,v,\lambda}\fs_{\lambda}(z)=0$, and linear independence of the Schur functions forces $a_{u,v,\la}=0$ for each $\la\in P_n^+$. Thus $\pi^{\infty}$ is injective, and surjectivity follows from the fact that the Schur functions span $\ZZ[z]^{G_n}$. 
\end{proof}

Theorem~\ref{thm:main6} now follows immediately from Corollary~\ref{cor:asymp} (see Remark~\ref{rem:characters}). Moreover, we can now prove Corollary~\ref{cor:main7}. 

\begin{proof}[Proof of Corollary~\ref{cor:main7}]
Under Lusztig's bijection, the cell $\Delta_1$ corresponds to the minimal nontrivial unipotent class of $\mathsf{Sp}_{2n}(\CC)$, given by partition $(2,1^{2n-2})$, and in the notation of \cite[Conjecture~10.5]{Lus:IV} one computes $F=\mathsf{Sp}_{2n-2}(\CC)\times C_2$. We take $Y=\sB$ with the trivial $F$-action. An irreducible $F$-equivariant vector bundle on $\sB\times\sB$ is then supported at one point $(u,v)$, with fibre an irreducible representation $V_{\la}$. Consequently $K_F(\sB\times\sB)\cong \mathsf{Mat}_{|\sB|}(\mathsf{Rep}(F))$. Specifically, $\st_{u^{-1}\sm_{\la}v}$ corresponds to the bundle supported at $(u,v)$ with fibre $V_{\la}$. Moreover, since $(u^{-1}\sm_{\la}v)^{-1}=v^{-1}\sm_{\la}u$ and $V_{\la}^*\cong V_{\la}$ (as both symplectic and $C_2$-characters are self-dual) the inversion has precisely the required effect on bundles (swapping the two coordinates, and dualising the fibre).
\end{proof}

\subsection{Further applications}\label{sec:applications2}

We conclude with further applications. Firstly, the isomorphism in Theorem~\ref{thm:main6} gives the explicit multiplication rules in~$\cJ_{\Delta_1}$ in terms of multiplicities in tensor products of irreducible highest weight representations of the group $\mathsf{Sp}_{2n-2}(\CC)\times C_2$. 

\begin{cor} For $u,v,x,y\in\sB$ and $\la,\mu\in P_n^+$ we have
$$
\st_{u^{-1}\sm_{\la}v}\st_{x^{-1}\sm_{\mu}y}=\delta_{v,x}\sum_{\nu\in P_n^+}N_{\la,\mu}^{\nu}\st_{u^{-1}\sm_{\nu}y}\quad\text{where}\quad V_{\la}\otimes V_{\mu}\cong \bigoplus_{\nu\in P_n^+}V_{\nu}^{\oplus N_{\la,\mu}^{\nu}}.
$$
\end{cor}

\begin{proof} This is an immediate consequence of Theorem~\ref{thm:main6}. 
\end{proof}

Secondly, we can easily determine the centre of~$\cJ_{\Delta_1}$. 

\begin{cor}
For each $\la\in P_n^+$ define $\mathsf{z}_{\la}=\sum_{u\in\sB}\st_{u^{-1}\sm_{\la}u}$. Then
$$
Z(\cJ_{\Delta_1})=\bigoplus_{\la\in P_n^+}\ZZ \mathsf{z}_{\la}\cong\mathsf{Rep}(\mathsf{Sp}_{2n-2}(\CC)\times C_2). 
$$
\end{cor}

\begin{proof}
The centre of $\mathsf{Mat}_{|\sB|}(\ZZ[z]^{G_n})$ consists precisely of the scalar matrices, and under the isomorphism $\mathsf{z}_{\la}$ corresponds to $\fs_{\la}(z)I$.  
\end{proof}

Thirdly we have the following characterisation of the cell $\Delta_1$ in terms of the representation~$\pi$ (informally, $\pi$ \textit{recognises} $\Delta_1$). 

\begin{cor} Let $F$ be a fundamental domain for the action of $\sT_n$ on $\mathbf{X}_n$. 
The matrix $\pi(T_w;F)$ attains degree $n(n-1)$ if and only if $w\in\Delta_1$. 
\end{cor}

\begin{proof}
By \cite[Proposition~6.5]{GLP:23} the value of the bound, and the set of elements recognised, does not depend on the particular fundamental domain. Hence, taking $F=\sF$, the result follows from Theorem~\ref{thm:bounded} and Corollary~\ref{cor:asymp}. 
\end{proof}

Similarly, we see that the character $\chi$ of $\pi$ recognises the involutions in $\Delta_1$. Note that this is now a basis independent statement. 

\begin{cor}
The character $\chi(T_w)$ attains degree $n(n-1)$ if and only if $w\in\Delta_1$ with $w^2=1$. 
\end{cor}

\begin{proof} 
If $w\notin\Delta_1$ then $\deg\chi(T_w)<n(n-1)$ by Theorem~\ref{thm:bounded}. If $w\in\Delta_1$ then $w=u^{-1}\sm_{\la}v$ for some $u,v\in \sB$ and $\la\in P_n^+$, and by the definition of leading matrices, and Theorem~\ref{thm:main5}, we have
\begin{align*}
\sq^{-n(n-1)}\chi(T_w)\in \mathsf{tr}(\fc(u^{-1}\sm_{\la}v))+\sq^{-1}\ZZ[z][\sq^{-1}]=\fs_{\la}(z)\delta_{u,v}+\sq^{-1}\ZZ[z][\sq^{-1}],
\end{align*}
and so $\deg\chi(T_w)=n(n-1)$ if and only if $w=u^{-1}\sm_{\la}u$, if and only if $w\in\Delta_1$ with $w^2=1$ (by Corollary~\ref{cor:involutions}). 
\end{proof}

Next, we have the following ``Asymptotic Plancherel Theorem'' for the cell~$\Delta_1$ (see also \cite[Section~5.3]{CGLP:24} for further discussion in type $\tilde{\sA}_n$). Note that $d\nu_{\infty}$ is the limiting density from Corollary~\ref{cor:asymptoticplancherel}. We write $\pi^{\infty}=\pi^{\infty}_{z,\epsilon}$ to emphasise the dependence on specialised $z$ and $\epsilon$. 

\begin{cor}[Asymptotic Plancherel Theorem for $\Delta_1$]
For $A=\sum_{w\in\Delta_1}a_w\st_w$ we have 
$$
\sum_{\epsilon\in\{-1,1\}}\int_{\TT^{n-1}}\mathsf{tr}\big(\pi^{\infty}_{z,\epsilon}(A)\pi^{\infty}_{z,\epsilon}(A)^*\big)d\nu_{\infty}=\sum_wa_w^2\quad\text{where}\quad d\nu_{\infty}=\frac{|D(z)|^2}{2|G_n|}\,dz.
$$
\end{cor}

\begin{proof}
This follows from Theorem~\ref{thm:main5} and the fact that the Schur functions $\fs_{\la}(z)$ are orthonormal with respect to the above measure. 
\end{proof}

Finally, we deduce the structure of the asymptotic algebra for the cell $\Delta_1'=\Delta_1\cap\Waff$ of the non-extended affine Weyl group. While the symplectic group $\mathsf{Sp}_{2n-2}(\CC)$ arises in the corollary below, we note (from the proof) that it does so in a somewhat twisted way. 

\begin{cor}
We have $\cJ_{\Delta_1'}\cong \mathsf{Mat}_N(\mathsf{Rep}(\mathsf{Sp}_{2n-2}(\CC)))$, where $N=2^{n-1}n!$. 
\end{cor}

\begin{proof}
An element $\la\in P_n$ is in $Q\cap P_n$ if and only if $\la_n=\la_1+\cdots+\la_{n-1}\pmod{2}$. Thus $\la_n\in\{0,1\}$ is determined by $\la'=(\la_1,\ldots,\la_{n-1})$. It follows that $\cJ_{\Delta_1'}$ is isomorphic to the full matrix algebra of $N\times N$ matrices over the representation ring of the group $(\mathsf{Sp}_{2n-2}(\CC)\times C_2)/\langle (-I,-1)\rangle\cong \mathsf{Sp}_{2n-2}(\CC)$ (via the isomorphism $(g,\eta)\mapsto \eta g$). The element $-I$ acts on $V_{\la'}$ by $(-1)^{\la_1+\cdots+\la_{n-1}}$.  
\end{proof}

\subsection{Comments on the $\tilde{\sA}_n$ case}\label{sec:An}

In this final section, let $\Wext=P\rtimes \Wfin$ be the extended affine Weyl group of type $\tilde{\sA}_n$. In this case a complete description of the decomposition into cells (one-sided and two-sided) is known by the work of Lusztig and Shi~\cite{Lus:IV,ShiAn}. The two-sided cells are in bijection with partitions of $n+1$, and for $\la\vdash n+1$ we write $\Delta_{\la}$ for the two-sided cell containing $\sw_{\la'}$, where $\la'$ is the dual partition, and $\sw_{\la'}$ is the longest element of the associated Young subgroup of the symmetric group~$\Wfin$. In \cite{Xibook} Xi determined the structure of the asymptotic algebra $\cJ=\bigoplus_{\la\vdash n+1}\cJ_{\Delta_{\la}}$, showing that $\cJ_{\Delta_{\la}}\cong \mathsf{Mat}_{N_{\la}}(\mathsf{Rep}(F_{\la}))$, where $N_{\la}=(n+1)!/(\la_1!\cdots \la_{l(\la)}!)$ and $F_{\la}$ is the (possibly disconnected) maximal reductive subgroup of the centraliser in $\mathsf{SL}_{n+1}(\CC)$ of the unipotent element $u_{\la}$ with Jordan blocks given by the partition~$\la$. 

In \cite{CGLP:24} we gave another approach to Xi's theorem, using techniques similar to those used in the present paper. In particular, in \cite[Corollary~7.5]{CGLP:24} we show that for each $\la\vdash n+1$ we have $\cJ_{\Delta_{\la}}\cong \mathsf{Mat}_{N_{\la}}(\ZZ[\zeta_{\la}]^{G_{\la}})$, here $\ZZ[\zeta_{\la}]$ is a certain quotient of a polynomial ring, and $G_{\la}$ is a symmetry group associated to the \textit{fundamental $\la$-alcove}~$\cA_{\la}$. There is an associated pseudo-translation group $\sT_{\la}$, and a set $\mathbb{X}_{\la}=\{w\in\Wext\mid wA_0\subseteq \cA_{\la}\}$. In \cite[Theorem~7.8]{CGLP:24} we show, conditional on an assumption relating to distinguished involutions, that a (signed) choice of fundamental domain for the action of $\sT_{\la}$ on $\mathbb{X}_{\la}$ may be chosen so that each leading matrix $\fc_{\la}(w)$, with $w\in \Delta_{\la}$, is of the form $\fc_{\la}(w)=\fs_{\gamma}(\zeta_{\la})E_{u,v}$, where $\fs_{\gamma}(\zeta_{\la})\in \ZZ[\zeta_{\la}]^{G_{\la}}$ is a Schur function. However this theorem was conditional on \cite[Conjecture~7.7]{CGLP:24}, where it is conjectured that (up to reordering the basis) each distinguished involution $d_i\in\Delta_{\la}$ has leading matrix of the form $\fc_{\la}(d_i)=E_{i,i}$. We take this opportunity to give a short proof of \cite[Conjecture~7.7]{CGLP:24}, thus making the description of the asymptotic algebra in \cite[Theorem~7.8]{CGLP:24} unconditional. Indeed, Conjecture~7.7 follows quickly from the following elementary general observation. 

\begin{lemma}\label{lem:elementary}
Let $G$ be a group and let $*$ be the involution on $\ZZ[G]$ given by $g^*=g^{-1}$. Every self adjoint idempotent matrix in $\mathsf{Mat}_N(\ZZ[G])$ is a diagonal matrix with diagonal entries in $\{0,1\}$. 
\end{lemma}

\begin{proof}
Suppose that $P=P^*=P^2$, and write $[P]_{u,v}=\sum_{g\in G}a_{u,v,g}g$ with $a_{u,v,g}\in\ZZ$. Fix a row $u$. Taking the coefficient of $1$ in the $(u,u)$-entry of $P=PP^*$ gives $a_{u,u,1}=\sum_{v,\,g}a_{u,v,g}^2$. The right hand side includes $a_{u,u,1}^2$, and so $a_{u,u,1}\geq a_{u,u,1}^2$, forcing $a_{u,u,1}=0$ or $a_{u,u,1}=1$. In the former case the entire $u$-row of $P$ is zero. In the latter case $a_{u,u,1}=1$ is the unique nonzero entry of the $u$-row. Hence the result. 
\end{proof}

\begin{cor}
Conjecture~7.7 of \cite{CGLP:24} holds. 
\end{cor}

\begin{proof}
In the notation of \cite{CGLP:24}, let $P_i=\fc_{\la}(d_i)$ and $G=P/Q_{\la}$. Using the asymptotic representation (see \cite[Proposition~4.12]{CGLP:24}) we have 
$$P_i^2=\fc_{\la}(d_i)^2=\pi^{\infty}_{\la}(\st_{d_i})^2=\pi^{\infty}_{\la}(\st_{d_i}^2)=\pi^{\infty}_{\la}(\st_{d_i})=\fc_{\la}(d_i)=P_i.
$$ Since $\fc_{\la}(w^{-1})=\fc_{\la}(w)^*$ for all $w\in\Wext$ (see \cite[(4.5)]{CGLP:24}) and $d_i^{-1}=d_i$ we have $P_i^*=P_i$. We have $P_i\neq 0$ by \cite[Theorem~5.10]{CGLP:24}. Thus by Lemma~\ref{lem:elementary} $P_i$ is diagonal with entries $0$ and $1$. Since we also have $\st_{d_i}\st_{d_j}=0$ whenever $i\neq j$ the asymptotic representation gives $P_iP_j=0$ if $i\neq j$. Since there are $N_{\la}$ distinct right cells in $\Delta_{\la}$, there are $N_{\la}$ distinct distinguished involutions in $\Delta_{\la}$, and it follows that $\fc_{\la}(d_i)=E_{\sigma(i),\sigma(i)}$ for some permutation $\sigma:\{1,\ldots,N_{\la}\}\to\{1,\ldots,N_{\la}\}$, hence the result. 
\end{proof}

\section*{AI declaration}

During the preparation of this work, the author used \texttt{ChatGPT-6 Astra} to assist in calculations, proofread the manuscript, suggest structural improvements and applications, and find examples and counterexamples. More specifically, AI was used in the inversion set calculations in Sections~\ref{sec:Jalcove} and~\ref{sec:main}, determining the bounding hyperplanes for the convex regions in the proof of Theorem~\ref{thm:determineB}, finding the examples of paths and cancellation in Example~\ref{ex:pathscancel}, and in suggesting some of the applications presented in Sections~\ref{sec:applications1} and~\ref{sec:applications2}. In all cases, the resulting mathematics was easily human-verified. 

After the manuscript was prepared by the author, AI was used for proofreading and checking for mathematical inconsistencies, omissions, and errors. The minor suggestions for improvements, additions, and corrections were checked by the author. The core innovations, ideas, and proofs in the paper are due to the author, and he takes full responsibility for the mathematical correctness and results.

\bibliographystyle{plain}

\bigskip
  \footnotesize
 \noindent James Parkinson, \textsc{School of Mathematics and Statistics, University of Sydney, NSW 2006, Australia}\par\nopagebreak
 \noindent \textit{E-mail address:} \texttt{jamesp@maths.usyd.edu.au}

\end{document}